\documentclass[10pt]{article}

\usepackage[english]{babel}

\usepackage{geometry}
\usepackage{amssymb, amsmath, latexsym, mathrsfs, verbatim, calc, amsthm}
\usepackage{color}
\usepackage{indentfirst}
\usepackage{subfigure}
\usepackage[hyphens]{url}

\usepackage{braket}
\usepackage{graphicx}

\usepackage[colorlinks=true, allcolors=blue]{hyperref}

\usepackage{setspace} 

\usepackage[title]{appendix}
\usepackage{algorithm}
\usepackage{algorithmic}

\usepackage{stackengine}

\DeclareMathOperator{\tr}{tr}

\numberwithin{equation}{subsection}	

\newcommand{\vertiii}[1]{{\left\vert\kern-0.25ex\left\vert\kern-0.25ex\left\vert #1 
		\right\vert\kern-0.25ex\right\vert\kern-0.25ex\right\vert}}

\theoremstyle{plain}
\theoremstyle{definition}

\newtheorem{problem}{Problem}[section]
\newtheorem{lemma}{Lemma}[section]
\newtheorem{theorem}{Theorem}[section]

\theoremstyle{definition}
\newtheorem{definition}{Definition}[section]
\newtheorem{remark}{Remark}[section]

\newcommand{\customfootnotetext}[2]{{
		\renewcommand{\thefootnote}{#1}
		\footnotetext[0]{#2}}}

\allowdisplaybreaks

\title{Linear Quadratic Extended Mean Field Games and Control Problems}
\author{
	\renewcommand{\thefootnote}{\arabic{footnote}}
	Alain Bensoussan\textsuperscript{$\dagger$,}\footnotemark[1] \quad
	Bohan Li\textsuperscript{$\mathsection$,}\footnotemark[2] \textsuperscript{$,3$} \quad
	Sheung Chi Phillip Yam\textsuperscript{$\mathparagraph$,}\footnotemark[3]
}

\date{\today}

\customfootnotetext{$\dagger$}{E-mail: axb046100@utdallas.edu }
\customfootnotetext{$\mathsection$}{E-mail: bohanli995@gmail.com }
\customfootnotetext{$\mathparagraph$}{E-mail: scpyam@sta.cuhk.edu.hk }
\customfootnotetext{$ $}{Author list in alphabetical order/all the authors have contributed equally to the paper}

\begin{document}
	\maketitle
	
	\begin{center}
	\textsuperscript{1} \textit{Naveen Jindal School of Management, The University of Texas at Dallas, Dallas, Texas, USA}\\
	\textsuperscript{2} \textit{Center for Financial Engineering, Soochow University, Suzhou, Jiangsu 215006, PRC}\\
	\textsuperscript{3} \textit{Department of Statistics, Chinese University of Hong Kong, Shatin, N.T., Hong Kong SAR}  
	\end{center}\vspace{0pt}
	
	
	\begin{abstract}
	We provide a thorough study of a general class of linear-quadratic extended mean field games and control problems in any dimensions where the mean field terms are allowed to be unbounded and there are also presence of cross terms in the objective functionals. Our investigation focuses on the unique existence of equilibrium strategies for the extended mean field problems by employing the stochastic maximum principle approach and the appropriate fixed point argument. Notably, the linearity of the state and adjoint processes allows us to derive a forward-backward ordinary differential equation governing the optimal mean field term. We provide two distinct proofs, accompanied by two sufficient conditions, that establish the unique existence of the equilibrium strategy over a global time horizon. Both conditions emphasize the importance of sufficiently small coefficients of sensitivity for the cross term, of state and control, and mean field term. To determine the required magnitude of these coefficients, we utilize the singular values of appropriate matrices and Weyl's inequalities. The present proposed theory is consistent with the classical one, namely, our theoretical framework encompasses classical linear-quadratic stochastic control problems as particular cases. Additionally, we establish sufficient conditions for the unique existence of solutions to a particular class of non-symmetric Riccati equations, and we illustrate a counterexample to the existence of equilibrium strategies. Furthermore, we also apply the stochastic maximum principle approach to examine linear-quadratic extended mean field type stochastic control problems. Comparing to the game setting, the unique existence of the corresponding adjoint equations of the extended mean field type control problems can always be warranted without the requirement of any sufficiently small coefficients; indeed, it can be established under a natural convexity condition. Finally, we conduct a comparative analysis between our method and the alternative master equation approach, specifically addressing the efficacy of the present proposed approach in solving common practical problems, for which the explicit forms of the equilibrium strategies can be obtained directly, even over any global time horizon.
	
	\end{abstract}
	\vfill
	
	
	
	\newpage
	\tableofcontents
	
\section{Introduction}

Over the past two decades, the theory of mean field games (MFG for short) has gained significant attention in the academic community (see \cite{lasry2007mean} \cite{carmona2018probabilistic} \cite{bensoussan2013mean} for more details), and it now becomes a powerful tool for a variety of disciplines in science and engineering. For example, in sociological research, the mean field theory has been applied to the study of pedestrian evacuation problem (see \cite{vermuyten2016review},\cite{aurell2018mean},\cite{salhab2017dynamic,salhab2018dynamic} for more details). In epidemiology, this method can be used to study the spread of infectious diseases in crowds (see \cite{aurell2022optimal} for more details). For other works, we refer to \cite{kizilkale2012mean}, \cite{nourian2013mean}, \cite{nourian2013epsilon}, \cite{caines2016epsilon}, \cite{lin2018open}, \cite{barreiro2019linear}, \cite{salhab2019collective}, \cite{firoozi2020epsilon}, \cite{stella2021mean}, \cite{fu2021mean}, \cite{bayraktar2022finite} and \cite{feng2021unified}.



A MFG is the approximation of the $N$ number of players/agents Nash game as $N$ tend to infinity, in which any negligible agents interact via a mean field term which represents the population distribution of the state process. To capture the phenomenon of individual agents being directly influenced by the policies of other agents, \cite{gomes2014existence} and \cite{gomes2016extended} introduced an enhanced and more encompassing variant of MFG known as {\em extended} MFG (EMFG for short). Within this framework, the mean field analogue of agent policies is incorporated into the model. In the context of EMFG, several notable works have contributed to our understanding of this field. For instance, {\iffalse\cite{carmona2015probabilistic} studies the EMFG using a probabilistic method with the assumption that no mean field term is in the drift and the running cost is separable in the mean field of controls. \fi}\cite{cardaliaguet2018mean} considers an optimal liquidation problem with the price impact generated by the purchases of all market participants; furthermore, they employ the method inherited in \cite{carmona2018probabilistic} to a class of EMFGs; it points out the necessity of certain uniform boundedness conditions in order to have their theory to be valid, which, however, are not met in the common linear-quadratic modelings. Another work \cite{bonnans2021schauder} focuses on investigating the Fokker–Planck-Hamilton-Jacobi-Bellman (FP-HJB) system, which corresponds to a setting with the absence of a mean field term in the drift while the cost function is separable in the state and control. \cite{carmona2021probabilistic} explores the EMFG in the context of a finite state Markov chain model, and they established the corresponding unique existence of the Nash equilibrium in the class of all closed-loop strategies. \cite{motte2022mean} delves into the examination of EMFG driven by discrete-time Markov decision processes (MDPs), with the presence of a common noise, in which they formulated the original problem as another one with a suitable lifted MDP in a Wasserstein space, and then they solved the corresponding Bellman fixed point equation. \cite{li2023linear} focuses on the linear quadratic extended mean field game and its master equation, where all the mean field terms are in some bounded functions and the running cost is separable (in state and control). The authors use a Stochastic Maximum Principle (SMP) to study the problem in one dimension. In contrast, our article allows the unboundedness of the mean field terms in both state and objective functionals, and the presence of the cross terms in the objective functional. More relevant works on this subject can be found in \cite{alasseur2020extended}, \cite{kobeissi2022classical} and \cite{pham2022mean}.



In a cooperative game involving $N$ agents, mean field type control (MFTC for short) can be regarded as an approximation, in which a central planner selects the policies for all agents. The distinctions between MFG and MFTC are elucidated in the work of \cite{bensoussan2013mean}; for further references on MFTC, we also recommend \cite{djehiche2015stochastic}, \cite{bensoussan2017interpretation}, \cite{ismail2019robust}, \cite{moon2020linear}, \cite{bensoussan2020control}, \cite{pham2022portfolio} and \cite{bensoussan2022dynamic}. By incorporating the mean field term of control into MFTC, the resulting extension is known as {\em extended} MFTC (EMFTC for short) problem, as named by \cite{acciaio2019extended}. In EMFTC, the objective functional and state process depend on the joint distribution of both state and control. Furthermore, this new model has been examined in other articles, including \cite{graber2016linear}, \cite{basei2019weak}, \cite{pham2018bellman}, \cite{djete2022mckean} and \cite{bensoussan2022optimization}, and they considered both technical matter and potential applications. For instance, \cite{graber2016linear} investigates the relationship between linear quadratic EMFTC problems and EMFGs. \cite{pham2018bellman} employed the dynamic programming principle and formulated the problem as a deterministic control problem, in which the new state represents the distribution of $X_t$; they derived the Bellman equation for this converted problem and focused on the linear quadratic case, which considers the presence of neither the cross term between the state variable and the mean field term of control, nor the cross term between the control variable and the mean field term of state in the objective function; yet our present work considers all possible scenarios. Moreover, \cite{basei2019weak} applied the martingale optimality principle to a linear quadratic EMFTC problem which is the same as that of \cite{pham2018bellman}, and established the unique existence by studying the corresponding symmetric Riccati equation.  In Chapter 3 of \cite{sun2020stochastic}, they studied a linear quadratic EMFTC problem by using the symmetric Riccati equation approach, and established an equivalence relation between the solvability of these equations and the uniform convexity of the cost functionals; noted that there is no cross interaction terms of the state/control processes and their corresponding mean field terms in their work, while we particularly pay attention here in ours. \cite{alasseur2020extended} focued on the dynamic grid optimization problem for both the EMFGs and EMFTC problems; they related the mean field equilibrium of a EMFG to the optimal control of an associated EMFTC problem, and then provided the explicit solution for a relatively simpler linear quadratic case in which there is no mean-field term of states in either state processes or the objective functional, while the mean-field term of control is also not present in the state processes; to this end, due to their application setting, it is truly not necessary to consider these terms. Instead in our framework, a more sophisticated model involving all these terms is intended to construct. \cite{li2021mean} considered a two-person non-zero sum game composed of two linear quadratic EMFTC problem each with a same setting as Chapter 3 of \cite{sun2020stochastic}; they also obtained similar results to that of \cite{sun2020stochastic}. \cite{bensoussan2022optimization} employed the master equation approach to investigate the application of EMFTC in the dynamic grid problem of renewable energy storage and the environmental risk problem. \cite{djete2022mckean} investigates the EMFTC problem with common noise, particularly emphasizing the equivalence of strong, weak, and relaxed formulations, as well as the limiting convergence of a large population problem to the EMFTC problem.

In the context of Nash games involving multiple players, solving for equilibrium strategy becomes particularly challenging when the state processes and objective functionals of each player are influenced by the actions of other players, both in terms of their state processes and controls. The works of \cite{carmona2018probabilistic}, \cite{djete2020mean}, \cite{bonnans2021schauder} and \cite{carmona2021probabilistic} have demonstrated that these problems can be reformulated as EMFGs in the limiting case. In the present study, we adopt this approach and propose a comprehensive solution under the common linear quadratic setting, where each player's multidimensional state process is linear in state, control, mean field term of both states and controls, while the objective functional is quadratic in them. In contrast to the works of \cite{carmona2018probabilistic}, \cite{djete2020mean}, \cite{bonnans2021schauder} and \cite{carmona2021probabilistic}, we allow the objective functional to be non-separable in state, control, mean field terms of states and controls, in particular, the objective functional is allowed to include a cross-term that consisting all of them. Our problem differs from those investigated in \cite{cardaliaguet2018mean} and \cite{djete2020mean}, since their works impose uniform boundedness assumptions on state processes or objective functionals, which do certainly fail to hold in the common linear quadratic models. To tackle our problems of study, we utilize the fixed point technique and the stochastic maximum principle to derive a system of forward-backward stochastic differential equations (FBSDEs). We then prove the unique existence of a global solution to the FBSDEs under two different assumptions of small coefficients of sensitivity. Particularly, the first proof enhanced the method first proposed by \cite{bensoussan2016linear}, in which there was only a local solution for the vanilla linear quadratic MFG. This method reveals that if the coefficients of mean field terms in the state process and those of objective functional are sufficiently small, a unique global solution can be guaranteed; besides, our newly proposed theory is consistent with that of \cite{bensoussan2016linear} in the absence of mean field term of controls. 
    \par For the second proof, we obtain the unique existence of a global solution under an alternative assumption by employing the homotopy method first introduced in \cite{peng1999fully}; for this second assumption, it is different from the first so that we can allow the coefficients of mean field term of controls to be relatively large in either the objective functional or the driving dynamics; the same large sensitivity for the mean field term of state can be allowed in objective functional too. Instead, we here assume the coefficient of cross-term in the objective functional to be relatively small together with certain positive-definiteness conditions on some other coefficients of the objective functional. To fulfill this demand of positive-definiteness conditions, we resort to the celebrated Weyl's inequalities and provide a tractable condition in relation to the singular values of the coefficient matrices. As a side-product of our main theory, we also enrich the study on a wide class of non-symmetric Riccati equations first studied in \cite{bensoussan2016linear}. We also illustrate a counterexample that shows, for certain initial state, that leads to non-uniqueness which results in two possible solutions; furthermore, for certain choice of initial state, even worse, the equilibrium strategies may not exist. Moreover, we shall present the usual limiting relationship between the EMFG and the Nash game with a large number of players in the linear quadratic setting. The convergence rates of both the equilibrium strategy and the objective functionals of the finite player Nash game to their EMFG counterparts are precisely obtained in Section \ref{lqmf.section.convergence_rate}.
    \par In addition, we extend our analysis to the linear-quadratic EMFTC problems. By employing the stochastic maximum principle approach, the EMFTC is characterized by a system of FBSDEs. We observe that in contrast to the game setting, the unique existence of the corresponding system of FBSDEs in linear-quadratic EMFTC problems can always be guaranteed without imposing further requirement of sufficiently small coefficients. Instead, we establish their unique existence under a natural convexity condition. To indicate the effectiveness of our proposed method, we conduct a comparison study between our proposed approach and the alternative master equation approach. We specifically demonstrate the effectiveness of our method in solving practical problems, so that the explicit forms of the equilibrium strategies can be obtained directly, even over a global time setting. Though we are not going to consider the common noise in the present article to avoid any unnecessary technicalities, yet our arguments here can be readily extended to scenarios involving common noise in both EMFG and EMFTC problems. An example including common noise will still be illustrated in Appendix \ref{lqmf.section.appendixB} to show that the same approach is indeed easily extendable.




The organization of our paper is as follows. Section \ref{lqmf.section.emfg} formulates a Nash game involving $N$ players. We then consider an EMFG which corresponds to the limiting Nash game as the number of agents tends to infinity. By employing fixed point analysis and the stochastic maximum principle, we derive the mean field equilibrium in terms of the unique solution of a system of FBSDEs including the mean field terms of both the state and controls as the fixed point, which can be further represented by the unique solution of a system of forward backward ordinary differential equations (FBODEs). Section \ref{lqmf.section.unique_existence} contains our novel results, in which we provide two sufficient conditions that ensure the unique existence of the mean field equilibrium strategy. The first condition is for both short time and global time horizons, while the second condition allows unique existence result over a global time horizon, and its truth follows by some application of Weyl's inequality. In Section \ref{lqmf.section.convergence_rate}, we establish the relationship between the $\varepsilon$-Nash equilibrium for the original Nash game involving $N$ players and the mean field equilibrium obtained in our EMFG. Furthermore, in Section \ref{lqmf.section.emftc} we also consider EMFTC problems, by employing the stochastic maximum principle, we obtain the optimal strategy under some natural convexity condition. In addition, we conduct a comparative analysis of our proposed approach with the work \cite{alasseur2020extended} and master equation approach in Appendix \ref{lqmf.section.appendixB} which also demonstrate the effectiveness of our approach when common noise is present. {\iffalse providing insights into the strengths and limitations of each method.\fi} 

\section{Problem Formulation} \label{lqmf.section.emfg}

\subsection{Stochastic Differential Game for $N$ players}

Consider a completed filtered probability space $(\Omega, \mathcal{F}, \mathbb{P})$ with the filtration
\begin{align*}
\mathcal{F}_t:=\sigma\left(\left(x_0^1, \ldots, x_0^N\right),\left(W_s^1, \ldots, W_s^N\right), s \leq t\right),
\end{align*}
where $W^1, \ldots, W^N$ denote $N$ independent $n$-dimensional standard Brownian motions defined on $(\Omega, \mathcal{F}, \mathbb{P})$ and $x_0^1, \ldots, x_0^N$ form a independent sample of size $N$, from a common distribution which is independent of the Brownian motion and these $x_0^i$ stand for the initial point of the following dynamics \eqref{lqmf.N_player.state_process}. Let us consider a $N$ players Nash-game in which the dynamics of the $i$-th player is given by
\begin{align}
\label{lqmf.N_player.state_process}
d y^i_t=\left(A_t y^i_t+B_t v^i_t+\bar{A}_t \frac{1}{N-1}\sum_{j\neq i}y^j_t+\bar{B}_t \frac{1}{N-1}\sum_{j\neq i}v^j_t\right) d t+\sigma_t d W^i_t, \quad y^i(0)=x^i_0,
\end{align}
where $A, B, \bar{A}, \bar{B}$ are bounded deterministic matrix-valued functions in time-$t$ of suitable sizes, $\sigma$ is a square-integrable function on $\mathbb{R}^{n\times n}$, and $v^i$ is the admissible control variable in $L_{\mathcal{F}_{\cdot}}^2\left(0, T ; \mathbb{R}^m\right)$, namely the space of ${\mathcal{F}_{\cdot}}$-adapted stochastic processes such that
\begin{align*}
\mathbb{E} \left(\int_0^T \|v^i_t\|^2 dt\right) < \infty.
\end{align*}
Here, $A$ measures the effect of the player's current state on his own dynamics, $B$ measures the impact of the player's self-regulation controlled by himself, $\bar{A}$ represents the overall influence of the current states of the other players, and $\bar{B}$ reflects the {\em spillover effect} of the other players' controls. Let $v^{-i}:=(v^1,...,v^{i-1},v^{i+1},...,v^N) \in \mathbb{R}^{m\times (N-1)}$ be the vector of the controls of all players except the $i$-th player, the objective functional for each player $i$ is formulated as follows:
\begin{align}
\nonumber
\mathcal{J}^i(v^i,v^{-i}) :=& \mathbb{E}\left[\frac{1}{2} \int_0^T (y^i_t)^{\top} Q_t y^i_t+(v^i_t)^{\top} P_t v^i_t+\left(y^i_t-S_t \frac{1}{N-1}\sum_{j\neq i}y^j_t\right)^{\top} \bar{Q}_t\left(y^i_t-S_t \frac{1}{N-1}\sum_{j\neq i}y^j_t\right) \right.\\
\nonumber
&\qquad \quad+\left. 2\left(y^i_t-\bar{S}_t \frac{1}{N-1}\sum_{j\neq i}y^j_t\right)^{\top} N_t\left(v^i_t-\bar{R}_t \frac{1}{N-1}\sum_{j\neq i}v^j_t\right) \right.\\
\nonumber
&\qquad \quad+\left. \left(v^i_t-R_t \frac{1}{N-1}\sum_{j\neq i}v^j_t\right)^{\top} \bar{P}_t\left(v^i_t-R_t \frac{1}{N-1}\sum_{j\neq i}v^j_t\right) d t\right] \\
\label{lqmf.N_player.objective}
&+\mathbb{E}\left[\frac{1}{2} (x^i_T)^{\top} Q_T x^i_T+\frac{1}{2}\left(x^i_T-S_T \frac{1}{N-1}\sum_{j\neq i}y^j_T\right)^{\top} \bar{Q}_T\left(x^i_T-S_T \frac{1}{N-1}\sum_{j\neq i}y^j_T\right)\right],
\end{align}
where $Q, \bar{Q}$, $P, \bar{P}$, $S, \bar{S}$, $R, \bar{R}$ and $N$ are bounded, deterministic matrix-valued functions in time-$t$ of suitable sizes. Here, $Q$ measures the quadratic cost of staying at the current state $y^i_t$, $P$ measures the quadratic cost of taking an action $v^i_t$, $\bar{Q}$ represents the cost of the state deviating from a benchmark, scaled by $S$, calibrated as the average of other players' optimal states. Similarly, $\bar{P}$ measures the penalty for an action deviating from a benchmark, scaled by $R$, of the average of the other players' optimal controls, and $N$ along with $\bar{S}$ and $\bar{R}$ represent the overall cross effect/interaction between deviations of state and control from the average of the crowd. We suppose that $Q, \bar{Q}, P, \bar{P}$ are non-negative definite such that $Q+\bar{Q} \geq \delta I$, $P+\bar{P} \geq \delta I$, for some $\delta>0$. If we take $\bar{B}=\bar{P}=\bar{R}=0$, the present model reduces to the linear quadratic mean field game in the literature; if we further take $\bar{A}=\bar{Q}=\bar{S}=0$, the present model reduces to the classical linear quadratic control problem.


The principal objective of each player is to minimize his own objective functional by properly controlling his own dynamics. In this classical nonzero-sum stochastic differential game framework, we aim to establish a Nash equilibrium $\left(v^1, \ldots, v^N\right)$ (see, for example, \cite{bensoussan2000stochastic}):

\begin{definition}
	The group of admissible controls $\{\hat{v}^{i}\}_{i\in N} \subset L_{\mathcal{F}_{\cdot}}^{2}\left(0, T ; \mathbb{R}^m\right)$ is a Nash equilibrium if it satisfies the inequalities
	\begin{align*}
	\mathcal{J}^i(\hat{v}^i,\hat{v}^{-i}) \leq \mathcal{J}^i(v^i,\hat{v}^{-i}), \quad \text{for} \quad 1 \leq i \leq N,
	\end{align*}
	where $\hat{v}^{-i} = (\hat{v}^1,\cdots,\hat{v}^{i-1},\hat{v}^{i+1},\hat{v}^N)$.
\end{definition}

\begin{problem} \label{lqmf.mfg.problem.N_player}
	Find the Nash equilibrium $(\hat{v}^1,\cdots,\hat{v}^N)$ with $\hat{v}_t^i \in L_{\mathcal{F}_{\cdot}}^{2}\left(0, T ; \mathbb{R}^m\right)$ for all $1\leq i \leq N$, such that it can minimize \eqref{lqmf.N_player.objective} which is governed by the dynamics \eqref{lqmf.N_player.state_process}.	
\end{problem}

Even though we have omitted the consideration of common noise in our current paper, our arguments can be readily extended to scenarios involving common noise. A relevant practical example including the common noise will still be presented in Section \ref{lqmf.section.master_equation} to illustrate the applicability of our approach in the presence of common noise.

	
\subsection{Extended Mean Field Game}

We first introduce a limiting mean field game problem, and then we shall show that the equilibrium problem of \eqref{lqmf.N_player.state_process} and \eqref{lqmf.N_player.objective} converges to this limiting problem in Section \ref{lqmf.section.epsilon-nash_equilibrium}.
\begin{problem} \label{lqmf.mfg.problem.mfg}
	Find a mean field equilibrium $\hat{v}_t \in L_{\mathcal{G}_{\cdot}}^{2}\left(0, T ; \mathbb{R}^m\right)$ where $\mathcal{G}_t$ is the smallest completed $\sigma$-field, with all null sets, including $\sigma\left(x_0, W_s, s \leq t\right)$ such that it can minimize
	\begin{align}
	\nonumber
	J(v):=& \mathbb{E}\left[\frac{1}{2} \int_0^T x_t^{\top} Q_t x_t+v_t^{\top} P_t v_t+\left(x_t-S_t \mathbb{E}\left[\hat{x}_t\right]\right)^{\top} \bar{Q}_t\left(x_t-S_t \mathbb{E}\left[\hat{x}_t\right]\right) \right.\\
	\nonumber
	&\qquad \quad+\left. 2\left(x_t-\bar{S}_t \mathbb{E}\left[\hat{x}_t\right]\right)^{\top} N_t\left(v_t-\bar{R}_t \mathbb{E}\left[\hat{v}_t\right]\right) + \left(v_t-R_t \mathbb{E}\left[\hat{v}_t\right]\right)^{\top} \bar{P}_t\left(v_t-R_t \mathbb{E}\left[\hat{v}_t\right]\right) d t\right] \\
	\label{lqmf.mfg.objective}
	&+\mathbb{E}\left[\frac{1}{2} x_T^{\top} Q_T x_T+\frac{1}{2}\left(x_T-S_T \mathbb{E}\left[\hat{x}_T\right]\right)^{\top} \bar{Q}_T\left(x_T-S_T \mathbb{E}\left[\hat{x}_T\right]\right)\right],
	\end{align}
	where the controlled dynamics is given by
	\begin{align}
	\label{lqmf.mfg.state_process}
	d x_t=\left(A_t x_t+B_t v_t+\bar{A}_t \mathbb{E}\left[\hat{x}_t\right]+\bar{B}_t \mathbb{E}\left[\hat{v}_t\right]\right) d t+\sigma_t d W_t, \quad x(0)=x_0,
	\end{align}
	and $\hat{x}_t$ is the corresponding dynamics under the control $\hat{v}_t$, and we commonly call this the dynamics of the representation agent.	
\end{problem}

If Problem \ref{lqmf.mfg.problem.mfg} were solvable, then for each $\mathcal{G}_t^i:=\sigma\left(x_0^i, W_s^i, s \leq t\right), 1 \leq i \leq N$, we could obtain an $\varepsilon$-equilibrium strategy $\hat{v}^i$ in $L_{\mathcal{G}^i}^2\left(0, T ; \mathbb{R}^m\right)$. Since $\mathbb{E}\left[\hat{x}_t\right]$ and $\mathbb{E}\left[\hat{v}_t\right]$ are deterministic processes, it is clear that $\hat{v}^1, \ldots, \hat{v}^N$ are i.i.d. due to the independence of $W_s^i$ for $i=1,\cdots,N$. We can anticipate the convergence of the $N$ player Nash game Problem \ref{lqmf.mfg.problem.N_player} to this extended mean field game, and the rigorous arguments will be provided later in Section \ref{lqmf.section.epsilon-nash_equilibrium}.


For any admissible control $v^1$, let $\left(y^1, \ldots, y^N\right)$ denote the dynamics in Problem \eqref{lqmf.N_player.objective} controlled by $\left(v^1, \hat{v}^{-1}\right)$. Similar to Section 2 of \cite{bensoussan2016linear}, we expect $y_t^1 \to x_t^1$ and $\hat{y}_t^i \to \hat{x}_t^i$ for $i \neq 1$. Likewise, as $N \to \infty$, we should also have $\mathcal{J}^1\left(v^1, \hat{v}^{-1}\right) \rightarrow J\left(v^1\right)$. Similarly, let $\left(\hat{y}^1, \ldots, \hat{y}^N\right)$ be the dynamics controlled by $\left(\hat{v}^1, \hat{v}^{-1}\right)$, we can anticipate that $\hat{y}_t^i \to \hat{x}_t^i$ for all $i$ and $\mathcal{J}^1\left(\hat{v}^1, \hat{v}^{-1}\right) \rightarrow J\left(\hat{v}^1\right)$ as $N \to \infty$, and so $\left(\hat{v}^1, \ldots, \hat{v}^N\right)$ is an $\epsilon$-Nash equilibrium as expected.

As a similar argument in Section 3 of \cite{bensoussan2016linear}, Problem \ref{lqmf.mfg.problem.mfg} can be solved by considering a fixed point problem where the mean field terms are taken as some fixed deterministic functions {\it a priori} to this new problem, by then one can apply classical results in the theory linear-quadratic stochastic control problem, whose solution can be found in standard manner; thereafter we look for a unique fixed point such that the expected optimal state process and the expected control both coincide with the pre-determined mean field terms. To this end, consider the following classical linear-quadratic stochastic control problem:
\begin{problem} \label{lqmf.mfg.problem.fixed_point}
	For any deterministic functions $z_t$ in $L^2\left(0, T ; \mathbb{R}^n\right)$ and $w_t$ in $L^2\left(0, T ; \mathbb{R}^m\right)$, find an optimal control $\hat{v}_t \in L_{\mathcal{G}_{\cdot}}^{2}\left(0, T ; \mathbb{R}^m\right)$ such that it minimizes
	\begin{align*}
	J^{z,w}(v):=& \mathbb{E}\left[\frac{1}{2} \int_0^T x_t^{\top} Q_t x_t+v_t^{\top} P_t v_t+\left(x_t-S_t z_t\right)^{\top} \bar{Q}_t\left(x_t-S_t z_t\right) \right.\\
	&\qquad \quad+\left. 2\left(x_t-\bar{S}_t z_t\right)^{\top} N_t\left(v_t-\bar{R}_t w_t\right) + \left(v_t-R_t w_t\right)^{\top} \bar{P}_t\left(v_t-R_t w_t\right) d t\right] \\
	&+\mathbb{E}\left[\frac{1}{2} x_T^{\top} Q_T x_T+\frac{1}{2}\left(x_T-S_T z_T\right)^{\top} \bar{Q}_T\left(x_T-S_T z_T\right)\right]
	\end{align*}
	with the dynamics
	\begin{align*}
	d x_t=\left(A_t x_t+B_t v_t+\bar{A}_t z_t+\bar{B}_t w_t\right) d t+\sigma_t d W_t, \quad x(0)=x_0.
	\end{align*}
\end{problem}
\begin{theorem} \label{lqmf.mfg.theorem_smp}
	Under the assumption that there exists a scalar constant $\delta >0$ such that for all $t \in [0,T]$,
	\begin{align}
        \label{lqmf.mfg.assumption.convexity}
	(Q_t+\bar{Q}_t)  -  N_t (P_t+\bar{P}_t)^{-1} N_t^{\top} > \delta I, \quad \text{and} \quad B_t (P_t+\bar{P}_t)^{-1} B_t^{\top} > \delta I,
	\end{align}
	Problem \ref{lqmf.mfg.problem.fixed_point} is uniquely solvable and the optimal control is given by
	\begin{align}
	\label{lqmf.mfg.fixed_point.optimal_control}
	\hat{v}_t = - (P_t+\bar{P}_t)^{-1} \left(N_t^{\top} \hat{x}_t - N_t^{\top}\bar{S}_t z_t - \bar{P}_t R_t w_t + B_t^{\top} p_t\right),
	\end{align}
	where $\hat{x}$ is the optimal state process:
	\begin{align}
		\label{lqmf.mfg.fixed_point.optimal_state_process}
		\left\{\begin{array}{cl}
			d \hat{x}_t&=\left(A_t \hat{x}_t+B_t \hat{v}_t+\bar{A}_t z_t+\bar{B}_t w_t\right) d t+\sigma_t d W_t, \\
			\hat{x}(0)&=x_0,
		\end{array}\right.
	\end{align}
	and $p$ is the corresponding adjoint process:
	\begin{align}
		\label{lqmf.mfg.fixed_point.adjoint_process}
		\left\{\begin{array}{ccl}
			-d p_t &=& \left(A_t^{\top}p_t +(Q_t+\bar{Q}_t)\hat{x}_t - \bar{Q}_t S_t z_t +  N_t \hat{v}_t - N_t \bar{R}_t w_t\right)dt - \theta_t d W_t, \\
			p_T &=& (Q_T+\bar{Q}_T)\hat{x}_T - \bar{Q}_T S_T z_T.
		\end{array}\right.
	\end{align}
\end{theorem}
Its proof is deferred to Appendix A.

\begin{remark}
The solution of \eqref{lqmf.mfg.fixed_point.adjoint_process} is given by $p_t = \Xi_t \hat{x}_t + \zeta_t$, where

\begin{align}
	\label{lqmf.mfg.fixed_point.equation.riccati.Xi}
	\left\{\begin{array}{ccl}
		0&=&\frac{d\Xi_t}{dt} + \Xi_t\left(A_t - B_t (P_t+\bar{P}_t)^{-1} N_t^{\top}\right) + \left(A_t^{\top} -  N_t (P_t+\bar{P}_t)^{-1} B_t^{\top} \right)\Xi_t \\
		& &- \Xi_t B_t (P_t+\bar{P}_t)^{-1} B_t^{\top} \Xi_t + (Q_t+\bar{Q}_t)  -  N_t (P_t+\bar{P}_t)^{-1} N_t^{\top}, \\
		\Xi_T &=& Q_T+\bar{Q}_T,
	\end{array}\right.
\end{align}

\begin{align}
	\label{lqmf.mfg.fixed_point.equation.ode.zeta}
	\left\{\begin{array}{ccl}
		0&=&\frac{d\zeta_t}{dt}+(A_t^{\top}- N_t (P_t+\bar{P}_t)^{-1}B_t^{\top} - \Xi_tB_t(P_t+\bar{P}_t)^{-1}B_t^{\top}) \zeta_t\\
		& &+\left(\Xi_t(B_t(P_t+\bar{P}_t)^{-1}N_t^{\top}\bar{S}_t +\bar{A}_t)+(N_t (P_t+\bar{P}_t)^{-1}N_t^{\top}\bar{S}_t- \bar{Q}_t S_t)\right) z_t \\
		& &+ \left(\Xi_t(B_t(P_t+\bar{P}_t)^{-1}\bar{P}_t R_t+\bar{B}_t)+ (N_t (P_t+\bar{P}_t)^{-1}\bar{P}_t R_t-N_t \bar{R}_t)\right) w_t, \\
		\zeta_T &=& -\bar{Q}_T S_T z_T.
	\end{array}\right.
\end{align}
Given $z_t$ and $w_t$, the symmetric Riccati equation \eqref{lqmf.mfg.fixed_point.equation.riccati.Xi} stands alone in this system of \eqref{lqmf.mfg.fixed_point.equation.riccati.Xi}-\eqref{lqmf.mfg.fixed_point.equation.ode.zeta} and it has a unique solution if the convexity condition \eqref{lqmf.mfg.assumption.convexity} is satisfied, by then the first order linear ordinary equation \eqref{lqmf.mfg.fixed_point.equation.ode.zeta} can be solved separately.
\end{remark}

Problem \ref{lqmf.mfg.problem.fixed_point} proposes a well-defined map $\mathcal{T}:(z,w) \mapsto (\mathbb{E}[\hat{x}],\mathbb{E}[\hat{v}])$, where $\hat{x}$ and $\hat{v}$ are the corresponding optimal solutions of Problem \ref{lqmf.mfg.problem.fixed_point} given $z_t$ and $w_t$. The objective of this paper is to find two continuous functions\footnote{Note that $\hat{x}_t$ is a continuous process as it is adapted to $\mathcal{G}_t$ which is the continuous filtration generated by the initial value $x_0$ and Brownian motion $\{W_s\}_{0\leq s\leq t}$. Meanwhile, if we assume $z_t$ and $w_t$ to be continuous then  $\hat{v}_t$ is also continuous as it is derived from the Hamiltonian \eqref{lqmf.mfg.fixed_point.optimal_control}} $(z_{\cdot},w_{\cdot})$ such that
\begin{align*}
\mathbb{E}\left[\hat{x}_t\right] = z_t, \quad \text{and} \quad \mathbb{E}\left[\hat{v}_t\right] = w_t, \quad \text{for all} \quad t \in [0,T].
\end{align*}
If the fixed point $(z_t,w_t)=(\mathbb{E}[\hat{x}_t],\mathbb{E}[\hat{v}_t])$ uniquely existed, then Problem \ref{lqmf.mfg.problem.mfg} is uniquely solvable with the mean field equilibrium given by $\hat{v}$ with the corresponding controlled process given by $\hat{x}$.

By taking expectation on both sides of \eqref{lqmf.mfg.fixed_point.optimal_state_process}, \eqref{lqmf.mfg.fixed_point.adjoint_process}, and denoting $
\bar{x}_t := \mathbb{E}\left[\hat{x}_t\right]$, $\bar{p}_t = \mathbb{E} \left[p_t\right]$, $\bar{v}_t := \mathbb{E}\left[\hat{v}_t\right]$, the forward-backward system \eqref{lqmf.mfg.fixed_point.optimal_state_process} and \eqref{lqmf.mfg.fixed_point.adjoint_process} becomes:
\begin{align}
\label{lqmf.mfg.fixed_point.expected.fbode}
\left\{\begin{array}{cclccl}
d \bar{x}_t&=&\left(A_t \bar{x}_t+B_t \bar{v}_t+\bar{A}_t z_t+\bar{B}_t w_t\right) d t, &\bar{x}(0)&=&\mathbb{E}[x_0]\\
-d \bar{p}_t &=& \left(A_t^{\top}\bar{p}_t +(Q_t+\bar{Q}_t)\bar{x}_t - \bar{Q}_t S_t z_t +  N_t \bar{v}_t - N_t \bar{R}_t w_t\right)dt, &\bar{p}_T &=& (Q_T+\bar{Q}_T)\bar{x}_T - \bar{Q}_T S_T z_T,\\
\end{array}\right.
\end{align}
where $\bar{v}_t$ is the expected value of \eqref{lqmf.mfg.fixed_point.optimal_control}:
\begin{align}
\label{lqmf.mfg.fixed_point.expected.optimal_control}
\bar{v}_t = - (P_t+\bar{P}_t)^{-1} \left(N_t^{\top} \bar{x}_t - N_t^{\top}\bar{S}_t z_t - \bar{P}_t R_t w_t + B_t^{\top} \bar{p}_t\right).
\end{align}
The original problem is solvable if and only if the map $\mathcal{T}:(z,w) \mapsto (\bar{x},\bar{v})$ admits a fixed point. By replacing both the input $(z,w)$ and the output $(\bar{x},\bar{v})$ in \eqref{lqmf.mfg.fixed_point.expected.fbode} and \eqref{lqmf.mfg.fixed_point.expected.optimal_control} with $(\xi,\upsilon)$ and working backwards, it is sufficient and necessary that the following forward-backward differential equation system admits a unique solution:
\begin{align}
\label{lqmf.mfg.expected.fbode_with_control}
\left\{\begin{array}{cclccl}
d \xi_t&=&\left((A_t+\bar{A}_t) \xi_t+(B_t +\bar{B}_t) \upsilon_t\right) d t, \quad &\xi(0)&=&\mathbb{E}[x_0]\\
-d \eta_t &=& \left(A_t^{\top}\eta_t +(Q_t+\bar{\mathcal{Q}}_t)\xi_t + \bar{\mathcal{R}}_t \upsilon_t\right)dt, \quad &\eta_T &=& (Q_T+\bar{\mathcal{Q}}_T)\xi_T,
\end{array}\right.
\end{align}
where
\begin{align}
\label{lqmf.mfg.expected.optimal_control}
\upsilon_t =  - (P_t+\bar{\mathcal{P}}_t)^{-1} \left(\bar{\mathcal{S}}_t \xi_t  + B_t^{\top} \eta_t\right),
\end{align}
where, for the sake of notational simplicity, we denote
\begin{align}
\label{lqmf.mfg.equalities.notational_implicity}
\bar{\mathcal{P}}_t := \bar{P}_t (I-R_t), \quad \bar{\mathcal{Q}}_t := \bar{Q}_t (I-S_t), \quad \bar{\mathcal{Q}}_T := \bar{Q}_T (I-S_T), \quad \bar{\mathcal{R}}_t := N_t (I-\bar{R}_t), \quad \bar{\mathcal{S}}_t := N_t^{\top}(I-\bar{S}_t).
\end{align}
By substituting \eqref{lqmf.mfg.expected.optimal_control} into \eqref{lqmf.mfg.expected.fbode_with_control}, we then aim to look for the solvability of
\begin{align}
\label{lqmf.mfg.expected.fbode}
\left\{\begin{array}{ccl}
	\frac{d}{d t}\left(\begin{array}{c}
		\xi_t \\
		-\eta_t
	\end{array}\right) &=&\left(\begin{array}{cc}
		A_t+\bar{A}_t -  (B_t+\bar{B}_t) (P_t+\bar{\mathcal{P}}_t)^{-1}\bar{\mathcal{S}}_t  & -(B_t+\bar{B}_t) (P_t+\bar{\mathcal{P}}_t)^{-1}B_t^{\top} \\
		Q_t+\bar{\mathcal{Q}}_t - \bar{\mathcal{R}}_t(P_t+\bar{\mathcal{P}}_t)^{-1}\bar{\mathcal{S}}_t & A_t^{\top}-\bar{\mathcal{R}}_t(P_t+\bar{\mathcal{P}}_t)^{-1}B_t^{\top}
	\end{array}\right)\left(\begin{array}{l}
		\xi_t \\
		\eta_t
	\end{array}\right), \\
	\xi(0) &=&\mathbb{E}\left[x_0\right], \quad \eta_T=\left(Q_T+\bar{\mathcal{Q}}_T\right) \xi_T.
\end{array}\right. 
\end{align}
Once the unique existence of \eqref{lqmf.mfg.expected.fbode} is established, the mean field equilibrium $\hat{v}$, by plugging \eqref{lqmf.mfg.expected.optimal_control} into the right hand side of \eqref{lqmf.mfg.fixed_point.optimal_control}, it can be expressed as
\begin{align}
    \hat{v}_t 
    \label{lqmf.mfg.mf_equilibrium}
    =& - (P_t+\bar{P}_t)^{-1} \left(N_t^{\top} \hat{x}_t + B_t^{\top} p_t + \left( \bar{P}_t R_t (P_t+\bar{\mathcal{P}}_t)^{-1}\bar{\mathcal{S}}_t- N_t^{\top}\bar{S}_t\right) \xi_t  + \bar{P}_t R_t (P_t+\bar{\mathcal{P}}_t)^{-1}B_t^{\top} \eta_t\right),
\end{align}
where $\hat{x}_t$ is the corresponding state process, $p_t$ is the adjoint process, which can be expressed as
\begin{align}
    \label{lqmf.mfg.adjoint_process_optimal}
    \left\{\begin{array}{ccl}
    -d p_t &=& \left(A_t^{\top}p_t +(Q_t+\bar{Q}_t)\hat{x}_t - \bar{Q}_t S_t \xi_t +  N_t \hat{v}_t - N_t \bar{R}_t \upsilon_t\right)dt - \theta_t d W_t, \\
    p_T &=& (Q_T+\bar{Q}_T)\hat{x}_T - \bar{Q}_T S_T \xi_T,
    \end{array}\right.
\end{align}
where $\xi_t = \mathbb{E}[\hat{x}_t]$, $\eta_t = \mathbb{E}[p_t]$ and $\upsilon_t = \mathbb{E}[\hat{v}_t]$.

\section{Unique Existence of the Solution for \eqref{lqmf.mfg.expected.fbode}} \label{lqmf.section.unique_existence}

\subsection{Over Small Time Duration}

We first provide a sufficient condition for the unique existence of the solution to Equation \eqref{lqmf.mfg.expected.fbode} over a small time duration by applying the standard Banach fixed point argument. This result is certainly not strong enough, we put it here just to compare it with the more motivating conditions to be introduced in Subsections \ref{lqmf.section.refined_condition} and \ref{lqmf.section.global}. Also define
{\small\begin{align*}
\begin{array}{ccl}
\alpha &:=& 2 + 2\sup_{0\leq t\leq T}\|A_t+\bar{A}_t -  (B_t+\bar{B}_t) (P_t+\bar{\mathcal{P}}_t)^{-1}\bar{\mathcal{S}}_t\| + 2\sup_{0\leq t\leq T}\|(B_t+\bar{B}_t) (P_t+\bar{\mathcal{P}}_t)^{-1}B_t^{\top}\|^2\max\{1,\|Q_T+\bar{\mathcal{Q}}_T\|^2\}, \\
\beta &:=& 2 + 2\sup_{0 \leq t \leq T}\|A_t^{\top}-\bar{\mathcal{R}}_t(P_t+\bar{\mathcal{P}}_t)^{-1}B_t^{\top}\|+\sup_{0\leq t\leq T}\|Q_t+\bar{\mathcal{Q}}_t - \bar{\mathcal{R}}_t(P_t+\bar{\mathcal{P}}_t)^{-1}\bar{\mathcal{S}}_t\|^2,
\end{array}
\end{align*}}
where $\|\cdot\|$ is the spectral matrix norm: $\|A\| = \sqrt{\tr{\left(A^{\top}A\right)}}$.
\begin{theorem} \label{lqmf.mfg.theorem.local.solution}
For all small enough time duration $T$ such that
\begin{align}
\label{lqmf.mfg.condition.local.solution}
e^{(\alpha+\beta)T} < 2,
\end{align}
then the forward-backward equation system \eqref{lqmf.mfg.expected.fbode} admits a unique solution.
\end{theorem}

\begin{remark}
	In general, $B_t (P_t+\bar{\mathcal{P}}_t)^{-1}B_t^{\top}$ and $Q_t$ can be large and the condition \eqref{lqmf.mfg.condition.local.solution} in Theorem \ref{lqmf.mfg.theorem.local.solution} is too restrictive. It limits the size of $T$ at an exponential rate $e^{(\alpha+\beta)T}$. Even if we reduce the linear-quadratic extended mean field game to a linear-quadratic control problem, which is mostly generically solvable, by setting $\bar{A}_t = \bar{B}_t = \bar{Q}_t = \bar{P}_t = \bar{S}_t = \bar{R}_t = 0$, the condition \eqref{lqmf.mfg.condition.local.solution} is not necessarily guaranteed since all $B_t$, $Q_t$ and $Q_T$ are still present. In the rest of this section, We shall next provide two far better conditions based on the works \cite{bensoussan2016linear} and \cite{peng1999fully} that not only cover the global solvability of most of the classical linear-quadratic stochastic control problems considered in the literature, but also can be applied to a wide class of extended mean field game problems.
\end{remark}

\subsection{A Refined Sufficient Condition} \label{lqmf.section.refined_condition}

We next apply an improved version of method proposed by \cite{bensoussan2016linear} to obtain a far less restrictive condition. Assume the matrix-valued function $Q_{\cdot}$ and $B_{\cdot} P_{\cdot}^{-1}B_{\cdot}^{\top}$ are invertible. We define the weighted norm $\vertiii{\cdot}_{0,T}$ by:
\begin{align*}
	\vertiii{\phi}_{0,T}:=\sup_{0 \leq t \leq T} \sqrt{\left\|(B_t P_t^{-1}B_t^{\top})^{\frac{1}{2}}\phi(T, t) Q_{T}^{1 / 2}\right\|^{2}+\int_t^T\left\|(B_t P_t^{-1}B_t^{\top})^{\frac{1}{2}}\phi(s, t) Q_{s}^{1 / 2}\right\|^{2} ds},
\end{align*}
and for any bounded matrix function $K: [0,T]\to R^{n\times n}$, define the weighted norm $\vertiii{\cdot}_{1,T}$, $\vertiii{\cdot}_{2,T}$ and $\vertiii{\cdot}_{3,T}$ as
\begin{itemize}
	\item[1.] $\vertiii{K}_{1,T}:=\sup_{0 \leq t \leq T} \|Q_t^{-1/2}K(t) Q_t^{-1/2} \|$;
	\item[2.] $\vertiii{K}_{2,T}:= \sup_{0 \leq t \leq T}\|Q_t^{-1/2} K(t)(B_t P_t^{-1}B_t^{\top})^{-\frac{1}{2}}\|$;
	\item[3.] $\vertiii{K}_{3,T} := \sup_{0 \leq t \leq T}\|(B_t P_t^{-1}B_t^{\top})^{-\frac{1}{2}}K(t)(B_t P_t^{-1}B_t^{\top})^{-\frac{1}{2}}\|$.
\end{itemize}

\begin{theorem} \label{lqmf.mfg.theorem.global_local.solution}
Let $\phi(t,s)$ be the fundamental solution associated with $A_t^{\top}$. 
The forward-backward equation system \eqref{lqmf.mfg.expected.fbode} admits a unique solution if either one of following conditions is satisfied:
\begin{enumerate}
	\item Most coefficients related to the mean field terms are small enough such that
	{\scriptsize
	\begin{align}
	\label{lqmf.mfg.condition.refined.global}
	\vertiii{(\bar{A} -  (B+\bar{B}) (P+\bar{\mathcal{P}})^{-1}\bar{\mathcal{S}})^{\top}}_{2,T}^2+\vertiii{\left(B_t P_t^{-1}\bar{\mathcal{P}}_t-\bar{B}_t\right)(P_t+\bar{\mathcal{P}}_t)^{-1}B_t^{\top}}_{3,T}^2+\vertiii{\bar{\mathcal{R}}(P+\bar{\mathcal{P}})^{-1}B^{\top}}_{2,T}^2 +\left(\vertiii{\bar{\mathcal{Q}}}_{1,T} + \vertiii{\bar{\mathcal{R}}(P+\bar{\mathcal{P}})^{-1}\bar{\mathcal{S}}}_{1,T}\right)^2 < 1;
	\end{align}}
	\item for some fixed $T_0 > 0$, $\vertiii{(\bar{A} -  (B+\bar{B}) (P+\bar{\mathcal{P}})^{-1}\bar{\mathcal{S}})^{\top}}_{2,T_0}^2 + \vertiii{\left(B_t P_t^{-1}\bar{\mathcal{P}}_t-\bar{B}_t\right)(P_t+\bar{\mathcal{P}}_t)^{-1}B_t^{\top}}_{3,T_0}^2 \neq 0$, \\$\left(\vertiii{\bar{\mathcal{Q}}}_{1,T_0} + \vertiii{\bar{\mathcal{R}}(P+\bar{\mathcal{P}})^{-1}\bar{\mathcal{S}}}_{1,T_0}\right)^2 + \vertiii{\bar{\mathcal{R}}(P+\bar{\mathcal{P}})^{-1}B^{\top}}_{2,T_0}^2<1$, and
	{\scriptsize
	\begin{align}
	\nonumber T <&  T_0 \bigwedge \left(\frac{1}{\vertiii{\phi}_{0,T_0}\sqrt{\vertiii{(\bar{A} -  (B+\bar{B}) (P+\bar{\mathcal{P}})^{-1}\bar{\mathcal{S}})^{\top}}_{2,T_0}^2 + \vertiii{\left(B_t P_t^{-1}\bar{\mathcal{P}}_t-\bar{B}_t\right)(P_t+\bar{\mathcal{P}}_t)^{-1}B_t^{\top}}_{3,T_0}^2}} \right.\\
	\label{lqmf.mfg.condition.refined.local} &\left. \qquad \cdot \frac{1 - \sqrt{\left(\vertiii{\bar{\mathcal{Q}}}_{1,T_0} + \vertiii{\bar{\mathcal{R}}(P+\bar{\mathcal{P}})^{-1}\bar{\mathcal{S}}}_{1,T_0}\right)^2 + \vertiii{\bar{\mathcal{R}}(P+\bar{\mathcal{P}})^{-1}B^{\top}}_{2,T_0}^2}}{ 1 + \sqrt{\left(\vertiii{\bar{\mathcal{Q}}}_{1,T_0} + \vertiii{\bar{\mathcal{R}}(P+\bar{\mathcal{P}})^{-1}\bar{\mathcal{S}}}_{1,T_0}\right)^2 + \vertiii{\bar{\mathcal{R}}(P+\bar{\mathcal{P}})^{-1}B^{\top}}_{2,T_0}^2}}\right)^2.
	\end{align}
	}
\end{enumerate}
\end{theorem}

\begin{proof}
Let $L^2_Q(0,T;\mathbb{R}^n)$ and $L^2_{B P^{-1}B}(0,T;\mathbb{R}^n)$ be two Hilbert spaces of functions endowed with the respective inner products, $
\braket{z,z'}_Q :=\int_0^T z_t^{\top}Q_tz'_t dt +z_T^{\top}Q_Tz'_T$, for any $z,z' \in L^2_Q(0,T;\mathbb{R}^n)$, and $\braket{q,q'}_{B P^{-1}B} := \int_0^T q_t^{\top}B_t P_t^{-1}B_t^{\top}q'_t dt$, for any $q,q' \in L^2_{B P^{-1}B}(0,T;\mathbb{R}^n)$.
We consider the Cartesian product space $L^2_{Q\times B P^{-1}B}(0,T;\mathbb{R}^{2n}) := L^2_Q(0,T;\mathbb{R}^n)\times L^2_{B P^{-1}B}(0,T;\mathbb{R}^n)$, where the inner product is defined as follows: for any $X:=(z,q)$ and $X':=(z',q')$ in $L^2_{Q\times B P^{-1}B}(0,T;\mathbb{R}^{2n})$, $
\braket{X,X'}_{Q\times B P^{-1}B} := \braket{z,z'}_Q + \braket{q,q'}_{B P^{-1}B}$.

Since the upper right corner block of the matrix in \eqref{lqmf.mfg.expected.fbode} can be rewritten as
\begin{align*}
-(B_t+\bar{B}_t) (P_t+\bar{\mathcal{P}}_t)^{-1}B_t^{\top} = - B_t P_t^{-1}B_t^{\top} +  \left(B_t P_t^{-1}\bar{\mathcal{P}}_t-\bar{B}_t\right)(P_t+\bar{\mathcal{P}}_t)^{-1}B_t^{\top},
\end{align*}
for any $X:=(z,q)^{\top} \in L^2_{Q\times B P^{-1}B}(0,T;\mathbb{R}^{2n})$, we consider the map $X:=(z,q)^{\top} \mapsto (\xi,\eta)^{\top}=:Y$, where $(\xi,\eta)$ satisfies
\begin{align}
\label{lqmf.mfg.expected.fbode.mapT}
\left\{\begin{array}{rcl}
\frac{d}{d t}\left(\begin{array}{c}
\xi_t \\
-\eta_t
\end{array}\right) &=&\left(\begin{array}{cc}
A_t  & -B_t P_t^{-1}B_t^{\top} \\
Q_t & A_t^{\top}
\end{array}\right)\left(\begin{array}{l}
\xi_t \\
\eta_t
\end{array}\right) \\
& &+ \left(\begin{array}{cc}
\bar{A}_t -  (B_t+\bar{B}_t) (P_t+\bar{\mathcal{P}}_t)^{-1}\bar{\mathcal{S}}_t  & \left(B_t P_t^{-1}\bar{\mathcal{P}}_t-\bar{B}_t\right)(P_t+\bar{\mathcal{P}}_t)^{-1}B_t^{\top} \\
\bar{\mathcal{Q}}_t - \bar{\mathcal{R}}_t(P_t+\bar{\mathcal{P}}_t)^{-1}\bar{\mathcal{S}}_t & -\bar{\mathcal{R}}_t(P_t+\bar{\mathcal{P}}_t)^{-1}B_t^{\top}
\end{array}\right)\left(\begin{array}{l}
z_t \\
q_t
\end{array}\right), \\
\xi(0) &=&\mathbb{E}\left[x_0\right], \quad \eta_T=Q_T\xi_T+\bar{\mathcal{Q}}_Tz_T.
\end{array}\right.
\end{align}
It is easy to check that the equation system \eqref{lqmf.mfg.expected.fbode.mapT} satisfies the monotonicity condition in \cite{peng1999fully}, thus it admits a unique solution pair $(\xi,\eta)$. Similar to the proof of Theorem \ref{lqmf.mfg.theorem.local.solution}, we shall also apply the Banach fixed point theorem to show that the map $X \mapsto Y$ has a unique fixed point, and it suffices to show that the map $X \mapsto Y$ is a contraction if $\mathbb{E}\left[x_0\right] = 0$.
It follows by the chain rule that
\begin{align*}
\xi_T^{\top}\eta_T =& \xi_T^{\top}Q_T \xi_T+\xi_T^{\top}\bar{\mathcal{Q}}_T z_T = \int_0^T \eta_t^{\top} d\xi_t + \int_0^T \xi_t^{\top} d\eta_t \\
=& \int_0^T \left(\eta_t^{\top}A_t \xi_t + \eta_t^{\top}\left(\bar{A}_t -  (B_t+\bar{B}_t) (P_t+\bar{\mathcal{P}}_t)^{-1}\bar{\mathcal{S}}_t\right) z_t  - \eta_t^{\top}B_t P_t^{-1}B_t^{\top}\eta_t +  \eta_t^{\top}\left(B_t P_t^{-1}\bar{\mathcal{P}}_t-\bar{B}_t\right)(P_t+\bar{\mathcal{P}}_t)^{-1}B_t^{\top} q_t\right) dt \\
&- \int_0^T \left(\xi_t^{\top} A_t^{\top}\eta_t -\xi_t^{\top}\bar{\mathcal{R}}_t(P_t+\bar{\mathcal{P}}_t)^{-1}B_t^{\top}q_t +\xi_t^{\top}Q_t\xi_t+\xi_t^{\top}\left(\bar{\mathcal{Q}}_t - \bar{\mathcal{R}}_t(P_t+\bar{\mathcal{P}}_t)^{-1}\bar{\mathcal{S}}_t \right)z_t\right)dt.
\end{align*}
By some rearrangments, we have
\begin{align}
\nonumber \|Y\|_{Q\times B P^{-1}B}^2 =& \|\xi\|_Q^2 + \|\eta\|_{B P^{-1}B^{\top}}^2 \\
\nonumber =& \int_0^T \xi_t^{\top}Q_t\xi_tdt + \xi_T^{\top}Q_T \xi_T + \int_0^T \eta_t^{\top}B_t P_t^{-1}B_t^{\top}\eta_tdt \\
\nonumber =& \int_0^T \left(\eta_t^{\top}\left(\bar{A}_t -  (B_t+\bar{B}_t) (P_t+\bar{\mathcal{P}}_t)^{-1}\bar{\mathcal{S}}_t\right) z_t +  \eta_t^{\top}\left(B_t P_t^{-1}\bar{\mathcal{P}}_t-\bar{B}_t\right)(P_t+\bar{\mathcal{P}}_t)^{-1}B_t^{\top} q_t\right) dt \\
\label{lqmf.mfg.inequality.Y_norm} &+ \int_0^T \left(\xi_t^{\top}\bar{\mathcal{R}}_t(P_t+\bar{\mathcal{P}}_t)^{-1}B_t^{\top}q_t-\xi_t^{\top}\left(\bar{\mathcal{Q}}_t - \bar{\mathcal{R}}_t(P_t+\bar{\mathcal{P}}_t)^{-1}\bar{\mathcal{S}}_t \right)z_t\right)dt - \xi_T^{\top}\bar{\mathcal{Q}}_T z_T.
\end{align}
We apply the Cauchy-Schwarz inequality and obtain
\begin{align*}
\|Y\|_{Q\times B P^{-1}B}^2
\leq& \vertiii{(\bar{A} -  (B+\bar{B}) (P+\bar{\mathcal{P}})^{-1}\bar{\mathcal{S}})^{\top}}_{2,T} \cdot \left|\int_0^T\eta_t^{\top}(B_t P_t^{-1}B_t^{\top})^{\frac{1}{2}}Q^{\frac{1}{2}}z_tdt\right| \\
&+ \vertiii{\left(B_t P_t^{-1}\bar{\mathcal{P}}_t-\bar{B}_t\right)(P_t+\bar{\mathcal{P}}_t)^{-1}B_t^{\top}}_{3,T} \cdot \left|\int_0^T\eta_t^{\top}B_t P_t^{-1}B_t^{\top}q_tdt\right| \\
&+ \vertiii{\bar{\mathcal{R}}(P+\bar{\mathcal{P}})^{-1}B^{\top}}_{2,T} \cdot\left|\int_0^T\xi_t^{\top}Q^{\frac{1}{2}}(B_t P_t^{-1}B_t^{\top})^{\frac{1}{2}}q_tdt\right| \\
&+\left(\vertiii{\bar{\mathcal{Q}}}_{1,T} + \vertiii{\bar{\mathcal{R}}(P+\bar{\mathcal{P}})^{-1}\bar{\mathcal{S}}}_{1,T}\right) \cdot\left(\left|\int_0^T\xi_t^{\top}Qz_tdt\right| + \left|\xi_T^{\top}Q_Tz_T\right|\right) \\
\leq& \vertiii{(\bar{A} -  (B+\bar{B}) (P+\bar{\mathcal{P}})^{-1}\bar{\mathcal{S}})^{\top}}_{2,T} \cdot \|\eta\|_{B P^{-1}B^{\top}}\|z\|_{Q} \\
&+ \vertiii{\left(B_t P_t^{-1}\bar{\mathcal{P}}_t-\bar{B}_t\right)(P_t+\bar{\mathcal{P}}_t)^{-1}B_t^{\top}}_{3,T} \cdot \|\eta\|_{B P^{-1}B^{\top}} \cdot \|q\|_{B P^{-1}B^{\top}} \\
&+ \vertiii{\bar{\mathcal{R}}(P+\bar{\mathcal{P}})^{-1}B^{\top}}_{2,T} \cdot \|\xi\|_{Q} \cdot \|q\|_{B P^{-1}B^{\top}} \\
&+\left(\vertiii{\bar{\mathcal{Q}}}_{1,T} + \vertiii{\bar{\mathcal{R}}(P+\bar{\mathcal{P}})^{-1}\bar{\mathcal{S}}}_{1,T}\right) \cdot \|\xi\|_{Q} \cdot \|z\|_{Q} \\
<& \bigg\{\vertiii{(\bar{A} -  (B+\bar{B}) (P+\bar{\mathcal{P}})^{-1}\bar{\mathcal{S}})^{\top}}_{2,T}^2+\vertiii{\left(B_t P_t^{-1}\bar{\mathcal{P}}_t-\bar{B}_t\right)(P_t+\bar{\mathcal{P}}_t)^{-1}B_t^{\top}}_{3,T}^2 \\
&\quad +\vertiii{\bar{\mathcal{R}}(P+\bar{\mathcal{P}})^{-1}B^{\top}}_{2,T}^2+\left(\vertiii{\bar{\mathcal{Q}}}_{1,T} + \vertiii{\bar{\mathcal{R}}(P+\bar{\mathcal{P}})^{-1}\bar{\mathcal{S}}}_{1,T}\right)^2\bigg\}^{\frac{1}{2}} \cdot \|Y\|_{Q\times B P^{-1}B} \cdot \|X\|_{Q\times B P^{-1}B},
\end{align*}
which gives, under the condition \eqref{lqmf.mfg.condition.refined.global}, that $\|Y\|_{Q\times B P^{-1}B}< \|X\|_{Q\times B P^{-1}B}$, and therefore, the map $X \mapsto Y$ is a contractive one. We next prove that the condition \eqref{lqmf.mfg.condition.refined.local} also introduces the unique existence of the solution of \eqref{lqmf.mfg.expected.fbode}. Let $\phi(t,s)$ be the fundamental solution associated with $A_t^{\top}$, then we have
\begin{align*}
\eta_t = \phi(T,t) \left(Q_T \xi_T+\bar{\mathcal{Q}}_T z_T\right) + \int_t^T \phi(s,t) \left(Q_s\xi_s+\left(\bar{\mathcal{Q}}_s - \bar{\mathcal{R}}_s(P_s+\bar{\mathcal{P}}_s)^{-1}\bar{\mathcal{S}}_s \right)z_s-\bar{\mathcal{R}}_t(P_t+\bar{\mathcal{P}}_t)^{-1}B_t^{\top}q_s\right) ds, \quad \text{for any} \quad t\in [0,T].
\end{align*}
By the Cauchy-Schwarz inequality, we have
{\small
\begin{align}
\nonumber
(B_t P_t^{-1}B_t^{\top})^{\frac{1}{2}}\phi(T,t) Q_T \xi_T+ \int_t^T (B_t P_t^{-1}B_t^{\top})^{\frac{1}{2}}\phi(s,t) Q_s\xi_s ds &\leq \left(\sqrt{\|(B_t P_t^{-1}B_t^{\top})^{\frac{1}{2}}\phi(T,t) Q_T^{\frac{1}{2}}\|^2 + \int_t^T \|(B_t P_t^{-1}B_t^{\top})^{\frac{1}{2}}\phi(s,t) Q_s^{\frac{1}{2}}\|^2 ds}\right) \cdot\|\xi\|_Q \\
\label{lamf.mfg.inequality1.refined}
&= \vertiii{\phi}_{0,T} \cdot \|\xi\|_Q, \quad \text{for any} \quad t\in [0,T],
\end{align}}
and similarly, we have
\begin{align}
\nonumber
&(B_t P_t^{-1}B_t^{\top})^{\frac{1}{2}}\phi(T,t) \bar{\mathcal{Q}}_T z_T + \int_t^T (B_t P_t^{-1}B_t^{\top})^{\frac{1}{2}}\phi(s,t) \bar{\mathcal{Q}}_sz_s ds \\
\nonumber
&\leq \vertiii{\bar{\mathcal{Q}}}_{1,T} \cdot\left(\left|(B_t P_t^{-1}B_t^{\top})^{\frac{1}{2}}\phi(T,t) Q_T z_T\right| + \left|\int_t^T (B_t P_t^{-1}B_t^{\top})^{\frac{1}{2}}\phi(s,t) Q_sz_s ds\right|\right) \\
\label{lamf.mfg.inequality2.refined}
&\leq \vertiii{\bar{\mathcal{Q}}}_{1,T} \cdot \vertiii{\phi}_{0,T} \cdot \|z\|_Q,
\end{align}
and
\begin{align}
\nonumber
&\int_t^T (B_t P_t^{-1}B_t^{\top})^{\frac{1}{2}}\phi(s,t) \left(\left( \bar{\mathcal{R}}_s(P_s+\bar{\mathcal{P}}_s)^{-1}\bar{\mathcal{S}}_s \right)z_s+\bar{\mathcal{R}}_t(P_t+\bar{\mathcal{P}}_t)^{-1}B_t^{\top}q_s\right) ds \\
\nonumber
&\leq\int_t^T (B_t P_t^{-1}B_t^{\top})^{\frac{1}{2}}\phi(s,t)  \bar{\mathcal{R}}_s(P_s+\bar{\mathcal{P}}_s)^{-1}\bar{\mathcal{S}}_s z_sds+\int_t^T(B_t P_t^{-1}B_t^{\top})^{\frac{1}{2}}\phi(s,t)\bar{\mathcal{R}}_t(P_t+\bar{\mathcal{P}}_t)^{-1}B_t^{\top}q_s ds \\
\nonumber
&=\vertiii{\bar{\mathcal{R}}(P+\bar{\mathcal{P}})^{-1}\bar{\mathcal{S}}}_{1,T} \cdot \left|\int_t^T (B_t P_t^{-1}B_t^{\top})^{\frac{1}{2}}\phi(s,t)  Q_s z_sds\right|\\
\nonumber
&\quad+\vertiii{\bar{\mathcal{R}}(P+\bar{\mathcal{P}})^{-1}B^{\top}}_{2,T} \cdot \left|\int_t^T(B_t P_t^{-1}B_t^{\top})^{\frac{1}{2}}\phi(s,t)Q_s^{\frac{1}{2}}(B_t P_t^{-1}B_t^{\top})^{\frac{1}{2}}q_s ds\right| \\
\label{lamf.mfg.inequality3.refined}
&\leq \vertiii{\bar{\mathcal{R}}(P+\bar{\mathcal{P}})^{-1}\bar{\mathcal{S}}}_{1,T} \cdot \vertiii{\phi}_{0,T} \cdot \|z\|_Q + \vertiii{\bar{\mathcal{R}}(P+\bar{\mathcal{P}})^{-1}B^{\top}}_{2,T} \cdot \|\vertiii{\phi}_{0,T} \cdot \|q\|_{B P^{-1}B^{\top}}.
\end{align}
Hence, combining \eqref{lamf.mfg.inequality1.refined}, \eqref{lamf.mfg.inequality2.refined} and \eqref{lamf.mfg.inequality3.refined} together, we obtain
{\footnotesize
\begin{align*}
\eta_t^{\top}B_t P_t^{-1}B_t^{\top}\eta_t \leq& \vertiii{\phi}_{0,T}^2 \cdot\bigg( \|\xi\|_Q + \left(\vertiii{\bar{\mathcal{Q}}}_{1,T} + \vertiii{\bar{\mathcal{R}}(P+\bar{\mathcal{P}})^{-1}\bar{\mathcal{S}}}_{1,T}\right) \|z\|_Q + \vertiii{\bar{\mathcal{R}}(P+\bar{\mathcal{P}})^{-1}B^{\top}}_{2,T} \cdot \|q\|_{B P^{-1}B^{\top}}\bigg)^2 \\
\leq& \vertiii{\phi}_{0,T}^2 \cdot\bigg( \|\xi\|_Q + \sqrt{\left(\vertiii{\bar{\mathcal{Q}}}_{1,T} + \vertiii{\bar{\mathcal{R}}(P+\bar{\mathcal{P}})^{-1}\bar{\mathcal{S}}}_{1,T}\right)^2 + \vertiii{\bar{\mathcal{R}}(P+\bar{\mathcal{P}})^{-1}B^{\top}}_{2,T}^2}\cdot \sqrt{\|z\|_Q^2 + \|q\|_{B P^{-1}B^{\top}}^2}\bigg)^2 \\
\leq& \vertiii{\phi}_{0,T}^2 \cdot\bigg( \|Y\|_{Q\times B P^{-1}B} + \sqrt{\left(\vertiii{\bar{\mathcal{Q}}}_{1,T} + \vertiii{\bar{\mathcal{R}}(P+\bar{\mathcal{P}})^{-1}\bar{\mathcal{S}}}_{1,T}\right)^2 + \vertiii{\bar{\mathcal{R}}(P+\bar{\mathcal{P}})^{-1}B^{\top}}_{2,T}^2} \cdot \|X\|_{Q\times B P^{-1}B}\bigg)^2,
\end{align*}
}
that is
{\footnotesize
\begin{align*}
\|\eta\|_{B P^{-1}B^{\top}} =& \sqrt{\int_0^T \eta_t^{\top}B_t P_t^{-1}B_t^{\top}\eta_t dt} \\
\leq& \sqrt{T}\vertiii{\phi}_{0,T} \cdot\bigg( \|Y\|_{Q\times B P^{-1}B} +\sqrt{\left(\vertiii{\bar{\mathcal{Q}}}_{1,T} + \vertiii{\bar{\mathcal{R}}(P+\bar{\mathcal{P}})^{-1}\bar{\mathcal{S}}}_{1,T}\right)^2 + \vertiii{\bar{\mathcal{R}}(P+\bar{\mathcal{P}})^{-1}B^{\top}}_{2,T}^2} \cdot \|X\|_{Q\times B P^{-1}B}\bigg).
\end{align*}
}
Hence, by \eqref{lqmf.mfg.inequality.Y_norm} and the condition \eqref{lqmf.mfg.condition.refined.local}, we deduce that
{\footnotesize
\begin{align*}
\|Y\|_{Q\times B P^{-1}B}^2 \leq& \vertiii{(\bar{A} -  (B+\bar{B}) (P+\bar{\mathcal{P}})^{-1}\bar{\mathcal{S}})^{\top}}_{2,T} \cdot \|\eta\|_{B P^{-1}B^{\top}} \cdot \|z\|_{Q} \\
&+ \vertiii{\left(B_t P_t^{-1}\bar{\mathcal{P}}_t-\bar{B}_t\right)(P_t+\bar{\mathcal{P}}_t)^{-1}B_t^{\top}}_{3,T}  \cdot \|\eta\|_{B P^{-1}B^{\top}} \cdot \|q\|_{B P^{-1}B^{\top}} \\
&+ \vertiii{\bar{\mathcal{R}}(P+\bar{\mathcal{P}})^{-1}B^{\top}}_{2,T} \cdot \|\xi\|_{Q} \cdot \|q\|_{B P^{-1}B^{\top}} +\left(\vertiii{\bar{\mathcal{Q}}}_{1,T} + \vertiii{\bar{\mathcal{R}}(P+\bar{\mathcal{P}})^{-1}\bar{\mathcal{S}}}_{1,T}\right)\|\xi\|_{Q} \cdot \|z\|_{Q} \\
\leq& \|\eta\|_{B P^{-1}B^{\top}} \cdot\sqrt{\vertiii{(\bar{A} -  (B+\bar{B}) (P+\bar{\mathcal{P}})^{-1}\bar{\mathcal{S}})^{\top}}_{2,T}^2 + \vertiii{\left(B_t P_t^{-1}\bar{\mathcal{P}}_t-\bar{B}_t\right)(P_t+\bar{\mathcal{P}}_t)^{-1}B_t^{\top}}_{3,T}^2} \cdot \|X\|_{Q\times B P^{-1}B}  \\
&+ \sqrt{\left(\vertiii{\bar{\mathcal{Q}}}_{1,T} + \vertiii{\bar{\mathcal{R}}(P+\bar{\mathcal{P}})^{-1}\bar{\mathcal{S}}}_{1,T}\right)^2 + \vertiii{\bar{\mathcal{R}}(P+\bar{\mathcal{P}})^{-1}B^{\top}}_{2,T}^2} \cdot \|Y\|_{Q\times B P^{-1}B} \cdot \|X\|_{Q\times B P^{-1}B} \\
\leq& \sqrt{T}\cdot\vertiii{\phi}_{0,T} \cdot\bigg( \|Y\|_{Q\times B P^{-1}B} + \sqrt{\left(\vertiii{\bar{\mathcal{Q}}}_{1,T} + \vertiii{\bar{\mathcal{R}}(P+\bar{\mathcal{P}})^{-1}\bar{\mathcal{S}}}_{1,T}\right)^2 + \vertiii{\bar{\mathcal{R}}(P+\bar{\mathcal{P}})^{-1}B^{\top}}_{2,T}^2} \cdot \|X\|_{Q\times B P^{-1}B}\bigg) \\
 &\cdot \sqrt{\vertiii{(\bar{A} -  (B+\bar{B}) (P+\bar{\mathcal{P}})^{-1}\bar{\mathcal{S}})^{\top}}_{2,T}^2 + \vertiii{\left(B_t P_t^{-1}\bar{\mathcal{P}}_t-\bar{B}_t\right)(P_t+\bar{\mathcal{P}}_t)^{-1}B_t^{\top}}_{3,T}^2} \cdot \|X\|_{Q\times B P^{-1}B} \\
&+ \sqrt{\left(\vertiii{\bar{\mathcal{Q}}}_{1,T} + \vertiii{\bar{\mathcal{R}}(P+\bar{\mathcal{P}})^{-1}\bar{\mathcal{S}}}_{1,T}\right)^2 + \vertiii{\bar{\mathcal{R}}(P+\bar{\mathcal{P}})^{-1}B^{\top}}_{2,T}^2} \cdot \|Y\|_{Q\times B P^{-1}B} \cdot \|X\|_{Q\times B P^{-1}B} \\
\leq& \sqrt{T}\cdot\vertiii{\phi}_{0,T_0} \cdot\bigg( \|Y\|_{Q\times B P^{-1}B} + \sqrt{\left(\vertiii{\bar{\mathcal{Q}}}_{1,T_0} + \vertiii{\bar{\mathcal{R}}(P+\bar{\mathcal{P}})^{-1}\bar{\mathcal{S}}}_{1,T_0}\right)^2 + \vertiii{\bar{\mathcal{R}}(P+\bar{\mathcal{P}})^{-1}B^{\top}}_{2,T_0}^2} \cdot \|X\|_{Q\times B P^{-1}B}\bigg) \\
 &\cdot \sqrt{\vertiii{(\bar{A} -  (B+\bar{B}) (P+\bar{\mathcal{P}})^{-1}\bar{\mathcal{S}})^{\top}}_{2,T_0}^2 + \vertiii{\left(B_t P_t^{-1}\bar{\mathcal{P}}_t-\bar{B}_t\right)(P_t+\bar{\mathcal{P}}_t)^{-1}B_t^{\top}}_{3,T_0}^2} \cdot \|X\|_{Q\times B P^{-1}B} \\
&+ \sqrt{\left(\vertiii{\bar{\mathcal{Q}}}_{1,T_0} + \vertiii{\bar{\mathcal{R}}(P+\bar{\mathcal{P}})^{-1}\bar{\mathcal{S}}}_{1,T_0}\right)^2 + \vertiii{\bar{\mathcal{R}}(P+\bar{\mathcal{P}})^{-1}B^{\top}}_{2,T_0}^2} \cdot \|Y\|_{Q\times B P^{-1}B} \cdot \|X\|_{Q\times B P^{-1}B}.
\end{align*}
}
Divide both sides by $\|X\|_{Q\times B P^{-1}B}$, and denote $\alpha:=\frac{\|Y\|_{Q\times B P^{-1}B}}{\|X\|_{Q\times B P^{-1}B}}$, it can be seem that $f(\alpha):=\alpha^2 - b(T)\alpha -c(T) < 0$
where
\begin{align}
    \nonumber b(T) :=& \sqrt{T} \cdot\vertiii{\phi}_{0,T_0} \cdot \sqrt{\vertiii{(\bar{A} -  (B+\bar{B}) (P+\bar{\mathcal{P}})^{-1}\bar{\mathcal{S}})^{\top}}_{2,T_0}^2 + \vertiii{\left(B_t P_t^{-1}\bar{\mathcal{P}}_t-\bar{B}_t\right)(P_t+\bar{\mathcal{P}}_t)^{-1}B_t^{\top}}_{3,T_0}^2} \\
    &+ \sqrt{\left(\vertiii{\bar{\mathcal{Q}}}_{1,T_0} + \vertiii{\bar{\mathcal{R}}(P+\bar{\mathcal{P}})^{-1}\bar{\mathcal{S}}}_{1,T_0}\right)^2 + \vertiii{\bar{\mathcal{R}}(P+\bar{\mathcal{P}})^{-1}B^{\top}}_{2,T_0}^2} > 0,
\end{align}
and
\begin{align}
    \nonumber c(T) :=& \sqrt{T} \cdot\vertiii{\phi}_{0,T_0} \cdot \sqrt{\left(\vertiii{\bar{\mathcal{Q}}}_{1,T_0} + \vertiii{\bar{\mathcal{R}}(P+\bar{\mathcal{P}})^{-1}\bar{\mathcal{S}}}_{1,T_0}\right)^2 + \vertiii{\bar{\mathcal{R}}(P+\bar{\mathcal{P}})^{-1}B^{\top}}_{2,T_0}^2} \\
    &\cdot \sqrt{\vertiii{(\bar{A} -  (B+\bar{B}) (P+\bar{\mathcal{P}})^{-1}\bar{\mathcal{S}})^{\top}}_{2,T_0}^2 + \vertiii{\left(B_t P_t^{-1}\bar{\mathcal{P}}_t-\bar{B}_t\right)(P_t+\bar{\mathcal{P}}_t)^{-1}B_t^{\top}}_{3,T_0}^2} > 0.
\end{align}
Note that $f(\alpha)$ is strictly convex (see Figure.\ref{lqmf.mfg.figure.quadratic_function}) and it attains its minimum at point $\frac{b(T)}{2}>0$, and it has an intercept $c(T)>0$, therefore we can deduce that all the positive solutions of the inequality $f(\alpha)<0$ lie in $(0,1)$ if and only if $f(1)>0$, that is to say,
\begin{align*}
&\sqrt{T}\vertiii{\phi}_{0,T_0}\bigg( 1 + \sqrt{\left(\vertiii{\bar{\mathcal{Q}}}_{1,T_0} + \vertiii{\bar{\mathcal{R}}(P+\bar{\mathcal{P}})^{-1}\bar{\mathcal{S}}}_{1,T_0}\right)^2 + \vertiii{\bar{\mathcal{R}}(P+\bar{\mathcal{P}})^{-1}B^{\top}}_{2,T_0}^2}\bigg) \\
\cdot& \sqrt{\vertiii{(\bar{A} -  (B+\bar{B}) (P+\bar{\mathcal{P}})^{-1}\bar{\mathcal{S}})^{\top}}_{2,T_0}^2 + \vertiii{\left(B_t P_t^{-1}\bar{\mathcal{P}}_t-\bar{B}_t\right)(P_t+\bar{\mathcal{P}}_t)^{-1}B_t^{\top}}_{3,T_0}^2}  \\
&+ \sqrt{\left(\vertiii{\bar{\mathcal{Q}}}_{1,T_0} + \vertiii{\bar{\mathcal{R}}(P+\bar{\mathcal{P}})^{-1}\bar{\mathcal{S}}}_{1,T_0}\right)^2 + \vertiii{\bar{\mathcal{R}}(P+\bar{\mathcal{P}})^{-1}B^{\top}}_{2,T_0}^2} < 1,
\end{align*}
which completes the proof.
\begin{figure}[t]
	\centering
	\begin{tabular}{cc}
		\begin{minipage}[t]{3.5in}
			\includegraphics[width=3in]{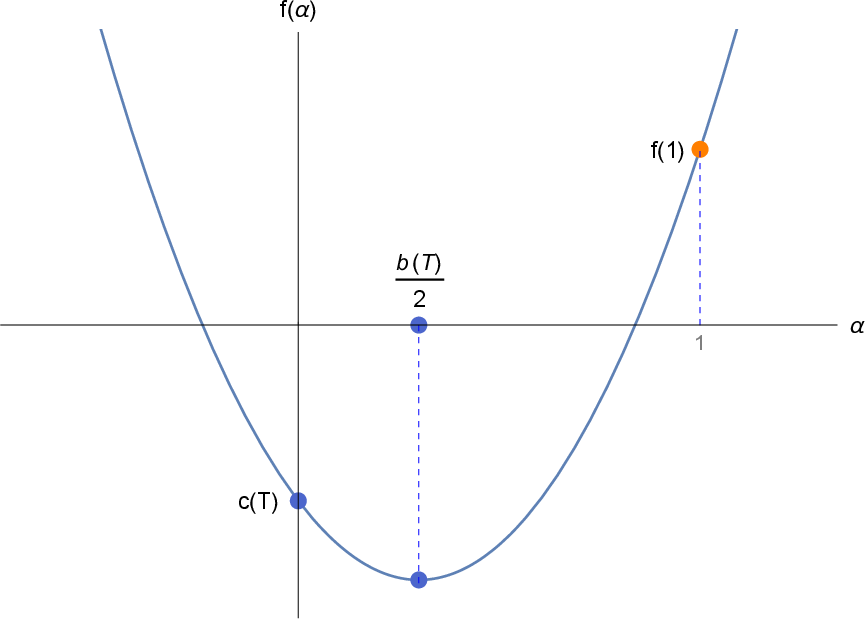}
			\caption{The graph of $f(\alpha)$ against $\alpha$.}
			\label{lqmf.mfg.figure.quadratic_function}
		\end{minipage}
		\begin{minipage}[t]{3,5in}
			\includegraphics[width=3in]{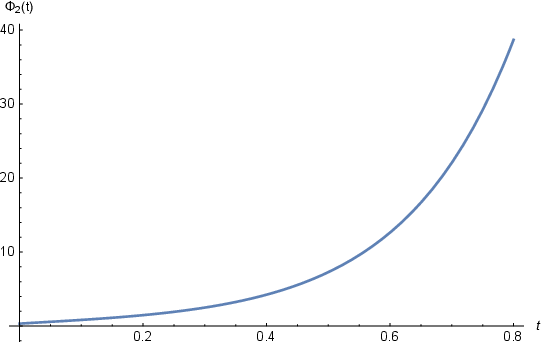}
			\caption{The graph of $\Phi_2$ against $t$ when parameters are chosen as \eqref{lqmf.counterexample.parameter1} and \eqref{lqmf.counterexample.parameter2}.}
			\label{lqmf.fig.Phi2}
		\end{minipage}
	\end{tabular}
\end{figure}


\end{proof}
\begin{remark}
	Comparing with the condition \eqref{lqmf.mfg.condition.local.solution}, the conditions in Theorem \ref{lqmf.mfg.theorem.global_local.solution} has three advantages as follows:
	\begin{enumerate}
	\item The conditions \eqref{lqmf.mfg.condition.refined.global} and \eqref{lqmf.mfg.condition.refined.local} do not require small value of $B_t$ and $Q_t$ as demanded by in Theorem \ref{lqmf.mfg.theorem.local.solution};
	\item The condition \eqref{lqmf.mfg.condition.refined.global} provides the global existence of the solution, and when we set $\bar{A}_t = \bar{B}_t = \bar{Q}_t = \bar{P}_t = \bar{S}_t = \bar{R}_t = 0$, it becomes 
	\begin{align}
	2\vertiii{B P^{-1}N^{\top}}_T^2+ \vertiii{NP^{-1}N^{\top}}_{T}^2 < 1,
	\end{align}
	for instance, the linear-quadratic control problem is uniquely solvable if the cross term $N$ is small enough, which is consistent with the results obtained before in \cite{nisio2015stochastic} and \cite{bensoussan2018estimation};
	\item The condition \eqref{lqmf.mfg.condition.refined.local} relaxes the constraints on the mean field coefficients $\bar{A}_t$, $\bar{B}_t$ and $\bar{P}_t$, but results in a small time duration. Since $T$ appears under a square root symbol, it can be much larger then that required in condition \eqref{lqmf.mfg.condition.local.solution} in Theorem \ref{lqmf.mfg.theorem.local.solution}.
\end{enumerate}	
\end{remark}
\begin{remark}
	The condition \eqref{lqmf.mfg.condition.refined.local} is a generalization of (5) in \cite{bensoussan2016linear}; indeed, if we take $\bar{B}_t = \bar{P}_t = N_t = 0$, \eqref{lqmf.mfg.condition.refined.local} becomes $\vertiii{\bar{A}}_{T_0}^2 \neq 0$, $\vertiii{\bar{\mathcal{Q}}}_{1,T_0}^2<1$, with $\vertiii{\bar{\mathcal{Q}}}_{1,T_0}:=\sup_{0 \leq t \leq T_0} \|Q_t^{-1/2}\bar{\mathcal{Q}}_t Q_t^{-1/2} \|$  and $\vertiii{\bar{A}}_{T_0}:=\sup_{0 \leq t \leq T_0}\| (B_t P_t^{-1}B_t^{\top})^{-\frac{1}{2}}\bar{A}_t Q_{t}^{-1 / 2} \|$,
	\begin{align*}
	T < T_0 \bigwedge \left(\frac{1 - \vertiii{\bar{\mathcal{Q}}}_{1,T_0}}{ \vertiii{\phi}_{0,T_0}\vertiii{\bar{A}}_{T_0}\left(1 + \vertiii{\bar{\mathcal{Q}}}_{1,T_0}\right)}\right)^2.
	\end{align*}
\end{remark}

\subsection{Global Unique Existence with Small Mean Field Sensitivities} \label{lqmf.section.global}

We shall apply the homotopy method first used in \cite{peng1999fully} (also see \cite{bensoussan2017linear}) so as to find another weaker sufficient condition for warranting unique existence of a global solution. This condition does not require the coefficients $\bar{Q}$ and $\bar{P}$ in the objective functional to be small enough or $S$ and $R$ to be very close to $I$ as that demanded in Theorem \ref{lqmf.mfg.theorem.global_local.solution}.

For any symmetric matrix $M_{e,1}$, we denote by $\lambda(M_{e,1})$ as any of its eigenvalues of $M_{e,1}$ with $\lambda_{min}(M_{e,1})$ being the smallest one and $\lambda_{max}(M_{e,1})$ being the largest one. For any matrix (not necessarily symmetric) $M_{e,2}$, we also denote by $\sigma(M_{e,2}) = \sqrt{\lambda(M_{e,2}^{\top}M_{e,2})}$ as any one of singular values of $M_{e,2}$ with $\sigma_{min}(M_{e,2})$ being the smallest one and $\sigma_{max}(M_{e,2})$ being the largest one. We further define:
\begin{align}
\label{lqmf.condition.K1}
K_1 :=& \inf_{t \in [0,T]}\lambda_{min}\left(\frac{Q_t+\bar{\mathcal{Q}}_t  +\left(Q_t+\bar{\mathcal{Q}}_t \right)^{\top}}{2}\right) > 0, \\
\label{lqmf.condition.K2}
K_2 :=& \inf_{t \in [0,T]}\lambda_{min}\left(\frac{(B_t+\bar{B}_t) (P_t+\bar{\mathcal{P}}_t)^{-1}B_t^{\top} + \left((B_t+\bar{B}_t) (P_t+\bar{\mathcal{P}}_t)^{-1}B_t^{\top}\right)^{\top}}{2}\right) > 0,
\end{align}
and
\begin{align*}
K_3 :=& \sup_{t \in [0,T]}\left(\sigma_{max}\left(\bar{A}_t -  (B_t+\bar{B}_t) (P_t+\bar{\mathcal{P}}_t)^{-1}\bar{\mathcal{S}}_t\right)\right)^2, \\
K_4 :=& \sup_{t \in [0,T]}\left(\sigma_{max}\left( \bar{\mathcal{R}}_t(P_t+\bar{\mathcal{P}}_t)^{-1}B_t^{\top} \right)\right)^2, \\
K_5 :=& \sup_{t \in [0,T]}\lambda_{max}\left( \frac{\bar{\mathcal{R}}_t(P_t+\bar{\mathcal{P}}_t)^{-1}\bar{\mathcal{S}}_t + \bar{\mathcal{S}}_t^{\top}(P_t^{\top}+\bar{\mathcal{P}}_t^{\top})^{-1}\bar{\mathcal{R}}_t^{\top}}{2} \right),
\end{align*}
note that all of them are independent of time $T$. We first focus on a result about the uniqueness of the solution.
\begin{theorem} \label{lqmf.mfg.theorem.global.uniqueness}
	(Uniqueness) Under conditions \eqref{lqmf.condition.K1} and \eqref{lqmf.condition.K2} such that
	\begin{align}
	\label{lqmf.condition.Final} &K_3+K_4+2K_2K_5 < 2K_1K_2,
	\end{align}
	the forward-backward equation system \eqref{lqmf.mfg.expected.fbode} has at most one solution.
\end{theorem}

\begin{remark} \label{lqmf.remark.global}
	The condition \eqref{lqmf.condition.Final} can be fulfilled as long as $K_3$, $K_4$ and $K_5$ are small enough, that is to say, the coefficients of mean field term $\bar{A}_t$ and the cross term $N_t$ in $\bar{\mathcal{S}}_t$ and $\bar{\mathcal{R}}_t$ are small enough, while the conditions \eqref{lqmf.condition.K1} and \eqref{lqmf.condition.K2} can hold especially when one has the positive-definiteness of $\bar{\mathcal{P}}_t = \bar{P}_t (I-R_t)$, $\bar{\mathcal{Q}}_t = \bar{Q}_t (I-S_t)$ and $\bar{B}_t (P_t+\bar{\mathcal{P}}_t)^{-1}B_t^{\top}$, for all $t\in[0,T]$. It is a much more relaxed condition compared with Condition \eqref{lqmf.mfg.condition.refined.global} in Theorem \ref{lqmf.mfg.theorem.global_local.solution}, where $\bar{B}$, $\bar{\mathcal{Q}}$ and $\bar{\mathcal{P}}$ are supposed to be small (see \eqref{lqmf.mfg.equalities.notational_implicity}, a small enough $\bar{\mathcal{Q}}$ implies that either $\bar{Q}$ is small enough or $S$ is close enough to $I$). We further give an alternative more tractable sufficient condition for \eqref{lqmf.condition.K1} and \eqref{lqmf.condition.K2} in light of the celebrated Weyl's inequality (see \cite{fulton2000eigenvalues} and \cite{knutson2001honeycombs}).
\end{remark}

\begin{proof}
	We consider the Hilbert space $L^2(0,T;\mathbb{R}^n)$ equipped with the inner product $\braket{z,z'}_T := z_T^{\top}z'_T + \int_0^T z_t^{\top}z'_tdt$.	Let $(\xi,\eta)$ and $(\xi',\eta')$ be two different solutions of \eqref{lqmf.mfg.expected.fbode}, and define $\delta\xi=\xi-\xi'$, $\delta\eta=\eta-\eta'$. It is clear that $(\delta\xi,\delta\eta)$ solves the system:
\begin{align*}
\left\{\begin{array}{rl}
\nonumber
\frac{d}{d t}\left(\begin{array}{c}
\delta\xi_t \\
-\delta\eta_t
\end{array}\right) &=\left(\begin{array}{cc}
A_t+\bar{A}_t -  (B_t+\bar{B}_t) (P_t+\bar{\mathcal{P}}_t)^{-1}\bar{\mathcal{S}}_t  & -(B_t+\bar{B}_t) (P_t+\bar{\mathcal{P}}_t)^{-1}B_t^{\top} \\
Q_t+\bar{\mathcal{Q}}_t - \bar{\mathcal{R}}_t(P_t+\bar{\mathcal{P}}_t)^{-1}\bar{\mathcal{S}}_t & A_t^{\top}-\bar{\mathcal{R}}_t(P_t+\bar{\mathcal{P}}_t)^{-1}B_t^{\top}
\end{array}\right)\left(\begin{array}{l}
\delta\xi_t \\
\delta\eta_t
\end{array}\right), \\
\delta\xi(0) &=0, \quad \delta\eta_T=\left(Q_T+\bar{\mathcal{Q}}_T\right) \delta\xi_T.
\end{array}\right.
\end{align*}
By the chain rule, we have
\begin{align}
\nonumber &\delta\xi_T^{\top}\left(Q_T+\bar{\mathcal{Q}}_T\right) \delta\xi_T + \int_0^T \delta\xi_t^{\top} \left(Q_t+\bar{\mathcal{Q}}_t \right) \delta\xi_t dt + \int_0^T \delta\eta_t^{\top}(B_t+\bar{B}_t) (P_t+\bar{\mathcal{P}}_t)^{-1}B_t^{\top} \delta\eta_t dt \\
\label{lqmf.equality.temp1} &= \int_0^T \delta\eta_t^{\top}\left(\bar{A}_t -  (B_t+\bar{B}_t) (P_t+\bar{\mathcal{P}}_t)^{-1}\bar{\mathcal{S}}_t\right) \delta\xi_t dt + \int_0^T \delta\xi_t^{\top} \bar{\mathcal{R}}_t(P_t+\bar{\mathcal{P}}_t)^{-1}B_t^{\top} \delta\eta_t dt + \int_0^T \delta\xi_t^{\top} \bar{\mathcal{R}}_t(P_t+\bar{\mathcal{P}}_t)^{-1}\bar{\mathcal{S}}_t \delta\xi_t dt.
\end{align}
For the left hand side, it follows from the conditions \eqref{lqmf.condition.K1} and \eqref{lqmf.condition.K2} that
\begin{align*}
\delta\xi_T^{\top}\left(Q_T+\bar{\mathcal{Q}}_T\right) \delta\xi_T + \int_0^T \delta\xi_t^{\top} \left(Q_t+\bar{\mathcal{Q}}_t \right) \delta\xi_t dt + \int_0^T \delta\eta_t^{\top}(B_t+\bar{B}_t) (P_t+\bar{\mathcal{P}}_t)^{-1}B_t^{\top} \delta\eta_t dt \geq K_1 \|\delta\xi\|_T^2 + K_2 \|\delta\eta\|_T^2.
\end{align*}
On the other hand, for the right hand side of \eqref{lqmf.equality.temp1}, recall the definitions of $K_3$, $K_4$ and $K_5$, and by the Young's inequality, we have, for all $\varepsilon > 0$,
\begin{align}
\label{lqmf.mfg.inequality1.homotopy}
\int_0^T \delta\eta_t^{\top}\left(\bar{A}_t -  (B_t+\bar{B}_t) (P_t+\bar{\mathcal{P}}_t)^{-1}\bar{\mathcal{S}}_t\right) \delta\xi_t dt &\leq \frac{K_3}{2 \varepsilon} \|\delta\xi\|_T^2 + \frac{\varepsilon}{2}\|\delta\eta\|_T^2, \\
\label{lqmf.mfg.inequality2.homotopy}
\int_0^T \delta\xi_t^{\top} \bar{\mathcal{R}}_t(P_t+\bar{\mathcal{P}}_t)^{-1}B_t^{\top} \delta\eta_t dt &\leq \frac{K_4}{2 \varepsilon} \|\delta\xi\|_T^2 + \frac{\varepsilon}{2}\|\delta\eta\|_T^2,
\end{align}
and
\begin{align}
\label{lqmf.mfg.inequality3.homotopy}
\int_0^T \delta\xi_t^{\top} \bar{\mathcal{R}}_t(P_t+\bar{\mathcal{P}}_t)^{-1}\bar{\mathcal{S}}_t \delta\xi_t dt \leq K_5 \int_0^T \delta\xi_t^{\top}\delta\xi_t dt \leq K_5 \|\delta\xi\|_T^2.
\end{align}
Combining these estimates of \eqref{lqmf.mfg.inequality1.homotopy}, \eqref{lqmf.mfg.inequality2.homotopy} and \eqref{lqmf.mfg.inequality3.homotopy}, we have
\begin{align*}
K_1 \|\delta\xi\|_T^2 + K_2 \|\delta\eta\|_T^2 \leq \frac{K_3}{2 \varepsilon} \|\delta\xi\|_T^2 + \frac{\varepsilon}{2}\|\delta\eta\|_T^2 + \frac{K_4}{2 \varepsilon} \|\delta\xi\|_T^2 + \frac{\varepsilon}{2}\|\delta\eta\|_T^2 + K_5 \|\delta\xi\|_T^2,
\end{align*}
which further implies
\begin{align*}
\left(K_1 - K_5 - \frac{K_3+K_4}{2 \varepsilon} \right) \|\delta\xi\|_T^2 + \left(K_2-\varepsilon\right)\|\delta\eta\|_T^2 \leq 0.
\end{align*}
To guarantee the uniqueness, it suffices to set $\varepsilon>0$ such that
\begin{align}
\label{lqmf.mfg.epsilon.choose}
0 < \frac{K_3+K_4}{2(K_1-K_5)} < \varepsilon < K_2,
\end{align}
which is ensured by the condition \eqref{lqmf.condition.Final}.
\end{proof}

\begin{theorem} \label{lqmf.mfg.theorem.global.existence}
	Under the conditions \eqref{lqmf.condition.K1}, \eqref{lqmf.condition.K2} and \eqref{lqmf.condition.Final} of Theorem \ref{lqmf.mfg.theorem.global.uniqueness}, the forward-backward equation system \eqref{lqmf.mfg.expected.fbode} admits a unique solution.
\end{theorem}

\begin{proof}
We choose the same $\varepsilon>0$ as in \eqref{lqmf.mfg.epsilon.choose}, and define
\begin{align*}
\beta_1 := K_1 - K_5 -\frac{K_3+K_4}{2 \varepsilon}, \quad \beta_2 := K_2-\frac{\varepsilon}{2}.
\end{align*}
We apply the homotopy method as shown in \cite{peng1999fully} and \cite{bensoussan2017linear}. Consider a family of equation systems indexed by the parameter $\alpha_0 \in [0,1]$ as follows:
\begin{align}
\label{lqmf.mfg.expected.fbode.family_indexed_by_alpha}
\left\{\begin{array}{rcl}
\frac{d \xi_t}{dt}&=& -(1-\alpha_0)\beta_2\eta_t + \alpha_0\left(\left(A_t+\bar{A}_t -  (B_t+\bar{B}_t) (P_t+\bar{\mathcal{P}}_t)^{-1}\bar{\mathcal{S}}_t\right) \xi_t  - (B_t+\bar{B}_t) (P_t+\bar{\mathcal{P}}_t)^{-1}B_t^{\top} \eta_t\right) + \phi_t , \\
-\frac{d \eta_t}{dt} &=& (1-\alpha_0)\beta_1\xi_t + \alpha_0\left(\left(A_t^{\top} -\bar{\mathcal{R}}_t(P_t+\bar{\mathcal{P}}_t)^{-1}B_t^{\top}\right) \eta_t +\left(\left(Q_t+\bar{\mathcal{Q}}_t\right) - \bar{\mathcal{R}}_t(P_t+\bar{\mathcal{P}}_t)^{-1}\bar{\mathcal{S}}_t\right) \xi_t\right) + \psi_t, \\
\xi_0 &=&\mathbb{E}\left[x_0\right], \quad \eta_T =(1-\alpha_0)\beta_1\xi_T + \alpha_0\left(Q_T+\bar{\mathcal{Q}}_T\right) \xi_T,
\end{array}\right.
\end{align}
where $\phi_{\cdot}$ and $\psi_{\cdot}$ are two fixed deterministic functions belonging to $L^2(0,T;\mathbb{R}^n)$. Our objective is to prove the existence of the solution of \eqref{lqmf.mfg.expected.fbode.family_indexed_by_alpha} for $\alpha_0 = 1$. It is clear that when $\alpha_0 = 0$ the existence is immediate, so we can apply the inductive argument to extend the existence result up to $\alpha_0 = 1$ by assuming, for a fixed parameter $0 \leq \alpha_0 < 1$, the existence is true and then checking whether the existence still holds for a slightly larger parameter $\alpha_0 + \delta_0$ for some of a uniform magnitude $\delta_0 >0$. We consider a map $\mathcal{S}: (z,q) \mapsto (\xi,\eta)$ such that
\begin{align}
\label{lqmf.mfg.expected.fbode.mapS}
\left\{\begin{array}{rcl}
\frac{d \xi_t}{dt}&=& -(1-\alpha_0)\beta_2\eta_t + \alpha_0\left(\left(A_t+\bar{A}_t -  (B_t+\bar{B}_t) (P_t+\bar{\mathcal{P}}_t)^{-1}\bar{\mathcal{S}}_t\right) \xi_t  - (B_t+\bar{B}_t) (P_t+\bar{\mathcal{P}}_t)^{-1}B_t^{\top} \eta_t\right) + \phi_t \\
& &+ \delta_0\left(\beta_2q_t + \left(A_t+\bar{A}_t -  (B_t+\bar{B}_t) (P_t+\bar{\mathcal{P}}_t)^{-1}\bar{\mathcal{S}}_t\right) z_t  - (B_t+\bar{B}_t) (P_t+\bar{\mathcal{P}}_t)^{-1}B_t^{\top} q_t\right), \\
-\frac{d \eta_t}{dt} &=& (1-\alpha_0)\beta_1\xi_t + \alpha_0\left(\left(A_t^{\top} -\bar{\mathcal{R}}_t(P_t+\bar{\mathcal{P}}_t)^{-1}B_t^{\top}\right) \eta_t +\left(\left(Q_t+\bar{\mathcal{Q}}_t\right) - \bar{\mathcal{R}}_t(P_t+\bar{\mathcal{P}}_t)^{-1}\bar{\mathcal{S}}_t\right) \xi_t\right) + \psi_t \\
& &-\delta_0\left(\beta_1z_t -\left(A_t^{\top} -\bar{\mathcal{R}}_t(P_t+\bar{\mathcal{P}}_t)^{-1}B_t^{\top}\right) q_t -\left(\left(Q_t+\bar{\mathcal{Q}}_t\right) - \bar{\mathcal{R}}_t(P_t+\bar{\mathcal{P}}_t)^{-1}\bar{\mathcal{S}}_t\right) z_t \right), \\
\xi_0 &=&\mathbb{E}\left[x_0\right], \quad \eta_T =(1-\alpha_0)\beta_1\xi_T + \alpha_0\left(Q_T+\bar{\mathcal{Q}}_T\right) \xi_T - \delta_0\left(\beta_1z_T-\left(Q_T+\bar{\mathcal{Q}}_T\right) z_T\right),
\end{array}\right.
\end{align}
especially if the map $\mathcal{S}$ admits a fixed point for a properly chosen $\delta$ of uniform size independent of the value of $\alpha_0$ at all, then the result follows by induction. We choose two pairs $\mathcal{S}(z,q)=(\xi,\eta)$ and $\mathcal{S}(z',q')=(\xi',\eta')$ and take the difference of them. The linearity of \eqref{lqmf.mfg.expected.fbode.mapS} implies the differences $\delta z:=z-z'$, $\delta q:=q-q'$, $\delta\xi:=\xi-\xi'$ and $\delta\eta:=\eta-\eta'$ satisfy
\begin{align*}
\left\{\begin{array}{ccl}
\frac{d \delta\xi_t}{dt}&=& -(1-\alpha_0)\beta_2\delta\eta_t + \alpha_0\left(\left(A_t+\bar{A}_t -  (B_t+\bar{B}_t) (P_t+\bar{\mathcal{P}}_t)^{-1}\bar{\mathcal{S}}_t\right) \delta\xi_t  - (B_t+\bar{B}_t) (P_t+\bar{\mathcal{P}}_t)^{-1}B_t^{\top} \delta\eta_t\right) \\
& &+ \delta_0\left(\beta_2\delta q_t + \left(A_t+\bar{A}_t -  (B_t+\bar{B}_t) (P_t+\bar{\mathcal{P}}_t)^{-1}\bar{\mathcal{S}}_t\right) \delta z_t  - (B_t+\bar{B}_t) (P_t+\bar{\mathcal{P}}_t)^{-1}B_t^{\top} \delta q_t\right), \\
-\frac{d \delta\eta_t}{dt} &=& (1-\alpha_0)\beta_1\delta\xi_t + \alpha_0\left(\left(A_t^{\top} -\bar{\mathcal{R}}_t(P_t+\bar{\mathcal{P}}_t)^{-1}B_t^{\top}\right) \delta\eta_t +\left(\left(Q_t+\bar{\mathcal{Q}}_t\right) - \bar{\mathcal{R}}_t(P_t+\bar{\mathcal{P}}_t)^{-1}\bar{\mathcal{S}}_t\right) \delta\xi_t\right) \\
& &-\delta_0\left(\beta_1\delta z_t -\left(A_t^{\top} -\bar{\mathcal{R}}_t(P_t+\bar{\mathcal{P}}_t)^{-1}B_t^{\top}\right) \delta q_t -\left(\left(Q_t+\bar{\mathcal{Q}}_t\right) - \bar{\mathcal{R}}_t(P_t+\bar{\mathcal{P}}_t)^{-1}\bar{\mathcal{S}}_t\right) \delta z_t \right), \\
\delta\xi_0 &=&0, \quad \delta\eta_T =(1-\alpha_0)\beta_1\delta\xi_T + \alpha_0\left(Q_T+\bar{\mathcal{Q}}_T\right) \delta\xi_T - \delta_0\left(\beta_1\delta z_T-\left(Q_T+\bar{\mathcal{Q}}_T\right) \delta z_T\right).
\end{array}\right.
\end{align*}
Applying the chain rule to $\delta\xi_t^{\top}\delta\eta_t$, we obtain
\begin{align}
\nonumber
& (1-\alpha_0)\beta_1\delta\xi_T^{\top}\delta\xi_T + \alpha_0\delta\xi_T^{\top}\left(Q_T+\bar{\mathcal{Q}}_T\right) \delta\xi_T = \delta\left(\beta_1\delta\xi_T^{\top}\delta z_T-\delta\xi_T^{\top}\left(Q_T+\bar{\mathcal{Q}}_T\right) \delta z_T\right) + \int_0^T d\delta\xi_t^{\top} \cdot \delta\eta_t + \int_0^T \delta\xi_t^{\top}d\delta\eta_t \\
\nonumber
&= \delta_0\left(\beta_1\delta\xi_T^{\top}\delta z_T-\delta\xi_T^{\top}\left(Q_T+\bar{\mathcal{Q}}_T\right) \delta z_T\right) \\
\nonumber
&\quad + \int_0^T \bigg\{-(1-\alpha_0)\beta_2\delta\eta_t + \alpha_0\left(\left(A_t+\bar{A}_t -  (B_t+\bar{B}_t) (P_t+\bar{\mathcal{P}}_t)^{-1}\bar{\mathcal{S}}_t\right) \delta\xi_t  - (B_t+\bar{B}_t) (P_t+\bar{\mathcal{P}}_t)^{-1}B_t^{\top} \delta\eta_t\right) \\
\nonumber
&\quad + \delta_0\left(\beta_2\delta q_t + \left(A_t+\bar{A}_t -  (B_t+\bar{B}_t) (P_t+\bar{\mathcal{P}}_t)^{-1}\bar{\mathcal{S}}_t\right) \delta z_t  - (B_t+\bar{B}_t) (P_t+\bar{\mathcal{P}}_t)^{-1}B_t^{\top} \delta q_t\right)\bigg\} \cdot \delta\eta_t dt \\
\nonumber
&\quad - \int_0^T \delta\xi_t^{\top}\bigg\{(1-\alpha_0)\beta_1\delta\xi_t + \alpha_0\left(\left(A_t^{\top} -\bar{\mathcal{R}}_t(P_t+\bar{\mathcal{P}}_t)^{-1}B_t^{\top}\right) \delta\eta_t +\left(\left(Q_t+\bar{\mathcal{Q}}_t\right) - \bar{\mathcal{R}}_t(P_t+\bar{\mathcal{P}}_t)^{-1}\bar{\mathcal{S}}_t\right) \delta\xi_t\right) \\
\nonumber
&\quad -\delta_0\left(\beta_1\delta z_t -\left(A_t^{\top} -\bar{\mathcal{R}}_t(P_t+\bar{\mathcal{P}}_t)^{-1}B_t^{\top}\right) \delta q_t -\left(\left(Q_t+\bar{\mathcal{Q}}_t\right) - \bar{\mathcal{R}}_t(P_t+\bar{\mathcal{P}}_t)^{-1}\bar{\mathcal{S}}_t\right) \delta z_t \right)\bigg\}dt \\
\nonumber
&= \delta_0\left(\beta_1\delta\xi_T^{\top}\delta z_T-\delta\xi_T^{\top}\left(Q_T+\bar{\mathcal{Q}}_T\right) \delta z_T\right) \\
\nonumber
&\quad -\alpha_0 \int_0^T \left(\delta\eta_t^{\top}(B_t+\bar{B}_t) (P_t+\bar{\mathcal{P}}_t)^{-1}B_t^{\top} \delta\eta_t + \delta\xi_t^{\top} \left(Q_t+\bar{\mathcal{Q}}_t\right)  \delta\xi_t\right) dt \\
\nonumber
&\quad +\alpha_0 \int_0^T \left(\delta\eta_t^{\top} \left(\bar{A}_t-(B_t+\bar{B}_t) (P_t+\bar{\mathcal{P}}_t)^{-1}\bar{\mathcal{S}}_t\right) \delta\xi_t + \delta\xi_t^{\top} \bar{\mathcal{R}}_t(P_t+\bar{\mathcal{P}}_t)^{-1}B_t^{\top} \delta\eta_t + \delta\xi_t^{\top}\bar{\mathcal{R}}_t(P_t+\bar{\mathcal{P}}_t)^{-1}\bar{\mathcal{S}}_t\delta\xi_t\right) dt \\
\nonumber
&\quad -(1-\alpha_0) \int_0^T \left(\beta_1\delta\xi_t^{\top}\delta\xi_t +\beta_2\delta\eta_t^{\top}\delta\eta_t\right) dt \\
\nonumber
&\quad +\delta_0 \int_0^T \bigg(\beta_1\delta\xi_t^{\top}\delta z_t+\beta_2\delta\eta_t^{\top}\delta q_t+\delta\eta_t^{\top}\left(A_t+\bar{A}_t -  (B_t+\bar{B}_t) (P_t+\bar{\mathcal{P}}_t)^{-1}\bar{\mathcal{S}}_t\right) \delta z_t-\delta\eta_t^{\top}(B_t+\bar{B}_t) (P_t+\bar{\mathcal{P}}_t)^{-1}B_t^{\top} \delta q_t \\
\label{lqmf.mfg.equation.xi_eta.delta}
&\quad -\delta\xi_t^{\top}\left(A_t^{\top} -\bar{\mathcal{R}}_t(P_t+\bar{\mathcal{P}}_t)^{-1}B_t^{\top}\right) \delta q_t + \delta\xi_t^{\top}\left(\left(Q_t+\bar{\mathcal{Q}}_t\right) - \bar{\mathcal{R}}_t(P_t+\bar{\mathcal{P}}_t)^{-1}\bar{\mathcal{S}}_t\right) \delta z_t\bigg) dt.
\end{align}
To estimate \eqref{lqmf.mfg.equation.xi_eta.delta}, we shall first introduce some results that as follows: first, by conditions \eqref{lqmf.condition.K1} and \eqref{lqmf.condition.K2}, we have
\begin{align}
\label{lqmf.global_existence.inequaility.2}
\int_0^T \left(\delta\eta_t^{\top}(B_t+\bar{B}_t) (P_t+\bar{\mathcal{P}}_t)^{-1}B_t^{\top} \delta\eta_t + \delta\xi_t^{\top} \left(Q_t+\bar{\mathcal{Q}}_t\right)  \delta\xi_t\right) dt \geq \int_0^T \left(K_1\delta\xi_t^{\top}\delta\xi_t +K_2\delta\eta_t^{\top}\delta\eta_t\right) dt.
\end{align}
Besides, by using the same estimates \eqref{lqmf.mfg.inequality1.homotopy}, \eqref{lqmf.mfg.inequality2.homotopy} and \eqref{lqmf.mfg.inequality3.homotopy} as in the proof of uniqueness, we obtain
\begin{align}
\nonumber &\int_0^T \left(\delta\eta_t^{\top} \left(\bar{A}_t-(B_t+\bar{B}_t) (P_t+\bar{\mathcal{P}}_t)^{-1}\bar{\mathcal{S}}_t\right) \delta\xi_t + \delta\xi_t^{\top} \bar{\mathcal{R}}_t(P_t+\bar{\mathcal{P}}_t)^{-1}B_t^{\top} \delta\eta_t + \delta\xi_t^{\top}\bar{\mathcal{R}}_t(P_t+\bar{\mathcal{P}}_t)^{-1}\bar{\mathcal{S}}_t\delta\xi_t\right) dt \\
\label{lqmf.global_existence.inequaility.3} \leq& \left(\frac{K_3+K_4}{2 \varepsilon} + K_5\right)\int_0^T \delta\xi_t^{\top}\delta\xi_t dt + \varepsilon\int_0^T \delta\eta_t^{\top}\delta\eta_t dt.
\end{align}
Hence, combining \eqref{lqmf.global_existence.inequaility.2} and \eqref{lqmf.global_existence.inequaility.3} and noticing the definitions of $\beta_1$ and $\beta_2$, we have
\begin{align}
\nonumber &-\alpha_0 \int_0^T \left(\delta\eta_t^{\top}(B_t+\bar{B}_t) (P_t+\bar{\mathcal{P}}_t)^{-1}B_t^{\top} \delta\eta_t + \delta\xi_t^{\top} \left(Q_t+\bar{\mathcal{Q}}_t\right)  \delta\xi_t\right) dt \\
\nonumber &+\alpha_0 \int_0^T \left(\delta\eta_t^{\top} \left(\bar{A}_t-(B_t+\bar{B}_t) (P_t+\bar{\mathcal{P}}_t)^{-1}\bar{\mathcal{S}}_t\right) \delta\xi_t + \delta\xi_t^{\top} \bar{\mathcal{R}}_t(P_t+\bar{\mathcal{P}}_t)^{-1}B_t^{\top}\delta \eta_t + \delta\xi_t^{\top}\bar{\mathcal{R}}_t(P_t+\bar{\mathcal{P}}_t)^{-1}\bar{\mathcal{S}}_t\delta\xi_t\right) dt \\
\nonumber &\leq -\alpha_0 \left(K_1 - K_5 - \frac{K_3+K_4}{2 \varepsilon}\right) \int_0^T \delta\xi_t^{\top}\delta\xi_t dt -\alpha_0 \left(K_2-\varepsilon\right)\int_0^T \delta\eta_t^{\top}\delta\eta_t dt \\
\label{lqmf.global_existence.inequaility.4} &= -\alpha_0 \beta_1 \int_0^T \delta\xi_t^{\top}\delta\xi_t dt -\alpha_0 \beta_2\int_0^T \delta\eta_t^{\top}\delta\eta_t dt.
\end{align}
In addition, by conditions \eqref{lqmf.condition.K1} and definition of $\beta_1$, we also have 
\begin{align}
	\label{lqmf.global_existence.inequaility.1}
	\delta\xi_T^{\top}\left(Q_T+\bar{\mathcal{Q}}_T\right) \delta\xi_T \geq K_1\delta\xi_T^{\top}\delta\xi_T \geq \beta_1\delta\xi_T^{\top}\delta\xi_T.
\end{align}
Therefore, substituting \eqref{lqmf.global_existence.inequaility.4} and \eqref{lqmf.global_existence.inequaility.1} into \eqref{lqmf.mfg.equation.xi_eta.delta}, we obtain
\begin{align*}
&\int_0^T \left(\beta_1\delta\xi_t^{\top}\delta\xi_t +\beta_2\delta\eta_t^{\top}\delta\eta_t\right) dt + \beta_1\delta\xi_T^{\top}\delta\xi_T \leq \delta_0\left(\beta_1\delta\xi_T^{\top}\delta z_T-\delta\xi_T^{\top}\left(Q_T+\bar{\mathcal{Q}}_T\right) \delta z_T\right) \\
&+\delta_0 \int_0^T \bigg(\beta_1\delta\xi_t^{\top}\delta z_t+\beta_2\delta\eta_t^{\top}\delta q_t+\delta\eta_t^{\top}\left(A_t+\bar{A}_t -  (B_t+\bar{B}_t) (P_t+\bar{\mathcal{P}}_t)^{-1}\bar{\mathcal{S}}_t\right) \delta z_t-\delta\eta_t^{\top}(B_t+\bar{B}_t) (P_t+\bar{\mathcal{P}}_t)^{-1}B_t^{\top} \delta q_t \\
&-\delta\xi_t^{\top}\left(A_t^{\top} -\bar{\mathcal{R}}_t(P_t+\bar{\mathcal{P}}_t)^{-1}B_t^{\top}\right) \delta q_t + \delta\xi_t^{\top}\left(\left(Q_t+\bar{\mathcal{Q}}_t\right) - \bar{\mathcal{R}}_t(P_t+\bar{\mathcal{P}}_t)^{-1}\bar{\mathcal{S}}_t\right) \delta z_t\bigg) dt \\
&\leq\delta_0 \mathcal{C}_0 \int_0^T (\delta\xi_t^{\top}\delta\xi_t + \delta\eta_t^{\top}\delta\eta_t) dt + \delta_0 \mathcal{C}_0 \delta\xi_T^{\top}\delta\xi_T + \delta_0 \mathcal{C}_0 \int_0^T (\delta z_t^{\top}\delta z_t + \delta q_t^{\top}\delta q_t) dt + \delta_0 \mathcal{C}_0 \delta z_T^{\top}\delta z_T,
\end{align*}
where 
\begin{align*}
\mathcal{C}_0 :=& \frac{1}{2}\left(1+\beta_1+\beta_2+\sup_{0\leq t\leq T}\|A_t+\bar{A}_t -  (B_t+\bar{B}_t) (P_t+\bar{\mathcal{P}}_t)^{-1}\bar{\mathcal{S}}_t\|+\sup_{0\leq t\leq T}\|(B_t+\bar{B}_t) (P_t+\bar{\mathcal{P}}_t)^{-1}B_t^{\top}\|\right) \\
&+\frac{1}{2}\left(\sup_{0\leq t\leq T}\|A_t^{\top} -\bar{\mathcal{R}}_t(P_t+\bar{\mathcal{P}}_t)^{-1}B_t^{\top}\|+\sup_{0\leq t\leq T}\|\left(Q_t+\bar{\mathcal{Q}}_t\right) - \bar{\mathcal{R}}_t(P_t+\bar{\mathcal{P}}_t)^{-1}\bar{\mathcal{S}}_t\|\right).
\end{align*}
We take $\mathcal{C}_1:=\frac{\mathcal{C}_0}{\min\{\beta_1,\beta_2\}}$, then
\begin{align*}
\int_0^T \left(\delta\xi_t^{\top}\delta\xi_t +\delta\eta_t^{\top}\delta\eta_t\right) dt + \delta\xi_T^{\top}\delta\xi_T \leq& \delta_0 \mathcal{C}_1 \int_0^T (\delta\xi_t^{\top}\delta\xi_t + \delta\eta_t^{\top}\delta\eta_t) dt + \delta_0 \mathcal{C}_1 \delta\xi_T^{\top}\delta\xi_T + \delta_0 \mathcal{C}_1 \int_0^T (\delta z_t^{\top}\delta z_t + \delta q_t^{\top}\delta q_t) dt + \delta_0 \mathcal{C}_1 \delta z_T^{\top}\delta z_T.
\end{align*}
Take $\delta_0 := \frac{1}{3 \mathcal{C}_1}$, then
\begin{align*}
\int_0^T \left(\delta\xi_t^{\top}\delta\xi_t +\delta\eta_t^{\top}\delta\eta_t\right) dt + \delta\xi_T^{\top}\delta\xi_T \leq  \frac{1}{2}\left(\int_0^T (\delta z_t^{\top}\delta z_t + \delta q_t^{\top}\delta q_t) dt + \delta z_T^{\top}\delta z_T\right).
\end{align*}
This proves that $\mathcal{S}$ is a contractive map on $L^2(0,T;\mathbb{R}^n)$, thus it admits a unique fixed point. By mathematical inductive approach, when we chose $\alpha_0=1$, the existence of solutions of \eqref{lqmf.mfg.expected.fbode.family_indexed_by_alpha} established. We then combine Theorem \ref{lqmf.mfg.theorem.global.uniqueness} to conclude the proof. Note that due to Theorem \ref{lqmf.mfg.theorem.global.uniqueness}, no matter what the small choice of $\delta_0$ is, the same solution is still constructed.
\end{proof}

\begin{theorem}
	For $m \geq n$ and, for all $t \in [0,T]$, $\bar{B}_t$, $S_t$ and $R_t$ satisfy
	\begin{align}
        \label{lqmf.mfg.condition.weyl1}
	&\sigma_{max}(\bar{B}_t) < \frac{\sigma_{min}(B_t)^2\sigma_{min}(P_t+\bar{P}_t)}{\sigma_{max}(B_t)\sigma_{max}(P_t+\bar{P}_t)}, \quad \sigma_{max}\left(S_t \right) < \frac{\lambda_{min}\left(Q_t\right) + \lambda_{min}\left(\bar{Q}_t\right)}{\lambda_{max}\left(\bar{Q}_t \right)}, \\
        \label{lqmf.mfg.condition.weyl2}
	\text{and} \quad &\sigma_{max}(R_t) < \frac{\sigma_{min}(P_t+\bar{P}_t)}{\sigma_{max}(\bar{P}_t)}\left(\frac{\sigma_{min}(B_t)^2\sigma_{min}(P_t+\bar{P}_t) - \sigma_{max}(B_t)\sigma_{max}(\bar{B}_t)\sigma_{max}(P_t+\bar{P}_t)}{\sigma_{min}(B_t)^2\sigma_{min}(P_t+\bar{P}_t) + \sigma_{max}(B_t)^2\sigma_{max}(P_t+\bar{P}_t)}\bigwedge 1\right),
	\end{align}
	then the conditions \eqref{lqmf.condition.K1} and \eqref{lqmf.condition.K2} are satisfied.
\end{theorem}

\begin{remark}
	This theorem provides a more tractable condition which implies that the small mean field term coefficients $\bar{B}_t$, $S_t$ and $R_t$ can replace the positive-definitiveness conditions \eqref{lqmf.condition.K1} and \eqref{lqmf.condition.K2}. Recall in Remark \ref{lqmf.remark.global}, the condition \ref{lqmf.condition.Final} holds as long as the mean field term $\bar{A}_t$ and the cross term $N_t$ are small. We can conclude that the global unique existence of the solution to the extended mean field game problem \ref{lqmf.mfg.objective} can be warranted if the mean field term coefficients $\bar{A}_t$, $\bar{B}_t$, $S_t$, $R_t$ and the cross term $N_t$ are small.
\end{remark}

\begin{proof}

By applying the Weyl's inequalities (as recalled in Appendix \ref{lqmf.section.appendixA0}), the condition \eqref{lqmf.condition.K1} becomes:
\begin{align*}
&\lambda_{min}\left(\frac{Q_t+\bar{\mathcal{Q}}_t +\left(Q_t+\bar{\mathcal{Q}}_t\right)^{\top}}{2}\right)  \vspace{1ex}\\
&\geq \frac{1}{2}\lambda_{min}\left(Q_t+Q_t^{\top}\right) + \frac{1}{2}\lambda_{min}\left(\bar{\mathcal{Q}}_t+\bar{\mathcal{Q}}_t^{\top}\right) &\text{(by \eqref{lqmf.mfg.inequality2.weyl})} \vspace{1ex}\\
&\geq \lambda_{min}\left(Q_t\right) + \frac{1}{2}\lambda_{min}\left(\bar{Q}_t+\bar{Q}_t^{\top}-\bar{Q}_t S_t-S_t^{\top}\bar{Q}_t^{\top} \right) \vspace{1ex}\\
&\geq \lambda_{min}\left(Q_t\right) + \frac{1}{2}\lambda_{min}\left(\bar{Q}_t+\bar{Q}_t^{\top}\right)-\frac{1}{2}\lambda_{max}\left(\bar{Q}_t S_t+S_t^{\top}\bar{Q}_t^{\top} \right)  \vspace{1ex}\\
&\geq \lambda_{min}\left(Q_t\right) + \lambda_{min}\left(\bar{Q}_t\right) - \frac{1}{2}\sigma_{max}\left(\bar{Q}_t S_t+S_t^{\top}\bar{Q}_t^{\top} \right)  \vspace{1ex}\\
&\geq \lambda_{min}\left(Q_t\right) + \lambda_{min}\left(\bar{Q}_t\right) - \sigma_{max}\left(\bar{Q}_tS_t \right) &\text{(by \eqref{lqmf.mfg.inequality5.weyl})} \vspace{1ex}\\
&\geq \lambda_{min}\left(Q_t\right) + \lambda_{min}\left(\bar{Q}_t\right) - \sigma_{max}\left(S_t \right)\sigma_{max}\left(\bar{Q}_t \right). &\text{(by \eqref{lqmf.mfg.inequality4.weyl})}
\end{align*}
Therefore, it suffices to impose the first inequality of \eqref{lqmf.mfg.condition.weyl1} to ensure the condition \eqref{lqmf.condition.K1}. Besides, consider the condition \eqref{lqmf.condition.K2}, by applying the Weyl's inequalities, we have
{\small
\begin{align*}
&\lambda_{min}\left(\frac{(B_t+\bar{B}_t) (P_t+\bar{\mathcal{P}}_t)^{-1}B_t^{\top} + \left((B_t+\bar{B}_t) (P_t+\bar{\mathcal{P}}_t)^{-1}B_t^{\top}\right)^{\top}}{2}\right) \\
&= \lambda_{min}\left(\frac{B_t (P_t+\bar{\mathcal{P}}_t)^{-1}B_t^{\top} + \left(B_t (P_t+\bar{\mathcal{P}}_t)^{-1}B_t^{\top}\right)^{\top} + \bar{B}_t (P_t+\bar{\mathcal{P}}_t)^{-1}B_t^{\top} + \left(\bar{B}_t (P_t+\bar{\mathcal{P}}_t)^{-1}B_t^{\top}\right)^{\top}}{2}\right) \\
&= \frac{1}{2}\bigg(\lambda_{min}\bigg(B_t (P_t+\bar{P}_t)^{-1}B_t^{\top} + \left(B_t (P_t+\bar{P}_t)^{-1}B_t^{\top}\right)^{\top} + \bar{B}_t (P_t+\bar{P}_t)^{-1}B_t^{\top} + \left(\bar{B}_t (P_t+\bar{P}_t)^{-1}B_t^{\top}\right)^{\top} \\
&\quad+ B_t ((P_t+\bar{\mathcal{P}}_t)^{-1}-(P_t+\bar{P}_t)^{-1})B_t^{\top} + \left(B_t ((P_t+\bar{\mathcal{P}}_t)^{-1}-(P_t+\bar{P}_t)^{-1})B_t^{\top}\right)^{\top} \\
&\quad+ \bar{B}_t ((P_t+\bar{\mathcal{P}}_t)^{-1}-(P_t+\bar{P}_t)^{-1})B_t^{\top} + \left(\bar{B}_t ((P_t+\bar{\mathcal{P}}_t)^{-1}-(P_t+\bar{P}_t)^{-1})B_t^{\top}\right)^{\top} \bigg)\bigg) \\
&\geq \frac{1}{2}\bigg(2\lambda_{min}\bigg(B_t (P_t+\bar{P}_t)^{-1}B_t^{\top}\bigg) + \lambda_{min}\bigg(\bar{B}_t (P_t+\bar{P}_t)^{-1}B_t^{\top} + \left(\bar{B}_t (P_t+\bar{P}_t)^{-1}B_t^{\top}\right)^{\top}\bigg) \\
&\quad+ \lambda_{min}\bigg(B_t ((P_t+\bar{\mathcal{P}}_t)^{-1}-(P_t+\bar{P}_t)^{-1})B_t^{\top} + \left(B_t ((P_t+\bar{\mathcal{P}}_t)^{-1}-(P_t+\bar{P}_t)^{-1})B_t^{\top}\right)^{\top}\bigg) \\
&\quad+ \lambda_{min}\bigg(\bar{B}_t ((P_t+\bar{\mathcal{P}}_t)^{-1}-(P_t+\bar{P}_t)^{-1})B_t^{\top} + \left(\bar{B}_t ((P_t+\bar{\mathcal{P}}_t)^{-1}-(P_t+\bar{P}_t)^{-1})B_t^{\top}\right)^{\top} \bigg)\bigg) &\text{(by \eqref{lqmf.mfg.inequality2.weyl})} \vspace{1ex}\\
&\geq \frac{1}{2}\bigg(2\lambda_{min}\bigg(B_t (P_t+\bar{P}_t)^{-1}B_t^{\top}\bigg) - \sigma_{max}\bigg(\bar{B}_t (P_t+\bar{P}_t)^{-1}B_t^{\top} + \left(\bar{B}_t (P_t+\bar{P}_t)^{-1}B_t^{\top}\right)^{\top}\bigg) \\
&\quad- \sigma_{max}\bigg(B_t ((P_t+\bar{\mathcal{P}}_t)^{-1}-(P_t+\bar{P}_t)^{-1})B_t^{\top} + \left(B_t ((P_t+\bar{\mathcal{P}}_t)^{-1}-(P_t+\bar{P}_t)^{-1})B_t^{\top}\right)^{\top}\bigg) \\
&\quad- \sigma_{max}\bigg(\bar{B}_t ((P_t+\bar{\mathcal{P}}_t)^{-1}-(P_t+\bar{P}_t)^{-1})B_t^{\top} + \left(\bar{B}_t ((P_t+\bar{\mathcal{P}}_t)^{-1}-(P_t+\bar{P}_t)^{-1})B_t^{\top}\right)^{\top} \bigg)\bigg) \\
&\geq \lambda_{min}\bigg(B_t (P_t+\bar{P}_t)^{-1}B_t^{\top}\bigg) - \sigma_{max}\bigg(\bar{B}_t (P_t+\bar{P}_t)^{-1}B_t^{\top}\bigg) \\
&\quad- \sigma_{max}\bigg(B_t ((P_t+\bar{\mathcal{P}}_t)^{-1}-(P_t+\bar{P}_t)^{-1})B_t^{\top}\bigg) - \sigma_{max}\bigg(\bar{B}_t ((P_t+\bar{\mathcal{P}}_t)^{-1}-(P_t+\bar{P}_t)^{-1})B_t^{\top}\bigg) &\text{(by \eqref{lqmf.mfg.inequality5.weyl})} \vspace{1ex}\\
&= \lambda_{min}\bigg(B_t (P_t+\bar{P}_t)^{-1}B_t^{\top}\bigg) - \sigma_{max}\bigg(\bar{B}_t (P_t+\bar{P}_t)^{-1}B_t^{\top}\bigg) \\
&\quad- \sigma_{max}\bigg(B_t (P_t+\bar{\mathcal{P}}_t)^{-1}\bar{P}_tR_t(P_t+\bar{P}_t)^{-1}B_t^{\top}\bigg) - \sigma_{max}\bigg(\bar{B}_t (P_t+\bar{\mathcal{P}}_t)^{-1}\bar{P}_tR_t(P_t+\bar{P}_t)^{-1}B_t^{\top}\bigg) \\
&\geq \frac{\sigma_{min}(B_t)^2}{\sigma_{max}(P_t+\bar{P}_t)} - \frac{\sigma_{max}(B_t)\cdot\sigma_{max}(\bar{B}_t)}{\sigma_{min}(P_t+\bar{P}_t)} \\
&\quad- \bigg(\sigma_{max}(B_t)^2 + \sigma_{max}(B_t)\cdot\sigma_{max}(\bar{B}_t)\bigg)\frac{\sigma_{max}(\bar{P}_t)\cdot\sigma_{max}(R_t)}{\sigma_{min}(P_t+\bar{\mathcal{P}}_t)\sigma_{min}(P_t+\bar{P}_t)}. &\text{(by \eqref{lqmf.mfg.inequality3.weyl} and \eqref{lqmf.mfg.inequality4.weyl})} \vspace{1ex}
\end{align*}
}

Therefore, to ensure the fulfillment of the condition \eqref{lqmf.condition.K2}, we only need both
\begin{align}
\label{lqmf.mfg.condition.weyl3}
\frac{\sigma_{min}(B_t)^2}{\sigma_{max}(P_t+\bar{P}_t)} - \frac{\sigma_{max}(B_t)\cdot\sigma_{max}(\bar{B}_t)}{\sigma_{min}(P_t+\bar{P}_t)} > 0,
\end{align}
and
\begin{align}
\label{lqmf.mfg.condition.weyl4}
\frac{\sigma_{max}(\bar{P}_t)\cdot\sigma_{max}(R_t)}{\sigma_{min}(P_t+\bar{\mathcal{P}}_t)\cdot\sigma_{min}(P_t+\bar{P}_t)} < \frac{\frac{\sigma_{min}(B_t)^2}{\sigma_{max}(P_t+\bar{P}_t)} - \frac{\sigma_{max}(B_t)\cdot\sigma_{max}(\bar{B}_t)}{\sigma_{min}(P_t+\bar{P}_t)}}{\sigma_{max}(B_t)^2 + \sigma_{max}(B_t)\cdot\sigma_{max}(\bar{B}_t)}.
\end{align}
Note that \eqref{lqmf.mfg.condition.weyl3} is indeed the second inequality of the condition \eqref{lqmf.mfg.condition.weyl1}. As for the inequality \eqref{lqmf.mfg.condition.weyl4}, by multiplying $\sigma_{min}(P_t+\bar{P}_t)$ on both sides, we have
\begin{align}
\label{lqmf.mfg.condition.weyl8}
\frac{\sigma_{max}(\bar{P}_t)\cdot\sigma_{max}(R_t)}{\sigma_{min}(P_t+\bar{\mathcal{P}}_t)} < \frac{\sigma_{min}(B_t)^2\frac{\sigma_{min}(P_t+\bar{P}_t)}{\sigma_{max}(P_t+\bar{P}_t)} - \sigma_{max}(B_t)\cdot\sigma_{max}(\bar{B}_t)}{\sigma_{max}(B_t)^2 + \sigma_{max}(B_t)\cdot\sigma_{max}(\bar{B}_t)}.
\end{align}
Since
\begin{align*}
\sigma_{min}(P_t+\bar{\mathcal{P}}_t) \geq \sigma_{min}(P_t+\bar{P}_t - \bar{P}_tR_t) \geq \sigma_{min}(P_t+\bar{P}_t) - \sigma_{max}(\bar{P}_t)\cdot\sigma_{max}(R_t), \quad \text{(by \eqref{lqmf.mfg.inequality4.weyl} and \eqref{lqmf.mfg.inequality5.weyl})}
\end{align*}
for if
\begin{align}
\label{lqmf.mfg.assumption.weyl7}
\sigma_{max}(\bar{P}_t)\sigma_{max}(R_t) < \sigma_{min}(P_t+\bar{P}_t),
\end{align}
then it is sufficient to have the following
\begin{align}
\label{lqmf.mfg.condition.weyl5}
\frac{\sigma_{max}(\bar{P}_t)\cdot\sigma_{max}(R_t)}{\sigma_{min}(P_t+\bar{P}_t) - \sigma_{max}(\bar{P}_t)\cdot\sigma_{max}(R_t)} < \frac{\sigma_{min}(B_t)^2\frac{\sigma_{min}(P_t+\bar{P}_t)}{\sigma_{max}(P_t+\bar{P}_t)} - \sigma_{max}(B_t)\cdot\sigma_{max}(\bar{B}_t)}{\sigma_{max}(B_t)^2 + \sigma_{max}(B_t)\cdot\sigma_{max}(\bar{B}_t)},
\end{align}
so as to guarantee \eqref{lqmf.mfg.condition.weyl8}. It is also clear that \eqref{lqmf.mfg.condition.weyl5} is equivalent to
\begin{align}
\label{lqmf.mfg.condition.weyl6}
\sigma_{max}(R_t) < \frac{\sigma_{min}(P_t+\bar{P}_t)}{\sigma_{max}(\bar{P}_t)}\cdot\frac{\sigma_{min}(B_t)^2\cdot\sigma_{min}(P_t+\bar{P}_t) - \sigma_{max}(B_t)\cdot\sigma_{max}(\bar{B}_t)\cdot\sigma_{max}(P_t+\bar{P}_t)}{\sigma_{min}(B_t)^2\cdot\sigma_{min}(P_t+\bar{P}_t) + \sigma_{max}(B_t)^2\cdot\sigma_{max}(P_t+\bar{P}_t)}.
\end{align}
Therefore, the conditions \eqref{lqmf.mfg.assumption.weyl7} and \eqref{lqmf.mfg.condition.weyl6} can warrant the condition \eqref{lqmf.mfg.condition.weyl2}.

\end{proof}

\subsection{Riccati Equation for Adjoint Equation}

We propose an {\em ansatz} for the solution of the system \eqref{lqmf.mfg.expected.fbode} in the form:
\begin{align}
\label{lqmf.mfg.ansatz.eta}
\eta_t = \Gamma_t \xi_t, \quad \text{for} \quad t\in[0,T].
\end{align}
By applying the chain rule and using \eqref{lqmf.mfg.expected.fbode} and \eqref{lqmf.mfg.ansatz.eta}, we have
\begin{align*}
\frac{d\eta_t}{dt} =& \frac{d\Gamma_t}{dt} \xi_t + \Gamma_t \frac{d\xi_t}{dt} \\
=&  \left(\frac{d\Gamma_t}{dt} + \Gamma_t (A_t+\bar{A}_t -  (B_t+\bar{B}_t) (P_t+\bar{\mathcal{P}}_t)^{-1}\bar{\mathcal{S}}_t) -\Gamma_t(B_t+\bar{B}_t) (P_t+\bar{\mathcal{P}}_t)^{-1}B_t^{\top}\Gamma_t \right) \xi_t.
\end{align*}
Recall \eqref{lqmf.mfg.expected.fbode} that $\eta_t$ satisfies the following equation:
\begin{align*}
-\frac{d\eta_t}{dt} = \left(Q_t+\bar{\mathcal{Q}}_t - \bar{\mathcal{R}}_t(P_t+\bar{\mathcal{P}}_t)^{-1}\bar{\mathcal{S}}_t\xi_t+(A_t^{\top}-\bar{\mathcal{R}}_t(P_t+\bar{\mathcal{P}}_t)^{-1}B_t^{\top})\Gamma_t\right)\xi_t.
\end{align*}
By combining the two equations, we obtain the following non-symmetric Riccati euqation:
\begin{align}
\nonumber
&\frac{d\Gamma_t}{dt} + \Gamma_t (A_t+\bar{A}_t -  (B_t+\bar{B}_t) (P_t+\bar{\mathcal{P}}_t)^{-1}\bar{\mathcal{S}}_t) +(A_t^{\top}-\bar{\mathcal{R}}_t(P_t+\bar{\mathcal{P}}_t)^{-1}B_t^{\top})\Gamma_t \\
\label{lqmf.mfg.equation.riccati.Gamma}
&-\Gamma_t(B_t+\bar{B}_t) (P_t+\bar{\mathcal{P}}_t)^{-1}B_t^{\top}\Gamma_t + Q_t+\bar{\mathcal{Q}}_t - \bar{\mathcal{R}}_t(P_t+\bar{\mathcal{P}}_t)^{-1}\bar{\mathcal{S}}_t = 0,
\end{align}
subject to the terminal condition $\Gamma_T = Q_T+\bar{\mathcal{Q}}_T$. 
\begin{theorem}
	Under any one of the following conditions:
	\begin{itemize}
		\item[1.] Condtion 1 in Theorem \ref{lqmf.mfg.theorem.global_local.solution};
		\item[2.] Condition 2 in Theorem \ref{lqmf.mfg.theorem.global_local.solution};
		\item[3.] Conditions \eqref{lqmf.condition.K1}, \eqref{lqmf.condition.K2} and \eqref{lqmf.condition.Final} in Theorem \ref{lqmf.mfg.theorem.global.uniqueness},
	\end{itemize}
	the nonsymmetric Riccati equation \eqref{lqmf.mfg.equation.riccati.Gamma} admits a unique solution.
\end{theorem}
\begin{proof}
Denote by 
\begin{align*}
\mathcal{A}_t := \left(\begin{array}{cc}
	A_t+\bar{A}_t -  (B_t+\bar{B}_t) (P_t+\bar{\mathcal{P}}_t)^{-1}\bar{\mathcal{S}}_t  & -(B_t+\bar{B}_t) (P_t+\bar{\mathcal{P}}_t)^{-1}B_t^{\top} \\
	-Q_t-\bar{\mathcal{Q}}_t + \bar{\mathcal{R}}_t(P_t+\bar{\mathcal{P}}_t)^{-1}\bar{\mathcal{S}}_t & -A_t^{\top}+\bar{\mathcal{R}}_t(P_t+\bar{\mathcal{P}}_t)^{-1}B_t^{\top}
\end{array}\right).
\end{align*}
We denote the fundamental solution associated with $\mathcal{A}_t$ by $\Phi(t,s)$. By the equation \eqref{lqmf.mfg.expected.fbode}, for any $t \in [0,T]$, we have
\begin{align*}
0 =& (Q_T+\bar{\mathcal{Q}}_T,-I)
\left(\begin{array}{c}\xi_T \\ \eta_T\end{array}\right) = (Q_T+\bar{\mathcal{Q}}_T,-I) \Phi(T,t)
\left(\begin{array}{c}\xi_t \\ \eta_t\end{array}\right) \\
=& (Q_T+\bar{\mathcal{Q}}_T,-I) \Phi(T,t)
\left( \left(\begin{array}{c}I \\ 0\end{array}\right)\xi_t+ \left(\begin{array}{c}0 \\ I\end{array}\right)\eta_t \right).
\end{align*}
By the uniqueness of the solution of the system \eqref{lqmf.mfg.expected.fbode}, the matrix $(Q_T+\bar{\mathcal{Q}}_T,-I) \Phi(T,t)
\left(\begin{array}{c}0 \\ I\end{array}\right)$ is invertible. Therefore, by the {\em ansatz} \eqref{lqmf.mfg.ansatz.eta}, we have
\begin{align*}
-(Q_T+\bar{\mathcal{Q}}_T,-I) \Phi(T,t)
\left(\begin{array}{c}0 \\ I\end{array}\right)\Gamma_t\xi_t  =& (Q_T+\bar{\mathcal{Q}}_T,-I) \Phi(T,t)
\left(\begin{array}{c}I \\ 0\end{array}\right)\xi_t,
\end{align*}
and hence
\begin{align*}
\Gamma_t  =& -\left[(Q_T+\bar{\mathcal{Q}}_T,-I) \Phi(T,t)
\left(\begin{array}{c}0 \\ I\end{array}\right)\right]^{-1}\left[(Q_T+\bar{\mathcal{Q}}_T,-I) \Phi(T,t)
\left(\begin{array}{c}I \\ 0\end{array}\right)\right].
\end{align*}
\end{proof}


\noindent {\bf A counterexample:}
Let $n=m=2$, $S=R=B=I$, $\bar{B}=0$, thus $\bar{\mathcal{P}}=\bar{\mathcal{Q}}=0$.
Let
\begin{align}
\label{lqmf.counterexample.parameter1} Q=\left(\begin{array}{cc}2.1 &-0.3 \\ -0.3 &0.2\end{array}\right), \quad P^{-1}=\left(\begin{array}{cc}0.5 &1.1 \\ 1.1 &3.2\end{array}\right),
\end{align}
\begin{align}
\label{lqmf.counterexample.parameter2} A=\bar{A}=\left(\begin{array}{cc}2.1 &-1.9 \\ -1.2 &1.7\end{array}\right), \quad \bar{\mathcal{S}}=\bar{\mathcal{R}}=\left(\begin{array}{cc}1 &-0.2 \\ 1.4 &0.7\end{array}\right).
\end{align}
It is clear that both the conditions \eqref{lqmf.condition.K1} and \eqref{lqmf.condition.K2} hold, while the condition \eqref{lqmf.condition.Final} is failed to be satisfied. Equation \eqref{lqmf.mfg.expected.fbode} can be rewritten as
\begin{align}
\label{lqmf.counterexample.equation} \frac{d}{d t}\left(\begin{array}{l}
\xi_t \\
\eta_t
\end{array}\right)=\mathcal{A}\left(\begin{array}{l}
\xi_t \\
\eta_t
\end{array}\right),
\end{align}
such that $\eta_T=Q \xi_T$ and $\xi_0=\mathbb{E}\left[x_0\right]$, where
\begin{align*}
\mathcal{A}=\left(
\begin{array}{cccc}
-6.24 & -4.47 & -0.5 & -1.1 \\
-7.98 & 1.38 & -1.1 & -3.2 \\
1.176 & -0.566 & -2.38 & -1.66 \\
-7.062 & -2.152 & -3.37 & -2.08 \\
\end{array}
\right).
\end{align*}
We denote the fundamental solution associated with $\mathcal{A}$ by $\Phi(t,s)$. The equation \eqref{lqmf.counterexample.equation} is (uniquely) solvable if and only if there is a (unique) $\eta_0$ such that
\begin{align}
\label{lqmf.mfg.euqation1.matrix}
-(Q,-I) \Phi(T,0)
\left(\begin{array}{c}0 \\ I\end{array}\right)\eta_0  =& (Q,-I) \Phi(T,0)
\left(\begin{array}{c}I \\ 0\end{array}\right)\mathbb{E}\left[x_0\right],
\end{align}
Denote by $\Phi_1(T) = \text{det}\left((Q,-I) \Phi(T,0)
\left(\begin{array}{c}0 \\ I\end{array}\right)\right)$ and $\Phi_2(T) = \text{det}\left((Q,-I) \Phi(T,0)
\left(\begin{array}{c}I \\ 0\end{array}\right)\right)$. Numerically, $\Phi_1(0.3) = 0.0145965$, $\Phi_1(0.31) = -0.0346916$, by the continuity, $\Phi_1$ admits a zero at some $T_0 \in [0.3,0.31]$, hence $\Phi_1(T_0)$ is not invertible. Hence the null space of $\Phi_1(T_0)$ is not an empty set, from which we can choose at least one element which is denoted by $\widetilde{\eta}$. If $\eta_0$ is a solution of \eqref{lqmf.mfg.euqation1.matrix}, then $\eta_0+\widetilde{\eta}$ is also the solution of \eqref{lqmf.mfg.euqation1.matrix}, which shows \eqref{lqmf.mfg.euqation1.matrix} has no unique solution. On the other hand, we then show that \eqref{lqmf.mfg.euqation1.matrix} fails to be solvable for any $\mathbb{E}\left[x_0\right]$. Since $\Phi_1(T_0)$ is singular, the dimension of its range space (denoted by $\mathcal{R}{\Phi_1(T_0)}$) is strictly less than the dimension of $\eta_0$. Referring to Figure \ref{lqmf.fig.Phi2}, $\Phi_2$ is invertible at $T_0$. We can choose a $\mathbb{E}\left[x_0\right] \neq 0$ such that $\Phi_2(T_0)\mathbb{E}\left[x_0\right]$ does not belong to $\mathcal{R}{\Phi_1(T_0)}$. In this case, there is no such $\eta_0$ solving \eqref{lqmf.mfg.euqation1.matrix}. This counterexample justifies our proposed conditions that are almost optimal.

\section{Mean Field Game Equilibrium Control as an $\varepsilon$-Nash One} \label{lqmf.section.convergence_rate}

In this section, we first deduce the precise convergence rates of the state process and the objective functional under the mean field equilibrium control between the EMFG and the $N$-player Nash game. Then we present the convergence rates of them under a dummy control. Based on these, we can ascertain that the mean field game equilibrium control is the $\varepsilon$-Nash equilibrium.

Recall the definition in the previous context, $\|M\|$ is the spectral matrix norm of matrix $M$. We further define that $\|M\|_T:= \sup_{0\leq t\leq T}\|M_t\|$. Define the following constants which will be used in this whole section:
\begin{align}
\label{lqmf.convergence.comstant.C}
\begin{array}{ccl}
C_1&:=&2\|A^{\top}-N (P+\bar{P})^{-1}B^{\top}\|_T+\|Q+\bar{Q} - N (P+\bar{P})^{-1}N^{\top}\|_T +\|Q+\bar{Q}\|_T^2; \\
C_2&:=&
\min\left\{\lambda_{min}\left(Q_T+\bar{Q}_T\right),\inf_{0\leq t\leq T}\lambda_{min}\left(Q_t+\bar{Q}_t - N_t (P_t+\bar{P}_t)^{-1}N_t^{\top}\right),\inf_{0\leq t\leq T}\lambda_{min}\left(B_t(P_t+\bar{P}_t)^{-1}B_t^{\top}\right)\right\}; \\
C_3&:=& 3K_1 +K_2-K_5+\sqrt{(K_1-K_5-K_2)^2+2(K_3+K_4)}; \\
C_4&:=&\|Q+\bar{\mathcal{Q}}\|_T^2+2\|A^{\top}-\bar{\mathcal{R}}(P+\bar{\mathcal{P}})^{-1}B^{\top}\|_T  +\|Q+\bar{\mathcal{Q}} - \bar{\mathcal{R}}(P+\bar{\mathcal{P}})^{-1}\bar{\mathcal{S}}\|; \\
C_5&:=&3\max\{\|A\|_T^2, \|\bar{A}\|_T^2, \|\bar{B}\|_T^2\}; \\
C_6&:=&\max\{\|Q\|_T,\|\bar{Q}\|_T,\|\bar{P}\|_T,\|N\|_T\}; \\
C_7&:=&\max\{\|S\|_T^2,\|R\|_T^2,\|\bar{S}\|_T^2,\|\bar{R}\|_T^2\};
\end{array}
\end{align}
where $K_1$, $K_2$, $K_3$, $K_4$ and $K_5$ are defined in Subsection \ref{lqmf.section.global}.

\subsection{Convergence Rate under the Mean Field Equilibrium Control}
\label{lqmf.section.epsilon-nash_equilibrium}

For any fixed $i = 1 \cdots N$, we consider Problem \ref{lqmf.mfg.problem.mfg} under the $\sigma$-field $\mathcal{G}_t^i:=\sigma\left(x_0^i, W_s^i, s \leq t\right)$, then the objective functional is defined as follows,
\begin{align}
	\nonumber
	J^i(v^i):=& \mathbb{E}\left[\frac{1}{2} \int_0^T (x^i_t)^{\top} Q_t x^i_t+(v^i_t)^{\top} P_t v^i_t+\left(x^i_t-S_t \mathbb{E}\left[\hat{x}^i_t\right]\right)^{\top} \bar{Q}_t\left(x^i_t-S_t \mathbb{E}\left[\hat{x}^i_t\right]\right) \right.\\
	\nonumber
	&+\left. 2\left(x^i_t-\bar{S}_t \mathbb{E}\left[\hat{x}^i_t\right]\right)^{\top} N_t\left(v^i_t-\bar{R}_t \mathbb{E}\left[\hat{v}^i_t\right]\right) + \left(v^i_t-R_t \mathbb{E}\left[\hat{v}^i_t\right]\right)^{\top} \bar{P}_t\left(v^i_t-R_t \mathbb{E}\left[\hat{v}^i_t\right]\right) d t\right] \\
	\label{lqmf.mfg.objective.appendix}
	&+\mathbb{E}\left[\frac{1}{2} (x^i_T)^{\top} Q_T x^i_T+\frac{1}{2}\left(x^i_T-S_T \mathbb{E}\left[\hat{x}^i_T\right]\right)^{\top} \bar{Q}_T\left(x^i_T-S_T \mathbb{E}\left[\hat{x}^i_T\right]\right)\right],
\end{align}
where we let $\hat{v}^i$ be the mean field equilibrium satisfying \eqref{lqmf.mfg.mf_equilibrium}, and $\hat{x}^i$ and $p^i$ are the corresponding trajectory of \eqref{lqmf.mfg.state_process} and \eqref{lqmf.mfg.adjoint_process_optimal} under $\hat{v}^i$, that is
\begin{align}
\label{lqmf.mfg.system.general}
    \left\{\begin{array}{cclccl}
    d \hat{x}^i_t&=&\left(A_t \hat{x}^i_t+B_t \hat{v}^i_t+\bar{A}_t \mathbb{E}\left[\hat{x}^i_t\right]+\bar{B}_t \mathbb{E}\left[\hat{v}^i_t\right]\right) d t+\sigma_t d W^i_t, &\hat{x}^i(0)&=&x^i_0, \\
    -d p^i_t &=& \left(A_t^{\top}p^i_t +(Q_t+\bar{Q}_t)\hat{x}^i_t - \bar{Q}_t S_t \mathbb{E}\left[\hat{x}^i_t\right] +  N_t \hat{v}^i_t - N_t \bar{R}_t \mathbb{E}\left[\hat{v}^i_t\right]\right)dt - \theta^i_t d W^i_t, &p^i_T &=& (Q_T+\bar{Q}_T)\hat{x}^i_T - \bar{Q}_T S_T \mathbb{E}\left[\hat{x}^i_T\right] \\
    \end{array}\right.
\end{align}
and
\begin{align}
    \label{lqmf.mfg.mf_equilibrium.general}
    \hat{v}^i_t 
    =& - (P_t+\bar{P}_t)^{-1} \left(N_t^{\top} \hat{x}^i_t + B_t^{\top} p^i_t + \left( \bar{P}_t R_t (P_t+\bar{\mathcal{P}}_t)^{-1}\bar{\mathcal{S}}_t- N_t^{\top}\bar{S}_t\right) \mathbb{E}\left[\hat{x}^i_t\right]  + \bar{P}_t R_t (P_t+\bar{\mathcal{P}}_t)^{-1}B_t^{\top} \mathbb{E}\left[p^i_t\right]\right),
\end{align}
then \eqref{lqmf.mfg.system.general} and \eqref{lqmf.mfg.mf_equilibrium.general} solves the problem \eqref{lqmf.mfg.objective.appendix}.

In this subsection, we focus on the convergences rates under the mean field equilibrium control, to facilitate our discussion, we define the following constants $M_{e,1}$, $M_{e,2}$ and $M_{e,3}$ and give them a subscript $e$, meanning the {\em equilibrium control}, to distinguish them from the others:
\begin{align}
\label{lqmf.convergence.comstant.Me}
\begin{array}{ccl}
M_{e,1}&:=&\left(1+\frac{C_1^2}{C_2^2}\right)\left(1-\frac{\delta}{\sup_{0\leq t\leq T}\lambda_{max}\{Q_t+\bar{Q}_t\}}\right)^{-1}\min\left\{\inf_{0\leq t\leq T}\lambda_{min}\left(Q_t+\bar{Q}_t\right),\inf_{0\leq t\leq T}\lambda_{min}\left(P_t+\bar{P}_t\right)\right\}^{-1}; \\
M_{e,2}&:=&\frac{4C_4}{C_3^2}\cdot \left(1+\|(P+\bar{\mathcal{P}})^{-1}\|_T^2\|\bar{\mathcal{S}}\|_T^2+\|(P+\bar{\mathcal{P}})^{-1}\|_T^2\|B^{\top}\|_T^2\right); \\
M_{e,3}&:=&M_{e,1}C_5e^{3C_5t};
\end{array}
\end{align}
where $C_1$, $C_2$, $C_3$, $C_4$ and $C_5$ were defined in \eqref{lqmf.convergence.comstant.C} and $\delta$ is the constant defined in the assumption \eqref{lqmf.mfg.assumption.convexity}.

Let $\hat{y}^i$ be the state process of the $i$-th player for the $N$-players Nash game under the mean field equilibrium $\hat{v}^i$:
\begin{align*}
d \hat{y}^i_t=&\left(A_t \hat{y}^i_t+B_t \hat{v}^1_t+\bar{A}_t \frac{1}{N-1}\sum_{j\neq i}\hat{y}^j_t+\bar{B}_t \frac{1}{N-1}\sum_{j\neq i}\hat{v}^j_t\right) d t+\sigma_t d W^i_t, \quad \hat{y}^i(0)=x^i_0.
\end{align*}

\begin{lemma} \label{lqmf.mfg.lemma.approximation_state.optimal}
	If the coeffecients satisfy the assumption specified in \eqref{lqmf.mfg.assumption.convexity}, for each fixed $i = 1,2,\cdots,N$, the convergence of $\hat{y}^i_t$ to $\hat{x}^i_t$ is:
 \begin{align}
\label{lqmf.mfg.convergence.xhat}
\mathbb{E}\left(\sup_{u \leq t}\|\hat{x}^i_u - \hat{y}^i_u\|^2\right)
\leq& \frac{M_{e,3}}{N-1}  \left(\mathbb{E}\left[\left\|\mathbb{E}\left[x^i_0\right]-x^i_0\right\|^2\right]+T\|\sigma\|_T^2\right),
\end{align}
where $M_{e,3}$ is defined in \eqref{lqmf.convergence.comstant.Me}.
\end{lemma}
We also write the objective functional \eqref{lqmf.N_player.objective} here for the sake of convenience:
\begin{align}
\nonumber
\mathcal{J}^i(v^i,v^{-i}) :=& \mathbb{E}\left[\frac{1}{2} \int_0^T (y^i_t)^{\top} Q_t y^i_t+(v^i_t)^{\top} P_t v^i_t+\left(y^i_t-S_t \frac{1}{N-1}\sum_{j\neq i}y^j_t\right)^{\top} \bar{Q}_t\left(y^i_t-S_t \frac{1}{N-1}\sum_{j\neq i}y^j_t\right) \right.\\
\nonumber
&\qquad \quad+\left. 2\left(y^i_t-\bar{S}_t \frac{1}{N-1}\sum_{j\neq i}y^j_t\right)^{\top} N_t\left(v^i_t-\bar{R}_t \frac{1}{N-1}\sum_{j\neq i}v^j_t\right)\right. \\
\nonumber
&\qquad \quad+ \left. \left(v^i_t-R_t \frac{1}{N-1}\sum_{j\neq i}v^j_t\right)^{\top} \bar{P}_t\left(v^i_t-R_t \frac{1}{N-1}\sum_{j\neq i}v^j_t\right) d t\right] \\
\label{lqmf.N_player.objective.appendix}
&+\mathbb{E}\left[\frac{1}{2} (x^i_T)^{\top} Q_T x^i_T+\frac{1}{2}\left(x^i_T-S_T \frac{1}{N-1}\sum_{j\neq i}y^j_T\right)^{\top} \bar{Q}_T\left(x^i_T-S_T \frac{1}{N-1}\sum_{j\neq i}y^j_T\right)\right],
\end{align}
where $v^i$ and $v^{-i}$ are dummy controls.
\begin{lemma}
	\label{lqmf.mfg.lemma.approximation_objective.optimal}
	Under the condition of Theorem \ref{lqmf.mfg.theorem.global.uniqueness}, for each fixed $i = 1,2,\cdots,N$, the convergence of the objective functional $\mathcal{J}^i(\hat{v}^i,\hat{v}^{-i})$ to $J^i(\hat{v}^i)$ is
	\begin{align*}
	\left|\mathcal{J}^i(\hat{v}^i,\hat{v}^{-i})-J^i(\hat{v}^i)\right| \leq \frac{42C_6(1+C_7)}{\sqrt{N-1}}\left\{M_{e,2} \left\|\mathbb{E}\left[x^i_0\right]\right\|^2 +(M_{e,1}+M_{e,3}(T+1)) \left(\mathbb{E}\left[\left\|\mathbb{E}\left[x^i_0\right]-x^i_0\right\|^2\right]+T\|\sigma\|_T^2\right)\right\},
	\end{align*}
	where $C_6$ and $C_7$ are defined in \eqref{lqmf.convergence.comstant.C}, and $M_{e,1}$, $M_{e,2}$ and $M_{e,3}$ are defined in \eqref{lqmf.convergence.comstant.Me}.
\end{lemma}

The proofs of Lemma \ref{lqmf.mfg.lemma.approximation_state.optimal} and \eqref{lqmf.mfg.lemma.approximation_objective.optimal} are put in Appendix \ref{lqmf.section.appendixA4}.
\subsection{Convergence Rate under the Dummy Control}

Given a group of controls $\{v^1,\hat{v}^2,\dots,\hat{v}^N\}$, where, for $i = 2,3,\cdots,N$,each $\hat{v}^i$ is the mean field equilibrium control for Problem \ref{lqmf.mfg.problem.mfg} under $\mathcal{G}_t^i:=\sigma\left(x_0^i, W_s^i, s \leq t\right)$ while $v^1$ is a dummy one, we consider the corresponding dynamics for each $2 \leq i \leq N$:
\begin{align}
\label{lqmf.mfg.state_process.dummy}
\begin{array}{ccl}
d x^1_t&=&\left(A_t x^1_t+B_t v^1_t+\bar{A}_t \mathbb{E}\left[\hat{x}^1_t\right]+\bar{B}_t \mathbb{E}\left[\hat{v}^1_t\right]\right) d t+\sigma_t d W^1_t, \quad x^1(0)=x^1_0; \\
d y^1_t&=&\left(A_t y^1_t+B_t v^1_t+\bar{A}_t \frac{1}{N-1}\sum_{j\neq 1}\hat{y}^j_t+\bar{B}_t \frac{1}{N-1}\sum_{j\neq 1}\hat{v}^j_t\right) d t+\sigma_t d W^1_t, \quad y^1(0)=x^1_0; \\
d \hat{x}^i_t&=&\left(A_t \hat{x}^i_t+B_t \hat{v}^i_t+\bar{A}_t \mathbb{E}\left[\hat{x}^1_t\right]+\bar{B}_t \mathbb{E}\left[\hat{v}^1_t\right]\right) d t+\sigma_t d W^i_t, \quad \hat{x}^i(0)=x^i_0; \\
d \hat{y}^i_t&=&\left(A_t \hat{y}^i_t+B_t \hat{v}^1_t+\bar{A}_t \frac{1}{N-1}\left(y^1_t+\sum_{j\neq 1,i}\hat{y}^j_t\right)+\bar{B}_t \frac{1}{N-1}\left(v^1_t+\sum_{j\neq 1,i}\hat{v}^j_t\right)\right) d t+\sigma_t d W^i_t, \quad \hat{y}^i(0)=x^i_0, \quad i \neq 1;
\end{array}
\end{align}
where $\hat{v}^1_t$ is the mean field equilibrium for Problem \ref{lqmf.mfg.problem.mfg} with respect to $\mathcal{G}_t^1:=\sigma\left(x_0^1, W_s^1, s \leq t\right)$ and $\hat{x}^1_t$ is the corresponding optimal state. The rest of this subsection concerns the convergences rates for a given dummy control $v^1$, we define the constants $M_{d,4}$, $M_{d,5}$, $M_{d,6}$, and $M_{d,7}$ as follows and give them a subscript $d$, meaning {\em dummy}, to distinguish them from the others:
\begin{align}
\label{lqmf.convergence.comstant.Md}
\begin{array}{ccl}
M_{d,4}&:=&(1+\|B\|_T)\exp\{T(2\|A\|_T+\|B\|_T)\}; \\
M_{d,5}&:=&\left(1+\|B\|_T+M_{e,2}\|\bar{A}\|_T+M_{e,2}\|\bar{B}\|_T\right)\exp\{T(2\|A\|_T+\|B\|_T+\|\bar{A}\|_T+\|\bar{B}\|_T)\}; \\
M_{d,6}&:=&6C_5e^{7C_5t}(M_{e,1}+M_{d,4}); \\
M_{d,7}&:=&6C_5e^{7C_5t}M_{d,4};
\end{array}
\end{align}
where $C_5$ is defined in \eqref{lqmf.convergence.comstant.C}, and $M_{e,1}$ and $M_{e,2}$ were defined in \eqref{lqmf.convergence.comstant.Me}.

\begin{lemma} \label{lqmf.mfg.lemma.approximation_state.pertubed}
	For each fixed $i = 1,2,\cdots,N$, the convergence of $x^1$ and $\hat{x}^i$ to $y^1$ and $\hat{y}^i$ together satisfy:
	\begin{align*}
	&\mathbb{E}\sup_{u\leq t}\|x^1_t-y^1_t\|^2 + \mathbb{E}\sup_{u\leq t}\|\hat{x}^i_t-\hat{y}^i_t\|^2 \\
	&\leq\frac{1}{N-1}\bigg\{M_{d,6}\left(\mathbb{E}\left[\left\|\mathbb{E}\left[x^i_0\right]-x^i_0\right\|^2\right]+T\|\sigma\|_T^2\right) +M_{d,7}\int_0^t\mathbb{E}\left[\left\|\mathbb{E}\left[v^1_s\right]-v^1_s\right\|^2\right] ds\bigg\},
	\end{align*}
	where $M_{d,6}$ and $M_{d,7}$ were defined in \eqref{lqmf.convergence.comstant.Md}.
\end{lemma}

\begin{lemma}
	\label{lqmf.mfg.lemma.approximation_objective.pertubed}
	The convergence of the objective functional is:
	\begin{align*}
	&\left|\mathcal{J}(v^1,\hat{v}^{-1})-J(v^1)\right| \\
	&\leq \frac{36C_6(1+C_7)}{\sqrt{N-1}}\bigg((M_{e,2}+M_{d,5}(T+1))\left\|\mathbb{E}\left[x^1_0\right]\right\|^2 +(M_{e,1}+M_{d,4}(T+1)+M_{d,6}(T+1))\left(\mathbb{E}\left[\left\|\mathbb{E}\left[x^1_0\right]-x^1_0\right\|^2\right] +T\|\sigma\|_T^2\right) \\
	&\quad+M_{d,5}(T+1)\int_0^t\left\|\mathbb{E}\left[v^1_s\right]\right\|^2 ds + \int_0^T\mathbb{E}\left[\left\|v^1_t\right\|^2\right]dt+(M_{d,4}+M_{d,7})(T+1) \int_0^t\mathbb{E}\left[\left\|\mathbb{E}\left[v^1_s\right]-v^1_s\right\|^2\right] ds\bigg),
	\end{align*}
	where $C_6$ and $C_7$ are defined in \eqref{lqmf.convergence.comstant.C}, and $M_{e,1}$ and $M_{e,2}$ are defined in \eqref{lqmf.convergence.comstant.Me}, and $M_{d,4}$, $M_{d,5}$, $M_{d,6}$ and $M_{d,7}$ are defined in \eqref{lqmf.convergence.comstant.Md}.
\end{lemma}

The proofs of Lemma \ref{lqmf.mfg.lemma.approximation_state.pertubed} and \eqref{lqmf.mfg.lemma.approximation_objective.pertubed} are put in Appendix \ref{lqmf.section.appendixA4}.

\subsection{Mean Field Game Equilibrium is the $\varepsilon$-Nash Equilibrium}

\begin{theorem}
	$\hat{v}$ is an $\varepsilon$-Nash equilibrium for the Nash game of a finite number of players, the convergence rate is as follows
	\begin{align*}
	\mathcal{J}(\hat{v}^1,\hat{v}^{-1}) \leq& \mathcal{J}(v^1,\hat{v}^{-1}) + \frac{1}{\sqrt{N-1}}\bigg\{L_1\left\|\mathbb{E}\left[x^1_0\right]\right\|^2+L_2\left(\mathbb{E}\left[\left\|\mathbb{E}\left[x^1_0\right]-x^1_0\right\|^2\right] +T\|\sigma\|_T^2\right) \\
	&+L_3\int_0^T\left\|\mathbb{E}\left[v^1_t\right]\right\|^2dt +L_4\int_0^T\mathbb{E}\left[\left\|v^1_t\right\|^2\right]dt + L_5\int_0^T\mathbb{E}\left[\left\|\mathbb{E}\left[v^1_t\right]-v^1_t\right\|^2\right] dt \bigg\},
	\end{align*}
	where
	\begin{align*}
	\begin{array}{ccl}
	L_1&:=&78C_6(1+C_7)M_{e,2}+36C_6(1+C_7)M_{d,5}(T+1); \\
	L_2&:=&78C_6(1+C_7)M_{e,1}+42C_6(1+C_7)M_{e,3}(T+1)+36C_6(1+C_7)M_{d,4}(T+1)+36C_6(1+C_7)M_{d,6}(T+1); \\
	L_3&:=&36C_6(1+C_7)M_{d,5}(T+1); \\
	L_4&:=&36C_6(1+C_7); \\
	L_5&:=&36C_6(1+C_7)(M_{d,4}+M_{d,7})(T+1);
	\end{array}
	\end{align*}
	where $C_6$ and $C_7$ were defined in \eqref{lqmf.convergence.comstant.C}, and $M_{e,1}$, $M_{e,2}$ and $M_{e,3}$ were defined in \eqref{lqmf.convergence.comstant.Me}, and $M_{d,4}$, $M_{d,5}$, $M_{d,6}$ and $M_{d,7}$ were defined in \eqref{lqmf.convergence.comstant.Md}.
\end{theorem}

\begin{proof}
	Recall the results of Lemma \ref{lqmf.mfg.lemma.approximation_objective.optimal} and Lemma \ref{lqmf.mfg.lemma.approximation_objective.pertubed}:
	\begin{align*}
	\left|\mathcal{J}(\hat{v}^1,\hat{v}^{-1})-J(\hat{v}^1)\right| = O\left(\frac{1}{\sqrt{N}}\right), \quad
	\left|\mathcal{J}(v^1,\hat{v}^{-1})-J(v^1)\right| = O\left(\frac{1}{\sqrt{N}}\right),
	\end{align*}
	by applying which we can conclude the present claim; indeed, since $\hat{v}^1$ is the optimal control for the mean field game, we have 
	\begin{align*}
	J(\hat{v}^1) \leq J(v^1).
	\end{align*}
	Therefore, 	by noting that $J(\hat{v}^1)-J(v^1) \leq 0$,
	\begin{align*}
	\mathcal{J}(\hat{v}^1,\hat{v}^{-1}) =& \mathcal{J}(\hat{v}^1,\hat{v}^{-1})-J(\hat{v}^1)+J(\hat{v}^1)-J(v^1)+J(v^1)-\mathcal{J}(v^1,\hat{v}^{-1})+\mathcal{J}(v^1,\hat{v}^{-1}) \\
	\leq& \left|\mathcal{J}(\hat{v}^1,\hat{v}^{-1})-J(\hat{v}^1)\right|+\left|J(v^1)-\mathcal{J}(v^1,\hat{v}^{-1})\right| + \mathcal{J}(v^1,\hat{v}^{-1}).
	\end{align*}
\end{proof}

\section{Extended Mean Field Type Control Problems} \label{lqmf.section.emftc}
Note that the extended mean field game is an approximation of the equilibrium strategies of a Nash game with a huge number of players, and the mean field of states and controls are defined exogenously by the equilibrium states and controls of other players. In contrast to the extended mean field game, mean field terms of the {\em extended mean field type control} are controlled {\em endogenously} by the current dynamics of state process and control. The celebrated Markowitz's mean variance problem of an individual investor is one of the special case. We refer to \cite{bensoussan2013mean}, \cite{bensoussan2017interpretation}, \cite{ismail2019robust}, \cite{alasseur2020extended}, \cite{bensoussan2020control}, \cite{pham2022portfolio} and \cite{bensoussan2022dynamic} for different well-received examples. To this end, we here study the linear quadratic extended mean field type control problem in a generic form as follows:
\begin{problem} \label{lqmf.mfc.problem.mfc}
	Find an optimal control $\hat{v}_t \in L_{\mathcal{G}_{\cdot}}^{2}\left(0, T ; \mathbb{R}^m\right)$ which minimizes
	\begin{align*}
	J(v):=& \mathbb{E}\left[\frac{1}{2} \int_0^T x_t^{\top} \mathsf{Q}_t x_t+v_t^{\top} \mathsf{P}_t v_t+\left(x_t-\mathsf{S}_t \mathbb{E}\left[x_t\right]\right)^{\top} \bar{\mathsf{Q}}_t\left(x_t-\mathsf{S}_t \mathbb{E}\left[x_t\right]\right) \right.\\
	&\qquad \quad+\left. 2\left(x_t-\bar{\mathsf{S}}_t \mathbb{E}\left[x_t\right]\right)^{\top} \mathsf{N}_t\left(v_t-\bar{\mathsf{R}}_t \mathbb{E}\left[v_t\right]\right) + \left(v_t-\mathsf{R}_t \mathbb{E}\left[v_t\right]\right)^{\top} \bar{\mathsf{P}}_t\left(v_t-\mathsf{R}_t \mathbb{E}\left[v_t\right]\right) d t\right] \\
	&+\mathbb{E}\left[\frac{1}{2} x_T^{\top} \mathsf{Q}_T x_T+\frac{1}{2}\left(x_T-\mathsf{S}_T \mathbb{E}\left[x_T\right]\right)^{\top} \bar{\mathsf{Q}}_T\left(x_T-\mathsf{S}_T \mathbb{E}\left[x_T\right]\right)\right]
	\end{align*}
	where the controlled dynamics is governed by
	\begin{align*}
	d x_t=\left(\mathsf{A}_t x_t+\mathsf{B}_t v_t+\bar{\mathsf{A}}_t \mathbb{E}\left[x_t\right]+\bar{\mathsf{B}}_t \mathbb{E}\left[v_t\right]\right) d t+\sigma_t d W_t, \quad x(0)=x_0.
	\end{align*}
\end{problem}
It is worth to mention that, in the extended mean field type control problem, the mean field terms are the expectation of \emph{admissible} control $\mathbb{E}\left[v_t\right]$ and its corresponding expectation of the state process $\mathbb{E}\left[x_t\right]$. It is not the case in the extended mean field game (see Problem \ref{lqmf.mfg.problem.mfg}), where the mean field terms are the expectation of the \emph{mean field equilibrium} $\mathbb{E}\left[\hat{v}_t\right]$ and that of the \emph{optimal} state process $\mathbb{E}\left[\hat{x}_t\right]$ of other players.

For simplicity, we denote by
\begin{align*}
\widetilde{\mathscr{Q}}_t:=(I-\mathsf{S}_t^{\top})\bar{\mathsf{Q}}_t(I-\mathsf{S}_t), \quad \widetilde{\mathscr{P}}_t:=(I-\bar{\mathsf{R}}_t^{\top})\bar{\mathsf{P}}_t(I-\bar{\mathsf{R}}_t), \quad
\widetilde{\mathscr{N}}_t := (I-\bar{\mathsf{S}}_t^{\top})\mathsf{N}_t(I-\bar{\mathsf{R}}_t).
\end{align*}

Define an admissible control variable
\begin{align}
    \nonumber \hat{v}_t =& (\mathsf{P}_t + \bar{\mathsf{P}}_t)^{-1}\mathsf{B}_t^{\top}\left(\mathbb{E}\left[p_t\right]-p_t\right)-\left(\mathsf{P}_t + \widetilde{\mathscr{P}}_t\right)^{-1}(\mathsf{B}_t^{\top}+\bar{\mathsf{B}}_t^{\top})\mathbb{E}\left[p_t\right]  \\
    \label{lqmf.mfc.optimal_control} &+(\mathsf{P}_t + \bar{\mathsf{P}}_t)^{-1}\mathsf{N}_t^{\top}\left(\mathbb{E}\left[\hat{x}_t\right]-\hat{x}_t\right)-\left(\mathsf{P}_t + \widetilde{\mathscr{P}}_t\right)^{-1} \widetilde{\mathscr{N}}_t^{\top} \mathbb{E}\left[\hat{x}_t\right],
 \end{align}
where $\hat{x}$ is the corresponding state process:
 \begin{align}
    \label{lqmf.mfc.optimal_state}
    d \hat{x}_t=\left(\mathsf{A}_t \hat{x}_t+\mathsf{B}_t \hat{v}_t+\bar{\mathsf{A}}_t \mathbb{E}\left[\hat{x}_t\right]+\bar{\mathsf{B}}_t \mathbb{E}\left[\hat{v}_t\right]\right) d t+\sigma_t d W_t, \quad \hat{x}(0)=x_0.
\end{align}
and $p$ is the adjoint process satisfying the BSDE:
\begin{align}
    \label{lqmf.mfc.adjoint_process}
    \left\{\begin{array}{ccl}
    -d p_t &=& \left(\mathsf{A}_t^{\top} p_t  + \bar{\mathsf{A}}_t^{\top} \mathbb{E}\left[p_t\right] +(\mathsf{Q}_t + \bar{\mathsf{Q}}_t) \hat{x}_t + ( - \bar{\mathsf{Q}}_t + \widetilde{\mathscr{Q}}_t  ) \mathbb{E}\left[\hat{x}_t\right]  +\mathsf{N}_t\hat{v}_t + ( - \mathsf{N}_t + \widetilde{\mathscr{N}}_t ) \mathbb{E}\left[\hat{v}_t\right]\right)dt - \theta_t d W_t, \\
    p_T &=& (\mathsf{Q}_T + \bar{\mathsf{Q}}_T) \hat{x}_T +   ( - \bar{\mathsf{Q}}_T + \widetilde{\mathscr{Q}}_t  ) \mathbb{E}\left[\hat{x}_T\right],
    \end{array}\right.
\end{align}

\begin{theorem} \label{lqmf.mfc.theorem_smp}
	If the fully-coupled FBSDE system \eqref{lqmf.mfc.optimal_control}, \eqref{lqmf.mfc.optimal_state} and \eqref{lqmf.mfc.adjoint_process} has a unique solution, then Problem \eqref{lqmf.mfc.problem.mfc} is solvable with the optimal control given by \eqref{lqmf.mfc.optimal_control}.	
\end{theorem}
Its proof is provided in Appendix \ref{lqmf.section.appendixA5}.

\begin{theorem}
	Suppose that there exists a scale constant $\delta >0$ such that for all $t \in [0,T]$, 
    \begin{align}
    \label{lqmf.mfc.assumption1.convexity}
    (\mathsf{Q}_t+\bar{\mathsf{Q}}_t)  -  \mathsf{N}_t (\mathsf{P}_t+\bar{\mathsf{P}}_t)^{-1} \mathsf{N}_t^{\top} \geq 0 \quad \text{and}\quad \mathsf{B}_t (\mathsf{P}_t + \bar{\mathsf{P}}_t)^{-1}\mathsf{B}_t^{\top} > \delta I,
    \end{align}
    and
    \begin{align}
    \label{lqmf.mfc.assumption2.convexity}
    \mathsf{Q}_t+\bar{\mathscr{Q}}_t - \widetilde{\mathscr{N}}_t\left(\mathsf{P}_t + \widetilde{\mathscr{P}}_t\right)^{-1}\widetilde{\mathscr{N}}_t^{\top} \geq 0 \quad \text{and}\quad (\mathsf{B}_t+\bar{\mathsf{B}}_t) \left(\mathsf{P}_t + \widetilde{\mathscr{P}}_t\right)^{-1}(\mathsf{B}_t^{\top}+\bar{\mathsf{B}}_t^{\top}) > \delta I,
    \end{align}
    are satisfied, the system \eqref{lqmf.mfc.optimal_control}, \eqref{lqmf.mfc.optimal_state} and \eqref{lqmf.mfc.adjoint_process} admits a unique solution, .
\end{theorem}

\begin{proof}
For simplicity, we denote by $\bar{x}_t := \mathbb{E}\left[\hat{x}_t\right]$, $\bar{p}_t := \mathbb{E}\left[p_t\right]$, and $\bar{v}_t := \mathbb{E}\left[\hat{v}_t\right]$.
By substituting \eqref{lqmf.mfc.optimal_control} into \eqref{lqmf.mfc.optimal_state} and \eqref{lqmf.mfc.adjoint_process}, we have
\begin{align}
\label{lqmf.mfc.fbsde}
\left\{\begin{array}{ccl}
d\left(\begin{array}{c}
\hat{x}_t \\
-p_t
\end{array}\right) &=&\left(\begin{array}{cc}
\mathsf{A}_t-\mathsf{B}_t(\mathsf{P}_t + \bar{\mathsf{P}}_t)^{-1}\mathsf{N}_t^{\top}  &  -\mathsf{B}_t (\mathsf{P}_t + \bar{\mathsf{P}}_t)^{-1}\mathsf{B}_t^{\top} \\
\mathsf{Q}_t + \bar{\mathsf{Q}}_t-\mathsf{N}_t(\mathsf{P}_t + \bar{\mathsf{P}}_t)^{-1}\mathsf{N}_t^{\top} & \mathsf{A}_t^{\top}-\mathsf{N}_t(\mathsf{P}_t + \bar{\mathsf{P}}_t)^{-1}\mathsf{B}_t^{\top}
\end{array}\right)\left(\begin{array}{l}
\hat{x}_t \\
p_t
\end{array}\right)dt +\mathcal{M}\left(\begin{array}{l}
\bar{x}_t \\
\bar{p}_t
\end{array}\right)dt +\left(\begin{array}{l}
\sigma_t \\
-\theta_t
\end{array}\right)dW_t, \\
\hat{x}(0) &=&x_0, \quad p_T =(\mathsf{Q}_T + \bar{\mathsf{Q}}_T) \hat{x}_T +   ( - \bar{\mathsf{Q}}_T + \widetilde{\mathscr{Q}}_t  ) \bar{x}_T,
\end{array}\right.
\end{align}
where $\mathcal{M} := \left(\begin{array}{cl}
\mathcal{M}_{11}  & \mathcal{M}_{12}  \\
\mathcal{M}_{21}  & \mathcal{M}_{22}
\end{array}\right)$ is a block matrix such that
\begin{align*}
\begin{array}{ccl}
\mathcal{M}_{11} &:=& \bar{\mathsf{A}}_t+\mathsf{B}_t(\mathsf{P}_t + \bar{\mathsf{P}}_t)^{-1}\mathsf{N}_t^{\top}-(\mathsf{B}_t+\bar{\mathsf{B}}_t)\left(\mathsf{P}_t + \widetilde{\mathscr{P}}_t\right)^{-1} \widetilde{\mathscr{N}}_t^{\top}; \\
\mathcal{M}_{12} &:=& \mathsf{B}_t (\mathsf{P}_t + \bar{\mathsf{P}}_t)^{-1}\mathsf{B}_t^{\top}-(\mathsf{B}_t+\bar{\mathsf{B}}_t)\left(\mathsf{P}_t + \widetilde{\mathscr{P}}_t\right)^{-1}(\mathsf{B}_t^{\top}+\bar{\mathsf{B}}_t^{\top}); \\
\mathcal{M}_{21} &:=& - \bar{\mathsf{Q}}_t + \widetilde{\mathscr{Q}}_t+\mathsf{N}_t(\mathsf{P}_t + \bar{\mathsf{P}}_t)^{-1}\mathsf{N}_t^{\top}-\widetilde{\mathscr{N}}_t\left(\mathsf{P}_t + \widetilde{\mathscr{P}}_t\right)^{-1} \widetilde{\mathscr{N}}_t^{\top}; \\
\mathcal{M}_{22} &:=& \bar{\mathsf{A}}_t^{\top}+\mathsf{N}_t(\mathsf{P}_t + \bar{\mathsf{P}}_t)^{-1}\mathsf{B}_t^{\top}- \widetilde{\mathscr{N}}_t\left(\mathsf{P}_t + \widetilde{\mathscr{P}}_t\right)^{-1}(\mathsf{B}_t^{\top}+\bar{\mathsf{B}}_t^{\top}).
\end{array}
\end{align*}
Since both $\mathsf{P}_t+\bar{\mathsf{P}}_t$ and $\mathsf{Q}_t+\bar{\mathsf{Q}}_t$ are positive definite matrices, by Theorem 2.2 of \cite{peng1999fully}, for any fixed $(\bar{x},\bar{p})$, the FBSDE system \eqref{lqmf.mfc.fbsde} possesses a unique solution if the conditions in \eqref{lqmf.mfc.assumption1.convexity} are satisfied. To this end, we need $\bar{x}_t = \mathbb{E}\left[\hat{x}_t\right]$ and $\bar{p}_t = \mathbb{E}\left[p_t\right]$, which can be guaranteed if there is a unique pair $(\bar{x},\bar{p})$ solving the following FBODEs:
\begin{align}
\label{lqmf.mfc.expected.fbode}
\left\{\begin{array}{ccl}
\frac{d}{d t}\left(\begin{array}{c}
\bar{x}_t \\
-\bar{p}_t
\end{array}\right) &=&\left(\begin{array}{cc}
\mathsf{A}_t+\bar{\mathsf{A}}_t -(\mathsf{B}_t+\bar{\mathsf{B}}_t) \left(\mathsf{P}_t + \widetilde{\mathscr{P}}_t\right)^{-1}  \widetilde{\mathscr{N}}_t^{\top}  &  -(\mathsf{B}_t+\bar{\mathsf{B}}_t) \left(\mathsf{P}_t + \widetilde{\mathscr{P}}_t\right)^{-1}(\mathsf{B}_t^{\top}+\bar{\mathsf{B}}_t^{\top}) \\
\mathsf{Q}_t+\widetilde{\mathscr{Q}}_t - \widetilde{\mathscr{N}}_t\left(\mathsf{P}_t + \widetilde{\mathscr{P}}_t\right)^{-1}\widetilde{\mathscr{N}}_t^{\top} & \mathsf{A}_t^{\top}+\bar{\mathsf{A}}_t^{\top}  - \widetilde{\mathscr{N}}_t\left(\mathsf{P}_t + \widetilde{\mathscr{P}}_t\right)^{-1}(\mathsf{B}_t^{\top}+\bar{\mathsf{B}}_t^{\top})
\end{array}\right)\left(\begin{array}{l}
\bar{x}_t \\
\bar{p}_t
\end{array}\right), \\
\bar{x}_0 &=&\mathbb{E}\left[x_0\right], \quad\bar{p}_T =\left(\mathsf{Q}_T + \widetilde{\mathscr{Q}}_t\right)\bar{x}_T.
\end{array}\right.
\end{align}
Note that both $\mathsf{P}_t+\widetilde{\mathscr{P}}_t$ and $\mathsf{Q}_t+\widetilde{\mathscr{Q}}_t$ are positive definite, by using Theorem 2.2 of \cite{peng1999fully} again, the system \eqref{lqmf.mfc.expected.fbode} has a unique solution if the conditions in \eqref{lqmf.mfc.assumption2.convexity} hold.
\end{proof}

\subsection{Closed-form Solutions for Extended Mean Field Type Control Problems}

To obtain the explicit solution for Problem \ref{lqmf.mfc.problem.mfc}, we can first examine the equation system \eqref{lqmf.mfc.expected.fbode} and suppose it has a solution in the form: $\bar{p}_t = \bar{\Gamma}_t\bar{x}_t$. By substituting the {\em ansatz} back into \eqref{lqmf.mfc.expected.fbode}, we obtain the following symmetric Riccati equation:
\begin{align}
	\nonumber
	&\frac{d\bar{\Gamma}_t}{dt} + \bar{\Gamma}_t \left(\mathsf{A}_t+\bar{\mathsf{A}}_t -(\mathsf{B}_t+\bar{\mathsf{B}}_t) \left(\mathsf{P}_t + \widetilde{\mathscr{P}}_t\right)^{-1}  \widetilde{\mathscr{N}}_t^{\top}\right) +\left(\mathsf{A}_t^{\top}+\bar{\mathsf{A}}_t^{\top}  - \widetilde{\mathscr{N}}_t\left(\mathsf{P}_t + \widetilde{\mathscr{P}}_t\right)^{-1}(\mathsf{B}_t^{\top}+\bar{\mathsf{B}}_t^{\top})\right)\bar{\Gamma}_t \\
	\label{lqmf.mfc.equation.riccati.Gamma}
	&-\bar{\Gamma}_t(\mathsf{B}_t+\bar{\mathsf{B}}_t) \left(\mathsf{P}_t + \widetilde{\mathscr{P}}_t\right)^{-1}(\mathsf{B}_t^{\top}+\bar{\mathsf{B}}_t^{\top})\bar{\Gamma}_t + \mathsf{Q}_t+\widetilde{\mathscr{Q}}_t - \widetilde{\mathscr{N}}_t\left(\mathsf{P}_t + \widetilde{\mathscr{P}}_t\right)^{-1}\widetilde{\mathscr{N}}_t^{\top} = 0,
\end{align}
with the terminal condition $\bar{\Gamma}_T = \mathsf{Q}_T + \widetilde{\mathscr{Q}}_t$. By the classical Riccati equation theory, \eqref{lqmf.mfc.equation.riccati.Gamma} admits a unique solution under the condition \eqref{lqmf.mfc.assumption2.convexity}. Denote $x^b_t:=\hat{x}_t-\bar{x}_t$ and $p^b_t:=p_t-\bar{p}_t$, then the difference of \eqref{lqmf.mfc.expected.fbode} from \eqref{lqmf.mfc.fbsde} gives the following FBSDEs:
\begin{align}
	\label{lqmf.mfc.fbsde.bias}
	\left\{\begin{array}{ccl}
		d\left(\begin{array}{c}
			x^b_t \\
			-p^b_t
		\end{array}\right) &=&\left(\begin{array}{cc}
			\mathsf{A}_t-\mathsf{B}_t(\mathsf{P}_t + \bar{\mathsf{P}}_t)^{-1}\mathsf{N}_t^{\top}  &  -\mathsf{B}_t (\mathsf{P}_t + \bar{\mathsf{P}}_t)^{-1}\mathsf{B}_t^{\top} \\
			\mathsf{Q}_t + \bar{\mathsf{Q}}_t-\mathsf{N}_t(\mathsf{P}_t + \bar{\mathsf{P}}_t)^{-1}\mathsf{N}_t^{\top} & \mathsf{A}_t^{\top}-\mathsf{N}_t(\mathsf{P}_t + \bar{\mathsf{P}}_t)^{-1}\mathsf{B}_t^{\top}
		\end{array}\right)\left(\begin{array}{l}
			x^b_t \\
			p^b_t
		\end{array}\right)dt +\left(\begin{array}{l}
			\sigma_t \\
			-\theta_t
		\end{array}\right)dW_t, \\
		x^b(0) &=&x_0-\mathbb{E}\left[x_0\right], \quad p^b_T =(\mathsf{Q}_T + \bar{\mathsf{Q}}_T) x^b_T,
	\end{array}\right.
\end{align}
We also guess its solution in the form: $p^b_t = \Xi^b_tx^b_t$, then we have the following symmetric Riccati equation:
\begin{align*}
	&\frac{d\Xi^b_t}{dt} + \Xi^b_t \left(\mathsf{A}_t-\mathsf{B}_t(\mathsf{P}_t + \bar{\mathsf{P}}_t)^{-1}\mathsf{N}_t^{\top}\right) +\left(\mathsf{A}_t^{\top}-\mathsf{N}_t(\mathsf{P}_t + \bar{\mathsf{P}}_t)^{-1}\mathsf{B}_t^{\top}\right)\Xi^b_t -\Xi^b_t\mathsf{B}_t (\mathsf{P}_t + \bar{\mathsf{P}}_t)^{-1}\mathsf{B}_t^{\top}\Xi^b_t + \mathsf{Q}_t + \bar{\mathsf{Q}}_t-\mathsf{N}_t(\mathsf{P}_t + \bar{\mathsf{P}}_t)^{-1}\mathsf{N}_t^{\top} = 0,
\end{align*}
subject to the terminal condition $\Xi^b_T = \mathsf{Q}_T + \bar{\mathsf{Q}}_T$, which possesses a unique solution if the condition \eqref{lqmf.mfc.assumption1.convexity} holds. At this point, we can conclude that $\hat{x}_t=\bar{x}_t+x^b_t$ and $p_t=\bar{p}_t+p^b_t$.

\bibliographystyle{abbrv}
\bibliography{reference}

\begin{appendices}
	
\section{Weyl's inequalities}	\label{lqmf.section.appendixA0}
We here recall the celebrated Weyl's inequalities (see \cite{fulton2000eigenvalues} and \cite{knutson2001honeycombs}), for any Hermitian matrices $A$ and $B$,
\begin{align}
	\label{lqmf.mfg.inequality1.weyl}
	\lambda_{max}(A) + \lambda_{min}(B) \leq \lambda_{max}(A+B) \leq \lambda_{max}(A) + \lambda_{max}(B),
\end{align}
and
\begin{align}
	\label{lqmf.mfg.inequality2.weyl}
	\lambda_{min}(A) + \lambda_{min}(B) \leq \lambda_{min}(A+B) \leq \lambda_{max}(A) + \lambda_{min}(B).
\end{align}
Furthermore, if matrices $A$ and $B$ are positive definite, we have
\begin{align}
	\label{lqmf.mfg.inequality3.weyl}
	\lambda_{min}(AB) \geq \lambda_{min}(A)\cdot\lambda_{min}(B), \quad \lambda_{max}(AB) \leq \lambda_{max}(A)\cdot\lambda_{max}(B).
\end{align}
For any matrices $A$ and $B$ of the same dimension, being not necessarily symmetric or square, by \cite{horn2012matrix}, we also have
\begin{align}
	\label{lqmf.mfg.inequality4.weyl}
	\sigma_{min}(AB^{\top}) \geq \sigma_{min}(A)\cdot\sigma_{min}(B), \quad \sigma_{max}(AB^{\top}) \leq \sigma_{max}(A)\cdot\sigma_{max}(B).
\end{align}
And by \cite{tao2010eigenvalues} and \cite{golub2013matrix}, we further have the Weyl's inequalities for singular values:
\begin{align}
	\label{lqmf.mfg.inequality5.weyl}
	\sigma_{max}(A+B) \leq \sigma_{max}(A) + \sigma_{max}(B), \quad \text{and} \quad \sigma_{min}(A+B) \leq \sigma_{max}(A) + \sigma_{min}(B).
\end{align}


\section{Proofs of Statements in Section \ref{lqmf.section.emfg}} \label{lqmf.section.appendixA1}

\subsection{Proof of Theorem \ref{lqmf.mfg.theorem_smp}}

\begin{proof}
	We take the G{\^a}teaux derivative of $J$ with respect to the control $v$:
	\begin{align*}
	\frac{d J(v+\theta\widetilde{v})}{d\theta}\Bigg|_{\theta=0} =& \mathbb{E}\left[ \int_0^T \widetilde{x}_t^{\top} Q_t x_t+\widetilde{v}_t^{\top} P_t v_t+\widetilde{x}_t^{\top} \bar{Q}_t\left(x_t-S_t z_t\right) \right.\\
	&\qquad \quad+\left. \widetilde{x}_t^{\top} N_t\left(v_t-\bar{R}_t w_t\right) + \widetilde{v}_t^{\top} N_t^{\top} \left(x_t-\bar{S}_t z_t\right) + \widetilde{v}_t^{\top} \bar{P}_t\left(v_t-R_t w_t\right) d t\right] \\
	&+\mathbb{E}\left[ \widetilde{x}_T^{\top} Q_T x_T+\widetilde{x}_T^{\top} \bar{Q}_T\left(x_T-S_T z_T\right)\right]
	\end{align*}
	where we define $\widetilde{x}_t := \lim_{\theta \to 0}\frac{ x^{v+\theta\widetilde{v}}_t - x^{v}_t}{\theta}$ which
	satisfies $d \widetilde{x}_t=\left(A_t \widetilde{x}_t+B_t \widetilde{v}_t\right) d t$, $\widetilde{x}(0)=0$. Since $J$ takes its minimum at $(\hat{v},\hat{x})$, it holds that, for any $\widetilde{v}_t\in L_{\mathcal{G}_{\cdot}}^{2}\left(0, T ; \mathbb{R}^m\right)$, $\frac{d J(\hat{v}+\theta\widetilde{v})}{\theta}\bigg|_{\theta=0} \geq 0$,
	which is the Euler condition:
	\begin{align}
	\nonumber 0 \leq \quad&\mathbb{E}\left[ \int_0^T \widetilde{x}_t^{\top} (Q_t+\bar{Q}_t) \hat{x}_t+\widetilde{v}_t^{\top} (P_t+\bar{P}_t) \hat{v}_t-\widetilde{x}_t^{\top} \bar{Q}_tS_t z_t \right.\\
	\nonumber &\qquad \quad+\left. \widetilde{x}_t^{\top} N_t \hat{v}_t- \widetilde{x}_t^{\top} N_t \bar{R}_t w_t + \widetilde{v}_t^{\top} N_t^{\top} \hat{x}_t - \widetilde{v}_t^{\top} N_t^{\top}\bar{S}_t z_t - \widetilde{v}_t^{\top} \bar{P}_t R_t w_t d t\right] \\
	\label{lqmf.mfg.Euler_condition} &+\mathbb{E}\left[ \widetilde{x}_T^{\top} (Q_T+\bar{Q}_T) \hat{x}_T-\widetilde{x}_T^{\top} \bar{Q}_TS_T z_T\right], \quad \text{for any} \quad \widetilde{v}_t\in L_{\mathcal{G}_{\cdot}}^{2}\left(0, T ; \mathbb{R}^m\right).
	\end{align}
	By the chain rule,
	\begin{align*}
	d \left(\widetilde{x}_t^{\top} p_t\right) 
	=& - \left(\widetilde{x}_t^{\top}A_t^{\top}p_t +\widetilde{x}_t^{\top}(Q_t+\bar{Q}_t)\hat{x}_t - \widetilde{x}_t^{\top}\bar{Q}_t S_t z_t + \widetilde{x}_t^{\top} N_t \hat{v}_t - \widetilde{x}_t^{\top} N_t \bar{R}_t w_t\right)dt \\
    &+ \left( \widetilde{x}_t^{\top}A_t^{\top}p_t+ \widetilde{v}_t^{\top}B_t^{\top}p_t\right) d t + \widetilde{x}_t^{\top} \theta_t dW_t.
	\end{align*}
	Taking integration and expectation of both sides, we get
	\begin{align*}
	\mathbb{E}\left[\widetilde{x}_T^{\top}(Q_T+\bar{Q}_T)\hat{x}_T - \widetilde{x}_T^{\top}\bar{Q}_T S_T z_T\right] = \mathbb{E}\left[ \int_0^T- \widetilde{x}_t^{\top}(Q_t+\bar{Q}_t)\hat{x}_t + \widetilde{x}_t^{\top}\bar{Q}_t S_t z_t - \widetilde{x}_t^{\top} N_t \hat{v}_t + \widetilde{x}_t^{\top} N_t \bar{R}_t w_t +   \widetilde{v}_t^{\top}B_t^{\top}p_t d t \right],
	\end{align*}
	which is then substituted back into \eqref{lqmf.mfg.Euler_condition}, we have, for any $\widetilde{v}_t$,
	\begin{align*}
	\mathbb{E}\left[ \int_0^T \widetilde{v}_t^{\top} (P_t+\bar{P}_t) \hat{v}_t + \widetilde{v}_t^{\top} N_t^{\top} \hat{x}_t - \widetilde{v}_t^{\top} N_t^{\top}\bar{S}_t z_t - \widetilde{v}_t^{\top} \bar{P}_t R_t w_t + \widetilde{v}_t^{\top} B_t^{\top} p_t dt \right] \geq 0,
	\end{align*}
	therefore the claim follows by the arbitrariness of $\widetilde{v}_t$.
\end{proof}

\section{Proofs of Statements in Section \ref{lqmf.section.unique_existence}} \label{lqmf.section.appendixA2}
\subsection{Proof of Theorem \ref{lqmf.mfg.theorem.local.solution}}

\begin{proof}
	Consider a Hilbert space $L^2_{\alpha}(0,T;\mathbb{R}^n)$ endowed with the inner product:
	\begin{align*}
	\braket{z,z'} := \int_0^T e^{-\alpha t} z_t^{\top}z'_tdt+e^{-\alpha T} z_T^{\top}z'_T, \quad \text{for any} \quad z, z' \in L^2_{\alpha}(0,T;\mathbb{R}^n).
	\end{align*}
	For any given $z \in L^2_{\alpha}(0,T;\mathbb{R}^n)$, define a pair $(\xi,\eta)$ being a solution of the following forward-backward system:
	\begin{align}
	\label{lqmf.mfg.expected.fbode.decoupled}
     \left\{\begin{array}{cclccl}
	\frac{d\xi_t}{d t} &=& \left(A_t+\bar{A}_t -  (B_t+\bar{B}_t) (P_t+\bar{\mathcal{P}}_t)^{-1}\bar{\mathcal{S}}_t\right)\xi_t  -(B_t+\bar{B}_t) (P_t+\bar{\mathcal{P}}_t)^{-1}B_t^{\top}\eta_t, &\xi(0) &=&\mathbb{E}\left[x_0\right]; \\
	-\frac{d\eta_t}{d t} &=& \left(Q_t+\bar{\mathcal{Q}}_t - \bar{\mathcal{R}}_t(P_t+\bar{\mathcal{P}}_t)^{-1}\bar{\mathcal{S}}_t\right)z_t  -\left(A_t^{\top}-\bar{\mathcal{R}}_t(P_t+\bar{\mathcal{P}}_t)^{-1}B_t^{\top}\right)\eta_t, &\eta_T&=&\left(Q_T+\bar{\mathcal{Q}}_T\right) z_T.
    \end{array}\right.
	\end{align}
	Since the system \eqref{lqmf.mfg.expected.fbode.decoupled} is a decoupled linear ODEs, there exists a unique solution, and the map $z \mapsto \xi$ is clearly a self-map in $L^2_{\alpha}(0,T;\mathbb{R}^n)$. Our objective is to show that it admits a fixed point. To apply the Banach fixed point theorem, we first take two different $z^1$ and $z^2$ in $L^2_{\alpha}(0,T;\mathbb{R}^n)$ and their corresponding solutions $(\xi^1,\eta^1)$ and $(\xi^2,\eta^2)$, then denote by $z:=z^1-z^2$, $\xi:=\xi^1-\xi^2$ and $\eta:=\eta^1-\eta^2$, thus $(\xi,\eta)$ solves for the system \eqref{lqmf.mfg.expected.fbode.decoupled} with $\mathbb{E}\left[x_0\right]$ replaced by $0$.
	For some constants $\alpha > 0$ and $\beta > 0$, we apply the chain rule to $e^{-\alpha t} \xi_t^{\top}\xi_t$, and by Young's inequality, we have
	\begin{align*}
	e^{-\alpha T} \xi_T^{\top}\xi_T=& \int_0^T\left[-\alpha e^{-\alpha t} \xi_t^{\top}\xi_t + 2 e^{-\alpha t} \xi_t^{\top}\frac{d\xi_t}{dt}\right]dt \\
	=& \int_0^T\big[-\alpha e^{-\alpha t} \xi_t^{\top}\xi_t \\
	&\quad\quad+ 2 e^{-\alpha t} \xi_t^{\top}\left(A_t+\bar{A}_t -  (B_t+\bar{B}_t) (P_t+\bar{\mathcal{P}}_t)^{-1}\bar{\mathcal{S}}_t\right)\xi_t  -2 e^{-\alpha t} \xi_t^{\top}(B_t+\bar{B}_t) (P_t+\bar{\mathcal{P}}_t)^{-1}B_t^{\top}\eta_t\big]dt \\
	\leq& \int_0^T\left(2\sup_{0\leq t\leq T}\|A_t+\bar{A}_t -  (B_t+\bar{B}_t) (P_t+\bar{\mathcal{P}}_t)^{-1}\bar{\mathcal{S}}_t\|-\alpha\right) e^{-\alpha t} \xi_t^{\top}\xi_tdt \\
	&-\int_0^T2 e^{-\alpha t} \xi_t^{\top}(B_t+\bar{B}_t) (P_t+\bar{\mathcal{P}}_t)^{-1}B_t^{\top}\eta_tdt \\
	\leq& \int_0^T\left(2\sup_{0\leq t\leq T}\|A_t+\bar{A}_t -  (B_t+\bar{B}_t) (P_t+\bar{\mathcal{P}}_t)^{-1}\bar{\mathcal{S}}_t\|-\alpha\right) e^{-\alpha t} \xi_t^{\top}\xi_tdt \\
	&+\int_0^T e^{-\alpha t} \left(\sup_{0\leq t\leq T}\|(B_t+\bar{B}_t) (P_t+\bar{\mathcal{P}}_t)^{-1}B_t^{\top}\|^2L\xi_t^{\top}\xi_t + \frac{1}{L}\eta_t^{\top}\eta_t\right) dt,
	\end{align*}
	where $L > 0$ is an arbitrary positive constant.
	If we take $\alpha := 2 + 2\sup_{0\leq t\leq T}\|A_t+\bar{A}_t -  (B_t+\bar{B}_t) (P_t+\bar{\mathcal{P}}_t)^{-1}\bar{\mathcal{S}}_t\| + \sup_{0\leq t\leq T}\|(B_t+\bar{B}_t) (P_t+\bar{\mathcal{P}}_t)^{-1}B_t^{\top}\|^2L$, then
	\begin{align*}
	\int_0^T e^{-\alpha t} \xi_t^{\top}\xi_tdt + e^{-\alpha T} \xi_T^{\top}\xi_T \leq& \frac{1}{L}\int_0^T e^{-\alpha t} \eta_t^{\top}\eta_t dt \leq \frac{1}{L}\int_0^T e^{\beta t} \eta_t^{\top}\eta_t dt,
	\end{align*}
	where $\beta$ is a positive constant to be determined. On the other hand, we have
	\begin{align*}
	0 \leq \eta_0^{\top}\eta_0 =& \ e^{\beta T} z_T^{\top}\left(Q_T+\bar{\mathcal{Q}}_T\right)^{\top} \left(Q_T+\bar{\mathcal{Q}}_T\right) z_T - \int_0^T\left[\beta e^{\beta t} \eta_t^{\top}\eta_t + 2 e^{\beta t} \eta_t^{\top}\frac{d\eta_t}{dt}\right]dt \\
	=& \ e^{\beta T} z_T^{\top}\left(Q_T+\bar{\mathcal{Q}}_T\right)^{\top} \left(Q_T+\bar{\mathcal{Q}}_T\right) z_T - \int_0^T\beta e^{\beta t} \eta_t^{\top}\eta_t dt \\
	&+ \int_0^T 2 e^{\beta t} \eta_t^{\top}\left(Q_t+\bar{\mathcal{Q}}_t - \bar{\mathcal{R}}_t(P_t+\bar{\mathcal{P}}_t)^{-1}\bar{\mathcal{S}}_t\right)z_tdt  -\int_0^T2 e^{\beta t} \eta_t^{\top}\left(A_t^{\top}-\bar{\mathcal{R}}_t(P_t+\bar{\mathcal{P}}_t)^{-1}B_t^{\top}\right)\eta_tdt  \\
	\leq& \ e^{(\alpha+\beta) T}\|Q_T+\bar{\mathcal{Q}}_T\|^2e^{-\alpha T} z_T^{\top}z_T - \int_0^T\left(\beta-2\sup_{0 \leq t \leq T}\|A_t^{\top}-\bar{\mathcal{R}}_t(P_t+\bar{\mathcal{P}}_t)^{-1}B_t^{\top}\|\right) e^{\beta t} \eta_t^{\top}\eta_t dt \\
	&+ \int_0^T 2 e^{\beta t} \eta_t^{\top}\left(Q_t+\bar{\mathcal{Q}}_t - \bar{\mathcal{R}}_t(P_t+\bar{\mathcal{P}}_t)^{-1}\bar{\mathcal{S}}_t\right)z_tdt \\
	\leq& \ e^{(\alpha+\beta) T}\|Q_T+\bar{\mathcal{Q}}_T\|^2e^{-\alpha T} z_T^{\top}z_T - \int_0^T\left(\beta-2\sup_{0 \leq t \leq T}\|A_t^{\top}-\bar{\mathcal{R}}_t(P_t+\bar{\mathcal{P}}_t)^{-1}B_t^{\top}\|\right) e^{\beta t} \eta_t^{\top}\eta_t dt \\
	&+ \int_0^T e^{\beta t} \left(\sup_{0\leq t\leq T}\|Q_t+\bar{\mathcal{Q}}_t - \bar{\mathcal{R}}_t(P_t+\bar{\mathcal{P}}_t)^{-1}\bar{\mathcal{S}}_t\|^2\eta_t^{\top}\eta_t + z_t^{\top}z_t\right)dt.
	\end{align*}
	If we take $\beta := 2 + 2\sup_{0 \leq t \leq T}\|A_t^{\top}-\bar{\mathcal{R}}_t(P_t+\bar{\mathcal{P}}_t)^{-1}B_t^{\top}\|+\sup_{0\leq t\leq T}\|Q_t+\bar{\mathcal{Q}}_t - \bar{\mathcal{R}}_t(P_t+\bar{\mathcal{P}}_t)^{-1}\bar{\mathcal{S}}_t\|^2$, then
	\begin{align*}
	\int_0^T e^{\beta t} \eta_t^{\top}\eta_t dt \leq& \ e^{(\alpha+\beta)T} \int_0^T e^{-\alpha t} z_t^{\top}z_tdt + e^{(\alpha+\beta) T}\|Q_T+\bar{\mathcal{Q}}_T\|^2e^{-\alpha T} z_T^{\top}z_T \\
	\leq& \ e^{(\alpha+\beta)T} \max\{1,\|Q_T+\bar{\mathcal{Q}}_T\|^2\}\left(\int_0^T e^{-\alpha t} z_t^{\top}z_tdt+e^{-\alpha T} z_T^{\top}z_T\right).
	\end{align*}
	Therefore, we can choose $L := 2\max\{1,\|Q_T+\bar{\mathcal{Q}}_T\|^2\}$, by then we obtain
	\begin{align*}
	\int_0^T e^{-\alpha t} \xi_t^{\top}\xi_tdt + e^{-\alpha T} \xi_T^{\top}\xi_T \leq& \ \frac{1}{L} \int_0^T e^{\beta t} \eta_t^{\top}\eta_t dt \\
	\leq& \ e^{(\alpha+\beta)T}\frac{1}{L} \max\{1,\|Q_T+\bar{\mathcal{Q}}_T\|^2\} \left(\int_0^T e^{-\alpha t} z_t^{\top}z_tdt+e^{-\alpha T} z_T^{\top}z_T\right) \\
	<& \int_0^T e^{-\alpha t} z_t^{\top}z_tdt+e^{-\alpha T} z_T^{\top}z_T,
	\end{align*}
	which shows that $z \mapsto \xi$ is a contractive map under the condition \eqref{lqmf.mfg.condition.local.solution}.
\end{proof}

\section{Some Useful Lemmas} \label{lqmf.section.appendixA3}

\subsection{Lemma \ref{lqmf.mfg.lemma.boundedness.x-Ex} and its proof}
\begin{lemma}
	\label{lqmf.mfg.lemma.boundedness.x-Ex}
	If the coeffecients satisfy the assumption \eqref{lqmf.mfg.assumption.convexity}, then
	\begin{align}
	\label{lqmf.mfg.estimate.x-Ex}
	\mathbb{E}\left[\left\|\mathbb{E}\left[\hat{x}^i_T\right]-\hat{x}^i_T\right\|^2+\int_0^T  \left\|\mathbb{E}\left[\hat{x}^i_t\right]-\hat{x}^i_t\right\|^2+  \left\|\mathbb{E}\left[\hat{v}^i_t\right]-\hat{v}^i_t\right\|^2dt\right] \leq M_{e,1}\left(\mathbb{E}\left[\left\|\mathbb{E}\left[x^i_0\right]-x^i_0\right\|^2\right]+T\|\sigma\|_T^2\right).
	\end{align}
	where the coefficient $M_{e,1}$ is given by
	\begin{align*}
	M_{e,1}:=\left(1+\frac{C_1^2}{C_2^2}\right)\left(1-\frac{\delta}{\sup_{0\leq t\leq T}\lambda_{max}\{Q_t+\bar{Q}_t\}}\right)^{-1}\min\left\{\inf_{0\leq t\leq T}\lambda_{min}\left(Q_t+\bar{Q}_t\right),\inf_{0\leq t\leq T}\lambda_{min}\left(P_t+\bar{P}_t\right)\right\}^{-1}.
	\end{align*}
\end{lemma}

\begin{proof}
	Let $\hat{x}^{b,i}_t := \mathbb{E}\left[\hat{x}^i_t\right]-\hat{x}^i_t$, $p^{b,i}_t := \mathbb{E}\left[p^i_t\right]-p^i_t$, $\theta^{b,i}_s:=\mathbb{E}\left[\theta^i_t\right]-\theta^i_t$ and $\hat{v}^{b,i}_t := \mathbb{E}\left[\hat{v}^i_t\right]-\hat{v}^i_t$. By \eqref{lqmf.mfg.system.general} and \eqref{lqmf.mfg.mf_equilibrium.general}, we have
	\begin{align}
	\label{lqmf.mfg.system.bias}
	\left\{\begin{array}{cclccl}
	d\hat{x}^{b,i}_t &=&  \left(A_t\hat{x}^{b,i}_t+B_t\hat{v}^{b,i}_t\right)dt - \sigma_t d W^i_t, \quad &\hat{x}^{b,i}(0) &=& \mathbb{E}\left[x^i_0\right]-x^i_0 := \hat{x}^{b,i}_0; \\
	-d p^{b,i}_t &=& \left(A_t^{\top}p^{b,i}_t +(Q_t+\bar{Q}_t)\hat{x}^{b,i}_t +  N_t \hat{v}^{b,i}_t\right)dt + \theta^{b,i}_t d W^i_t, \quad
	&p^{b,i}_T &=& (Q_T+\bar{Q}_T)\hat{x}^{b,i}_T,
	\end{array}\right.
	\end{align}
	and
	\begin{align}
	\label{lqmf.mfg.control.bias}
	\hat{v}^{b,i}_t = - (P_t+\bar{P}_t)^{-1} \left(N_t^{\top} \hat{x}^{b,i}_t + B_t^{\top} p^{b,i}_t\right).
	\end{align}
	By substituting \eqref{lqmf.mfg.control.bias} into \eqref{lqmf.mfg.system.bias}, we have
	\begin{align}
	\label{lqmf.mfg.system2.bias}
	\left\{\begin{array}{cclccl}
	d\hat{x}^{b,i}_t &=&  \left(\left(A_t- B_t(P_t+\bar{P}_t)^{-1}N_t^{\top}\right) \hat{x}^{b,i}_t - B(P_t+\bar{P}_t)^{-1}B_t^{\top} p^{b,i}_t\right)dt - \sigma_t d W^i_t, &\hat{x}^{b,i}(0) &=& \hat{x}^{b,i}_0;\\
	-d p^{b,i}_t &=& \left(\left(A_t^{\top}-N_t (P_t+\bar{P}_t)^{-1}B_t^{\top}\right)p^{b,i}_t +\left(Q_t+\bar{Q}_t - N_t (P_t+\bar{P}_t)^{-1}N_t^{\top}\right) \hat{x}^{b,i}_t\right)dt + \theta^{b,i}_t d W^i_t, &p^{b,i}_T &=& (Q_T+\bar{Q}_T)\hat{x}^{b,i}_T.
	\end{array}\right.
	\end{align}
	Applying It{\^o}'s Lemma to $\left\|p^{b,i}_s\right\|^2$, we have
	\begin{align}
	\nonumber
	&\mathbb{E}\left[\left\|p^{b,i}_t\right\|^2\right] - \mathbb{E}\left[\left\|p^{b,i}_T\right\|^2\right] + \mathbb{E}\left[\int_t^T\tr\left(\theta^{b,i}_s(\theta^{b,i}_s)^{\top}\right)ds\right] \\
	\nonumber
	&= \int_t^T\mathbb{E}\left[2(p^{b,i}_s)^{\top}\left(A_s^{\top}-N_s (P_s+\bar{P}_s)^{-1}B_s^{\top}\right)p^{b,i}_s +2(p^{b,i}_s)^{\top}\left((Q_s+\bar{Q}_s) - N_s (P_s+\bar{P}_s)^{-1}N_s^{\top}\right) \hat{x}^{b,i}_s\right]ds \\
	\nonumber
	&\leq \int_t^T\mathbb{E}\left[2\|A^{\top}-N (P+\bar{P})^{-1}B^{\top}\|_T\left\|p^{b,i}_s\right\|^2 +\|Q+\bar{Q} - N (P+\bar{P})^{-1}N^{\top}\|_T\left( \left\|p^{b,i}_s\right\|^2+\left\|\hat{x}^{b,i}_s\right\|^2\right)\right]ds \\
	\nonumber 
	&= \left(2\|A^{\top}-N (P+\bar{P})^{-1}B^{\top}\|_T+\|Q+\bar{Q} - N (P+\bar{P})^{-1}N^{\top}\|_T\right)\int_t^T\mathbb{E}\left[\left\|p^{b,i}_s\right\|^2\right]ds \\
	\label{lqmf.mfg.p-Ep.inequality1}
	&\quad+\|Q+\bar{Q} - N (P+\bar{P})^{-1}N^{\top}\|_T\int_t^T\mathbb{E}\left[\left\|\hat{x}^{b,i}_s\right\|^2\right]ds.
	\end{align}
	Substituting the terminal value of \eqref{lqmf.mfg.system.bias} into \eqref{lqmf.mfg.p-Ep.inequality1}, we have
	\begin{align}
	\nonumber
	&\mathbb{E}\left[\left\|p^{b,i}_t\right\|^2\right] + \mathbb{E}\left[\int_t^T\tr\left(\theta^{b,i}_s(\theta^{b,i}_s)^{\top}\right)ds\right] \\
	\nonumber
	&\leq \left(2\|A^{\top}-N (P+\bar{P})^{-1}B^{\top}\|_T+\|Q+\bar{Q} - N (P+\bar{P})^{-1}N^{\top}\|_T\right)\int_t^T\mathbb{E}\left[\left\|p^{b,i}_s\right\|^2\right]ds \\
	\nonumber
	&\quad+\|Q+\bar{Q}\|_T^2\mathbb{E}\left[\left\|\hat{x}^{b,i}_T\right\|^2\right] + \|Q+\bar{Q} - N (P+\bar{P})^{-1}N^{\top}\|_T\int_t^T\mathbb{E}\left[\left\|\hat{x}^{b,i}_s\right\|^2\right]ds \\
	\label{lqmf.mfg.p-Ep.inequality2}
	&\leq C_1 \left(\mathbb{E}\left[\left\|\hat{x}^{b,i}_T\right\|^2\right]+\int_t^T\mathbb{E}\left[\left\|\hat{x}^{b,i}_s\right\|^2\right]+\int_t^T\mathbb{E}\left[\left\|p^{b,i}_s\right\|^2\right]ds\right),
	\end{align}
	where
	\begin{align*}
	C_1:=&2\|A^{\top}-N (P+\bar{P})^{-1}B^{\top}\|_T+\|Q+\bar{Q} - N (P+\bar{P})^{-1}N^{\top}\|_T +\|Q+\bar{Q}\|_T^2.
	\end{align*}
	On the other hand, we have
	\begin{align}
	\nonumber
	&(\hat{x}^{b,i}_T)^{\top}(Q_T+\bar{Q}_T)\hat{x}^{b,i}_T - (\hat{x}^{b,i}_0)^{\top}p^{b,i}_0 \\
	\nonumber
	&= \int_0^T  (\hat{x}^{b,i}_t)^{\top}dp^{b,i}_t + \int_0^T  (p^{b,i}_t)^{\top}d\hat{x}^{b,i}_t + \int_0^T \tr\left(\sigma_t (\theta^{b,i}_t)^{\top}\right) dt \\
	\nonumber
	&= -\int_0^T  (\hat{x}^{b,i}_t)^{\top}\left(\left(A_t^{\top}-N_t (P_t+\bar{P}_t)^{-1}B_t^{\top}\right)p^{b,i}_t +\left(Q_t+\bar{Q}_t - N_t (P_t+\bar{P}_t)^{-1}N_t^{\top}\right) \hat{x}^{b,i}_t\right)dt \\
	\nonumber
	&\quad+ \int_0^T  (p^{b,i}_t)^{\top}\left(\left(A_t- B_t(P_t+\bar{P}_t)^{-1}N_t^{\top}\right) \hat{x}^{b,i}_t - B_t(P_t+\bar{P}_t)^{-1}B_t^{\top} p^{b,i}_t\right)dt + \int_0^T \tr\left(\sigma_t (\theta^{b,i}_t)^{\top}\right) dt \\
	\nonumber
	&\quad- \int_0^T  (\hat{x}^{b,i}_t)^{\top}\theta_t d W^i_t- \int_0^T(p^{b,i}_t)^{\top}\sigma_t d W^i_t \\
	\nonumber
	&= -\int_0^T  (\hat{x}^{b,i}_t)^{\top}\left(Q_t+\bar{Q}_t - N_t (P_t+\bar{P}_t)^{-1}N_t^{\top}\right) \hat{x}^{b,i}_tdt - \int_0^T  (p^{b,i}_t)^{\top} B_t(P_t+\bar{P}_t)^{-1}B_t^{\top} p^{b,i}_tdt \\
	\label{lqmf.mfg.dual_relation.bias}
	&\quad+ \int_0^T \tr\left(\sigma_t (\theta^{b,i}_t)^{\top}\right) dt - \int_0^T  (\hat{x}^{b,i}_t)^{\top}\theta_t d W^i_t- \int_0^T(p^{b,i}_t)^{\top}\sigma_t d W^i_t.
	\end{align}
	Taking expectation of both sides and using the assumption \eqref{lqmf.mfg.assumption.convexity}, we have
	\begin{align*}
	&\mathbb{E}\left[(\hat{x}^{b,i}_0)^{\top}p^{b,i}_0\right] + \mathbb{E}\left[\int_0^T \tr\left(\sigma_t (\theta^{b,i}_t)^{\top}\right) dt\right] = \mathbb{E}\left[(\hat{x}^{b,i}_T)^{\top}(Q_T+\bar{Q}_T)\hat{x}^{b,i}_T\right] \\
	&+\mathbb{E}\left[\int_0^T  (\hat{x}^{b,i}_t)^{\top}\left(Q_t+\bar{Q}_t - N_t (P_t+\bar{P}_t)^{-1}N_t^{\top}\right) \hat{x}^{b,i}_tdt\right] + \mathbb{E}\left[\int_0^T  (p^{b,i}_t)^{\top} B_t(P_t+\bar{P}_t)^{-1}B_t^{\top} p^{b,i}_tdt\right] \\
	&\geq \lambda_{min}\left(Q_T+\bar{Q}_T\right)\mathbb{E}\left[\left\|\hat{x}^{b,i}_T\right\|^2\right] + \inf_{0\leq t\leq T}\lambda_{min}\left(Q_t+\bar{Q}_t - N_t (P_t+\bar{P}_t)^{-1}N_t^{\top}\right)\mathbb{E}\left[\int_0^T\left\|\hat{x}^{b,i}_t\right\|^2dt\right] \\
	&\quad+ \inf_{0\leq t\leq T}\lambda_{min}\left(B_t(P_t+\bar{P}_t)^{-1}B_t^{\top}\right)\mathbb{E}\left[\int_0^T\left\|p^{b,i}_t\right\|^2dt\right] \\
	&\geq C_2\left(\mathbb{E}\left[\left\|\hat{x}^{b,i}_T\right\|^2\right]+\mathbb{E}\left[\int_0^T\left\|\hat{x}^{b,i}_t\right\|^2dt\right]+\mathbb{E}\left[\int_0^T\left\|p^{b,i}_t\right\|^2dt\right]\right),
	\end{align*}
	where
	\begin{align*}
	C_2:=
	\min\left\{\lambda_{min}\left(Q_T+\bar{Q}_T\right),\inf_{0\leq t\leq T}\lambda_{min}\left(Q_t+\bar{Q}_t - N_t (P_t+\bar{P}_t)^{-1}N_t^{\top}\right),\inf_{0\leq t\leq T}\lambda_{min}\left(B_t(P_t+\bar{P}_t)^{-1}B_t^{\top}\right)\right\},
	\end{align*}
	hence by Young's inequality, we have
	\begin{align}
	\nonumber
	&\mathbb{E}\left[\left\|\hat{x}^{b,i}_T\right\|^2\right]+\mathbb{E}\left[\int_0^T\left\|\hat{x}^{b,i}_t\right\|^2dt\right]+\mathbb{E}\left[\int_0^T\left\|p^{b,i}_t\right\|^2dt\right] \\
	\label{lqmf.mfg.p-Ep.inequality3}
	&\leq \frac{1}{2C_2}\left(L\mathbb{E}\left[\left\|\hat{x}^{b,i}_0\right\|^2\right] + \frac{1}{L}\mathbb{E}\left[\left\|p^{b,i}_0\right\|^2\right] + LT\|\sigma\|_T^2+\frac{1}{L}\mathbb{E}\left[\int_0^T \tr\left(\theta^{b,i}_t (\theta^{b,i}_t)^{\top}\right) dt\right]\right).
	\end{align}
	Pluging \eqref{lqmf.mfg.p-Ep.inequality3} into \eqref{lqmf.mfg.p-Ep.inequality2}, we have
	\begin{align*}
	&\mathbb{E}\left[\left\|p^{b,i}_t\right\|^2\right] + \mathbb{E}\left[\int_t^T\tr\left(\theta^{b,i}_s(\theta^{b,i}_s)^{\top}\right)ds\right] \\
	&\leq  \frac{C_1}{2C_2}\left(L\mathbb{E}\left[\left\|\hat{x}^{b,i}_0\right\|^2\right] + \frac{1}{L}\mathbb{E}\left[\left\|p^{b,i}_0\right\|^2\right] + LT\|\sigma\|_T^2+\frac{1}{L}\mathbb{E}\left[\int_0^T \tr\left(\theta^{b,i}_t (\theta^{b,i}_t)^{\top}\right) dt\right]\right).
	\end{align*}
	By setting $t:=0$ and $L:=\frac{C_1}{C_2}$, then
	\begin{align}
	\label{lqmf.mfg.p-Ep.inequality4}
	&\mathbb{E}\left[\left\|p^{b,i}_0\right\|^2\right] + \mathbb{E}\left[\int_0^T\tr\left(\theta^{b,i}_s(\theta^{b,i}_s)^{\top}\right)ds\right] 
	\leq  \frac{C_1^2}{C_2^2} \cdot \left(\mathbb{E}\left[\left\|\hat{x}^{b,i}_0\right\|^2\right]+T\|\sigma\|_T^2\right).
	\end{align}
	In order to obtain \eqref{lqmf.mfg.estimate.x-Ex}, we substitute \eqref{lqmf.mfg.control.bias} into \eqref{lqmf.mfg.dual_relation.bias}, then
	\begin{align}
	\nonumber
	&(\hat{x}^{b,i}_0)^{\top}p^{b,i}_0 + \int_0^T \tr\left(\sigma_t (\theta^{b,i}_t)^{\top}\right) dt \\
	\nonumber
	&= (\hat{x}^{b,i}_T)^{\top}(Q_T+\bar{Q}_T)\hat{x}^{b,i}_T + \int_0^T  (\hat{x}^{b,i}_t)^{\top}\left(Q_t+\bar{Q}_t\right) \hat{x}^{b,i}_tdt + \int_0^T  (\hat{v}^{b,i}_t)^{\top} (P_t+\bar{P}_t) \hat{v}^{b,i}_tdt \\
	\label{lqmf.mfg.dual_relation.bias.xv}
	&\quad + 2\int_0^T  (\hat{v}^{b,i}_t)^{\top} N_t^{\top} \hat{x}^{b,i}_tdt + \int_0^T  (\hat{x}^{b,i}_t)^{\top}\theta_t d W^i_t+ \int_0^T(p^{b,i}_t)^{\top}\sigma_t d W^i_t.
	\end{align}
	Since there exists a scale constant $\delta > 0$ such that $Q_t+\bar{Q}_t  -  N_t (P_t+\bar{P}_t)^{-1} N_t^{\top} > \delta I$ for all $t \in [0,T]$, we have
	\begin{align*}
	I  -  (Q_t+\bar{Q}_t)^{-\frac{1}{2}}N_t (P_t+\bar{P}_t)^{-1} N_t^{\top}(Q_t+\bar{Q}_t)^{-\frac{1}{2}} > \delta(Q_t+\bar{Q}_t)^{-1}.
	\end{align*}
	Denote $\mathcal{N}_t:=(Q_t+\bar{Q}_t)^{-\frac{1}{2}}N_t (P_t+\bar{P}_t)^{-\frac{1}{2}}$ and $\delta_1:=\delta\inf_{0\leq t\leq T}\lambda_{min}\{(Q_t+\bar{Q}_t)^{-1}\}=\frac{\delta}{\sup_{0\leq t\leq T}\lambda_{max}\{Q_t+\bar{Q}_t\}}$.
	By Weyl's inequality \eqref{lqmf.mfg.inequality1.weyl}, we have,
	\begin{align*}
	1 > \lambda_{max}(\mathcal{N}_t\mathcal{N}_t^{\top}+\delta(Q_t+\bar{Q}_t)^{-1}) > \lambda_{max}(\mathcal{N}_t\mathcal{N}_t^{\top}) + \delta\lambda_{min}((Q_t+\bar{Q}_t)^{-1}) = \lambda_{max}(\mathcal{N}_t^{\top}\mathcal{N}_t) + \delta\lambda_{min}((Q_t+\bar{Q}_t)^{-1}).
	\end{align*}
	Hence, we obtain
	\begin{align*}
	\mathcal{N}_t\mathcal{N}_t^{\top} < (1-\delta_1)I \quad \text{and} \quad \mathcal{N}_t^{\top}\mathcal{N}_t < (1-\delta_1)I, \quad \text{for all} \quad t \in [0,T].
	\end{align*}
	Note that
	\begin{align*}
	&\int_0^T  (\hat{x}^{b,i}_t)^{\top}N_t (P_t+\bar{P}_t)^{-1} N_t^{\top} \hat{x}^{b,i}_tdt + \int_0^T  (\hat{v}^{b,i}_t)^{\top} (P_t+\bar{P}_t) \hat{v}^{b,i}_tdt + 2\int_0^T  (\hat{v}^{b,i}_t)^{\top} N_t^{\top} \hat{x}^{b,i}_tdt \\
	&=\int_0^T  \left\|(P_t+\bar{P}_t)^{\frac{1}{2}} \hat{v}^{b,i}_t+(P_t+\bar{P}_t)^{-\frac{1}{2}} N_t^{\top} \hat{x}^{b,i}_t\right\|^2dt \geq 0,
	\end{align*}
	and
	\begin{align*}
	&\int_0^T  (\hat{x}^{b,i}_t)^{\top}\left(Q_t+\bar{Q}_t\right) \hat{x}^{b,i}_tdt + \int_0^T  (\hat{v}^{b,i}_t)^{\top} (P_t+\bar{P}_t)^{\frac{1}{2}} \mathcal{N}_t^{\top}\mathcal{N}_t (P_t+\bar{P}_t)^{\frac{1}{2}} \hat{v}^{b,i}_tdt + 2\int_0^T  (\hat{v}^{b,i}_t)^{\top} N_t^{\top} \hat{x}^{b,i}_tdt \\
	&= \int_0^T \left\|\left(Q_t+\bar{Q}_t\right)^{\frac{1}{2}} \hat{x}^{b,i}_t+\mathcal{N}_t (P_t+\bar{P}_t)^{\frac{1}{2}} \hat{v}^{b,i}_t\right\|^2 \geq 0.
	\end{align*}
	Hence
	\begin{align*}
	&2\int_0^T  (\hat{x}^{b,i}_t)^{\top}\left(Q_t+\bar{Q}_t\right) \hat{x}^{b,i}_tdt + 2\int_0^T  (\hat{v}^{b,i}_t)^{\top} (P_t+\bar{P}_t) \hat{v}^{b,i}_tdt  + 4\int_0^T  (\hat{v}^{b,i}_t)^{\top} N_t^{\top} \hat{x}^{b,i}_tdt \\
	&=\int_0^T  (\hat{x}^{b,i}_t)^{\top}(Q_t+\bar{Q}_t)^{\frac{1}{2}}\left(I  -  \mathcal{N}_t\mathcal{N}_t^{\top}\right) (Q_t+\bar{Q}_t)^{\frac{1}{2}}\hat{x}^{b,i}_tdt + \int_0^T  (\hat{v}^{b,i}_t)^{\top} (P_t+\bar{P}_t)^{\frac{1}{2}} \left(I-\mathcal{N}_t^{\top}\mathcal{N}_t\right) (P_t+\bar{P}_t)^{\frac{1}{2}} \hat{v}^{b,i}_tdt \\
	&\quad+ \int_0^T  (\hat{x}^{b,i}_t)^{\top}N_t (P_t+\bar{P}_t)^{-1} N_t^{\top} \hat{x}^{b,i}_tdt + \int_0^T  (\hat{v}^{b,i}_t)^{\top} (P_t+\bar{P}_t) \hat{v}^{b,i}_tdt + 2\int_0^T  (\hat{v}^{b,i}_t)^{\top} N_t^{\top} \hat{x}^{b,i}_tdt \\
	&\quad+ \int_0^T  (\hat{x}^{b,i}_t)^{\top}\left(Q_t+\bar{Q}_t\right) \hat{x}^{b,i}_tdt + \int_0^T  (\hat{v}^{b,i}_t)^{\top} (P_t+\bar{P}_t)^{\frac{1}{2}} \mathcal{N}_t^{\top}\mathcal{N}_t (P_t+\bar{P}_t)^{\frac{1}{2}} \hat{v}^{b,i}_tdt + 2\int_0^T  (\hat{v}^{b,i}_t)^{\top} N_t^{\top} \hat{x}^{b,i}_tdt \\
	&\geq (1-\delta_1)\int_0^T  (\hat{x}^{b,i}_t)^{\top}(Q_t+\bar{Q}_t)\hat{x}^{b,i}_tdt + (1-\delta_1)\int_0^T  (\hat{v}^{b,i}_t)^{\top} (P_t+\bar{P}_t) \hat{v}^{b,i}_tdt \\
	&\geq \left(1-\frac{\delta}{\sup_{0\leq t\leq T}\lambda_{max}\{Q_t+\bar{Q}_t\}}\right)\min\left\{\inf_{0\leq t\leq T}\lambda_{min}\left(Q_t+\bar{Q}_t\right),\inf_{0\leq t\leq T}\lambda_{min}\left(P_t+\bar{P}_t\right)\right\} \\
	&\quad\cdot \left(\int_0^T  \left\|\hat{x}^{b,i}_t\right\|^2dt+\int_0^T  \left\|\hat{v}^{b,i}_t\right\|^2dt\right).
	\end{align*}
	Take expectation of both sides of \eqref{lqmf.mfg.dual_relation.bias.xv} and by Young's inequality, we have
	\begin{align*}
	&\mathbb{E}\left[\left\|\hat{x}^{b,i}_0\right\|^2\right] + \mathbb{E}\left[\left\|p^{b,i}_0\right\|^2\right] + T\|\sigma\|_T^2+\mathbb{E}\left[\int_0^T \tr\left(\theta^{b,i}_t (\theta^{b,i}_t)^{\top}\right) dt\right] \\
	&\geq 2\mathbb{E}\left[(\hat{x}^{b,i}_0)^{\top}p^{b,i}_0\right] + 2\mathbb{E}\left[\int_0^T \tr\left(\sigma_t (\theta^{b,i}_t)^{\top}\right) dt\right] \\
	&= 2\mathbb{E}\left[(\hat{x}^{b,i}_T)^{\top}(Q_T+\bar{Q}_T)\hat{x}^{b,i}_T\right] \\
	&\quad+ 2\mathbb{E}\left[\int_0^T  (\hat{x}^{b,i}_t)^{\top}\left(Q_t+\bar{Q}_t\right) \hat{x}^{b,i}_tdt\right] + 2\mathbb{E}\left[\int_0^T  (\hat{v}^{b,i}_t)^{\top} (P_t+\bar{P}_t) \hat{v}^{b,i}_tdt\right] + 4\mathbb{E}\left[\int_0^T  (\hat{v}^{b,i}_t)^{\top} N_t^{\top} \hat{x}^{b,i}_tdt\right] \\
	&\geq \left(1-\frac{\delta}{\sup_{0\leq t\leq T}\lambda_{max}\{Q_t+\bar{Q}_t\}}\right)\min\left\{\inf_{0\leq t\leq T}\lambda_{min}\left(Q_t+\bar{Q}_t\right),\inf_{0\leq t\leq T}\lambda_{min}\left(P_t+\bar{P}_t\right)\right\} \\
	&\quad\cdot \left(\mathbb{E}\left[\left\|\hat{x}^{b,i}_T\right\|^2\right] + \int_0^T  \mathbb{E}\left[\left\|\hat{x}^{b,i}_t\right\|^2\right]dt+\int_0^T  \mathbb{E}\left[\left\|\hat{v}^{b,i}_t\right\|^2\right]dt\right).
	\end{align*}
	Therefore, by \eqref{lqmf.mfg.p-Ep.inequality4}, we deduce \eqref{lqmf.mfg.estimate.x-Ex}.
\end{proof}

\subsection{Lemma \ref{lqmf.mfg.lemma.boundedness.Ex2} and its proof}
\begin{lemma}
	\label{lqmf.mfg.lemma.boundedness.Ex2}
	Under the condition of Theorem \ref{lqmf.mfg.theorem.global.uniqueness}, we have
	\begin{align}
	\label{lqmf.mfg.estimate.Ex2}
	&\mathbb{E}\left[\left\|\hat{x}^i_T\right\|^2+\int_0^T  \left\|\hat{x}^i_t\right\|^2+ \left\|\hat{v}^i_t\right\|^2dt\right] 
	\leq M_{e,2} \left\|\mathbb{E}\left[x^i_0\right]\right\|^2 +M_{e,1} \left(\mathbb{E}\left[\left\|\mathbb{E}\left[x^i_0\right]-x^i_0\right\|^2\right]+T\|\sigma\|_T^2\right),
	\end{align}
	where the coefficient $M_{e,2}$ is given by
	\begin{align*}
	M_{e,2}:=&\frac{4C_4}{C_3^2}\cdot \left(1+\|(P+\bar{\mathcal{P}})^{-1}\|_T^2\|\bar{\mathcal{S}}\|_T^2+\|(P+\bar{\mathcal{P}})^{-1}\|_T^2\|B^{\top}\|_T^2\right),
	\end{align*}
	and $M_{e,1}$ was defined in Lemma \ref{lqmf.mfg.lemma.boundedness.x-Ex}.
\end{lemma}

\begin{proof}
	Let $\bar{x}^i_t := \mathbb{E}\left[\hat{x}^i_t\right]$, $\bar{p}^i_t := \mathbb{E}\left[p^i_t\right]$, and $\bar{v}^i_t := \mathbb{E}\left[\hat{v}^i_t\right]$. By \eqref{lqmf.mfg.expected.fbode} and \eqref{lqmf.mfg.expected.optimal_control}, we have
	\begin{align}
	\left\{\begin{array}{cclccl}
	d \bar{x}^i_t&=&\left((A_t+\bar{A}_t -  (B_t+\bar{B}_t) (P_t+\bar{\mathcal{P}}_t)^{-1}\bar{\mathcal{S}}_t) \bar{x}^i_t-(B_t+\bar{B}_t) (P_t+\bar{\mathcal{P}}_t)^{-1}B_t^{\top} \bar{p}^i_t\right) d t, &\bar{x}^i_t(0)&=&\mathbb{E}[x_0],\\
	-d \bar{p}^i_t &=& \left((A_t^{\top}-\bar{\mathcal{R}}_t(P_t+\bar{\mathcal{P}}_t)^{-1}B_t^{\top})\bar{p}^i_t  +(Q_t+\bar{\mathcal{Q}}_t - \bar{\mathcal{R}}_t(P_t+\bar{\mathcal{P}}_t)^{-1}\bar{\mathcal{S}}_t)\bar{x}^i_t \right)dt, &\bar{p}^i_T &=& (Q_T+\bar{\mathcal{Q}}_T)\bar{x}^i_T,
	\end{array}\right.
	\end{align}
	and
	\begin{align}
	\label{lqmf.mfg.optimal_control.bar}
	\bar{v}^i_t =  - (P_t+\bar{\mathcal{P}}_t)^{-1} \left(\bar{\mathcal{S}}_t \bar{x}^i_t  + B_t^{\top} \bar{p}^i_t\right).
	\end{align}
	By the similar argument as in Theorem \ref{lqmf.mfg.theorem.global.uniqueness}, for any $\varepsilon >0$ satisfying \eqref{lqmf.mfg.epsilon.choose},  we have
	\begin{align}
	\label{lqmf.mfg.dual_relation.bar}
	&\left(\bar{x}^i_T\right)^{\top}(Q_T+\bar{\mathcal{Q}}_T)\bar{x}^i_T + \left(K_1 - K_5 - \frac{K_3+K_4}{2 \varepsilon} \right)\int_0^T \left\|\bar{x}^i_t\right\|^2dt+\left(K_2-\varepsilon\right)\int_0^T\left\|\bar{p}^i_t\right\|^2dt = \left(\bar{x}^i_0\right)^{\top}\bar{p}^i_0 \leq \frac{1}{2}\left(\left\|\bar{x}^i_0\right\|^2+\left\|\bar{p}^i_0\right\|^2\right).
	\end{align}
	We choose $\varepsilon := \frac{K_2-K_1+K_5+\sqrt{(K_1-K_5-K_2)^2+2(K_3+K_4)}}{2}$, then
	\begin{align}
	\nonumber
	\frac{L}{2}\left\|\bar{x}^i_0\right\|^2+\frac{1}{2L}\left\|\bar{p}^i_0\right\|^2	\geq& \left(\bar{x}^i_T\right)^{\top}(Q_T+\bar{\mathcal{Q}}_T)\bar{x}^i_T + \left(K_1 - K_5 - \frac{K_3+K_4}{2 \varepsilon} \right)\int_0^T \left\|\bar{x}^i_t\right\|^2dt+\left(K_2-\varepsilon\right)\int_0^T\left\|\bar{p}^i_t\right\|^2dt \\
	\label{lqmf.mfg.relationship.Ex_Ep}
	\geq& \frac{1}{2}C_3\left(\left\|\bar{x}^i_T\right\|^2 + \int_0^T \left\|\bar{x}^i_t\right\|^2dt+\int_0^T\left\|\bar{p}^i_t\right\|^2dt\right),
	\end{align}
	where $C_3:= 3K_1 +K_2-K_5+\sqrt{(K_1-K_5-K_2)^2+2(K_3+K_4)}$.
	By chain rule, we have,
	\begin{align}
	\nonumber
	\|\bar{p}^i_t\|^2-\|\bar{p}^i_T\|^2 =& -2\int_t^T\bar{p}^i_sd\bar{p}^i_s \\
	\nonumber
	&= 2\int_t^T(\bar{p}^i_s)^{\top}\left((A_s^{\top}-\bar{\mathcal{R}}_s(P_s+\bar{\mathcal{P}}_s)^{-1}B_s^{\top})\bar{p}^i_s  +(Q_s+\bar{\mathcal{Q}}_s - \bar{\mathcal{R}}_s(P_s+\bar{\mathcal{P}}_s)^{-1}\bar{\mathcal{S}}_s)\bar{x}^i_s \right)ds \\
	\label{lqmf.mfg.ito.ep2}
	&\leq 2\|A^{\top}-\bar{\mathcal{R}}(P+\bar{\mathcal{P}})^{-1}B^{\top}\|_T\int_t^T\|\bar{p}^i_s\|^2ds  +\|Q+\bar{\mathcal{Q}} - \bar{\mathcal{R}}(P+\bar{\mathcal{P}})^{-1}\bar{\mathcal{S}}\|\int_t^T(\|\bar{p}^i_s\|^2+\|\bar{x}^i_s\|^2) ds.
	\end{align}
	Substituting \eqref{lqmf.mfg.relationship.Ex_Ep} into \eqref{lqmf.mfg.ito.ep2} and let $t=0$, we have
	\begin{align*}
	\|\bar{p}^i_0\|^2 \leq& \left(\|Q+\bar{\mathcal{Q}}\|_T^2+2\|A^{\top}-\bar{\mathcal{R}}(P+\bar{\mathcal{P}})^{-1}B^{\top}\|_T  +\|Q+\bar{\mathcal{Q}} - \bar{\mathcal{R}}(P+\bar{\mathcal{P}})^{-1}\bar{\mathcal{S}}\|\right) \left(\left\|\bar{x}^i_T\right\|^2 + \int_0^T \left\|\bar{x}^i_t\right\|^2dt+\int_0^T\left\|\bar{p}^i_t\right\|^2dt\right) \\
	\leq& \frac{C_4}{C_3}\left(L\left\|\bar{x}^i_0\right\|^2+\frac{1}{L}\left\|\bar{p}^i_0\right\|^2\right),
	\end{align*}
	where $C_4;=\|Q+\bar{\mathcal{Q}}\|_T^2+2\|A^{\top}-\bar{\mathcal{R}}(P+\bar{\mathcal{P}})^{-1}B^{\top}\|_T  +\|Q+\bar{\mathcal{Q}} - \bar{\mathcal{R}}(P+\bar{\mathcal{P}})^{-1}\bar{\mathcal{S}}\|$. We set $L:=\frac{2C_4}{C_3}$, hence
	\begin{align}
	\label{lqmf.mfg.estimtate.Ep0}
	\|\bar{p}^i_0\|^2 
	\leq \frac{4C_4^2}{C_3^2}\left\|\bar{x}^i_0\right\|^2.
	\end{align}
	Substitute \eqref{lqmf.mfg.estimtate.Ep0} into the left hand side of \eqref{lqmf.mfg.relationship.Ex_Ep},
	\begin{align}
	\nonumber
	\left(\left\|\bar{x}^i_T\right\|^2 + \int_0^T \left\|\bar{x}^i_t\right\|^2dt+\int_0^T\left\|\bar{p}^i_t\right\|^2dt\right) \leq& \frac{L}{C_3}\left\|\bar{x}^i_0\right\|^2+\frac{1}{C_3L}\left\|\bar{p}^i_0\right\|^2 \\
	\nonumber
	\leq& \left(\frac{L}{C_3} + \frac{4C_4^2}{C_3^3L}\right)\left\|\bar{x}^i_0\right\|^2 \\
	=& \frac{4C_4}{C_3^2}\left\|\bar{x}^i_0\right\|^2.
	\end{align}
	Note that by \eqref{lqmf.mfg.optimal_control.bar}, we have
	\begin{align}
	\left\|\bar{v}^i_t\right\| \leq \|(P+\bar{\mathcal{P}})^{-1}\|_T\left(\|\bar{\mathcal{S}}\|_T \|\bar{x}^i_t\|  + \|B^{\top}\|_T \|\bar{p}^i_t\|\right),
	\end{align}
	then
	\begin{align}
	\nonumber
	&\left\|\bar{x}^i_T\right\|^2 + \int_0^T \left\|\bar{x}^i_t\right\|^2dt+\int_0^T\left\|\bar{v}^i_t\right\|^2dt \leq \left\|\bar{x}^i_T\right\|^2 + \int_0^T \left\|\bar{x}^i_t\right\|^2dt+\int_0^T\|(P+\bar{\mathcal{P}})^{-1}\|_T^2\left(\|\bar{\mathcal{S}}\|_T^2 \|\bar{x}^i_t\|^2  + \|B^{\top}\|_T^2 \|\bar{p}^i_t\|^2\right)dt \\
	\nonumber
	&\leq \left(1+\|(P+\bar{\mathcal{P}})^{-1}\|_T^2\|\bar{\mathcal{S}}\|_T^2+\|(P+\bar{\mathcal{P}})^{-1}\|_T^2\|B^{\top}\|_T^2\right)\left(\left\|\bar{x}^i_T\right\|^2 + \int_0^T \left\|\bar{x}^i_t\right\|^2dt+\int_0^T\left\|\bar{p}^i_t\right\|^2dt\right) \\
	\label{lqmf.mfg.estimate.(Ex)2}
	&\leq \frac{4C_4}{C_3^2}\cdot \left(1+\|(P+\bar{\mathcal{P}})^{-1}\|_T^2\|\bar{\mathcal{S}}\|_T^2+\|(P+\bar{\mathcal{P}})^{-1}\|_T^2\|B^{\top}\|_T^2\right) \left\|\bar{x}^i_0\right\|^2.
	\end{align}
	Note that $\mathbb{E}\left[\left\|\hat{x}^i_T\right\|^2\right] = \left\|\bar{x}^i_T\right\|^2+\mathbb{E}\left[\left\|\mathbb{E}\left[\hat{x}^i_T\right]-\hat{x}^i_T\right\|^2\right]$, by \eqref{lqmf.mfg.estimate.x-Ex}, we have
	\begin{align*}
	&\mathbb{E}\left[\left\|\hat{x}^i_T\right\|^2\right] + \int_0^T  \mathbb{E}\left[\left\|\hat{x}^i_t\right\|^2\right]dt+\int_0^T  \mathbb{E}\left[\left\|\hat{v}^i_t\right\|^2\right]dt \\
	&\leq \left\|\bar{x}^i_T\right\|^2 + \int_0^T \left\|\bar{x}^i_t\right\|^2dt+\int_0^T\left\|\bar{v}^i_t\right\|^2dt \\
	&\quad+ \mathbb{E}\left[\left\|\mathbb{E}\left[\hat{x}^i_T\right]-\hat{x}^i_T\right\|^2\right] + \int_0^T  \mathbb{E}\left[\left\|\mathbb{E}\left[\hat{x}^i_t\right]-\hat{x}^i_t\right\|^2\right]dt+\int_0^T  \mathbb{E}\left[\left\|\mathbb{E}\left[\hat{v}^i_t\right]-\hat{v}^i_t\right\|^2\right]dt \\
	&\leq \frac{4C_4}{C_3^2}\cdot \left(1+\|(P+\bar{\mathcal{P}})^{-1}\|_T^2\|\bar{\mathcal{S}}\|_T^2+\|(P+\bar{\mathcal{P}})^{-1}\|_T^2\|B^{\top}\|_T^2\right) \left\|\mathbb{E}\left[x^i_0\right]\right\|^2 \\
	&\quad+\left(1-\frac{\delta}{\sup_{0\leq t\leq T}\lambda_{max}\{Q_t+\bar{Q}_t\}}\right)^{-1}\cdot\left(\min\left\{\inf_{0\leq t\leq T}\lambda_{min}\left(Q_t+\bar{Q}_t\right),\inf_{0\leq t\leq T}\lambda_{min}\left(P_t+\bar{P}_t\right)\right\}\right)^{-1} \\
	&\quad\cdot \left(1+\frac{C_1^2}{C_2^2}\right) \cdot \left(\mathbb{E}\left[\left\|\mathbb{E}\left[x^i_0\right]-x^i_0\right\|^2\right]+T\|\sigma\|_T^2\right).
	\end{align*}
	
\end{proof}

\subsection{Lemma \ref{lqmf.mfg.lemma.boundedness.x-Ex.dummy} and its proof}
\begin{lemma}
	\label{lqmf.mfg.lemma.boundedness.x-Ex.dummy}
	For an arbitrary strategy $v^1_t \in L_{\mathcal{G}_{\cdot}}^{2}\left(0, T ; \mathbb{R}^m\right)$, we have
	\begin{align}
	\label{lqmf.mfg.estimate.x-Ex.dummy}
	\mathbb{E}\left[\left\|\mathbb{E}\left[x^1_t\right]-x^1_t\right\|^2\right]\leq& M_{d,4}\left(\mathbb{E}\left[\left\|\mathbb{E}\left[x^1_0\right]-x^1_0\right\|^2\right]+ \int_0^t\mathbb{E}\left[\left\|\mathbb{E}\left[v^1_s\right]-v^1_s\right\|^2\right] ds +T\|\sigma\|_T^2\right).
	\end{align}
	where
	\begin{align}
	M_{d,4}:=(1+\|B\|_T)\exp\{T(2\|A\|_T+\|B\|_T)\}.
	\end{align}
\end{lemma}

\begin{proof}
	Denote by $x^{b,1}_t := \mathbb{E}\left[x^1_t\right]-x^1_t$, $v^{b,1}_t := \mathbb{E}\left[v^1_t\right]-v^1_t$, by \eqref{lqmf.mfg.state_process.dummy}, we have
	\begin{align*}
	d x^{b,1}_t=&\left(A_t x^{b,1}_t+B_t v^{b,1}_t\right) d t-\sigma_t d W^1_t, \quad x^{b,1}(0)=x^{b,1}_0.
	\end{align*}
	Note that $x^1_0$ is a random variable, $x^{b,1}_0 := \mathbb{E}\left[x^1_0\right]-x^1_0$ is not necessarily vanished. By It{\^o}'s Lemma,
	\begin{align*}
	\mathbb{E}\left[\left\|x^{b,1}_t\right\|^2\right]-\mathbb{E}\left[\left\|x^{b,1}_0\right\|^2\right]=& \ 2\mathbb{E}\left[\int_0^t\left(x^{b,1}_s\right)^{\top}\left(A_s x^{b,1}_s+B_s v^{b,1}_s\right) ds\right] + \int_0^t \|\sigma_s\|^2ds \\
	\leq& \ 2\|A\|_T\mathbb{E}\left[\int_0^t\left\|x^{b,1}_s\right\|^2 ds\right] + \|B\|_T\mathbb{E}\left[\int_0^t\left\|x^{b,1}_s\right\|^2+\left\|v^{b,1}_s\right\|^2 ds\right] + \int_0^t \|\sigma_s\|^2ds.
	\end{align*}
	By Gr{\"o}nwall's inequality,
	\begin{align*}
	\mathbb{E}\left[\left\|x^{b,1}_t\right\|^2\right]\leq& \exp\{T(2\|A\|_T+\|B\|_T)\}\cdot\left(\mathbb{E}\left[\left\|x^{b,1}_0\right\|^2\right]+ \|B\|_T\int_0^t\mathbb{E}\left[\left\|v^{b,1}_s\right\|^2\right] ds + \int_0^t \|\sigma_s\|^2ds\right).
	\end{align*}
\end{proof}

\subsection{Lemma \ref{lqmf.mfg.lemma.boundedness.Ex2.dummy} and its proof}
\begin{lemma}
	\label{lqmf.mfg.lemma.boundedness.Ex2.dummy}
	For an arbitrary strategy $v^1_t \in L_{\mathcal{G}_{\cdot}}^{2}\left(0, T ; \mathbb{R}^m\right)$, we have
	\begin{align*}
	\label{lqmf.mfg.estimate.Ex^2.dummy}
	\mathbb{E}\left[\left\|x^1_t\right\|^2\right]\leq& M_{d,5}\left(\left\|\mathbb{E}\left[x^1_0\right]\right\|^2+\int_0^t\left\|\mathbb{E}\left[v^1_s\right]\right\|^2 ds\right) \\ &+M_{d,4}\left(\mathbb{E}\left[\left\|\mathbb{E}\left[x^1_0\right]-x^1_0\right\|^2\right]+ \int_0^t\mathbb{E}\left[\left\|\mathbb{E}\left[v^1_s\right]-v^1_s\right\|^2\right] ds +T\|\sigma\|_T^2\right).
	\end{align*}
	where
	\begin{align}
	M_{d,5}:=\left(1+\|B\|_T+M_{e,2}\|\bar{A}\|_T+M_{e,2}\|\bar{B}\|_T\right)\exp\{T(2\|A\|_T+\|B\|_T+\|\bar{A}\|_T+\|\bar{B}\|_T)\}.
	\end{align}
	and $M_{d,4}$ was defined in Lemma \ref{lqmf.mfg.lemma.boundedness.x-Ex.dummy}.
\end{lemma}

\begin{proof}
	Let $\tilde{x}^1_t := \mathbb{E}\left[x^1_t\right]$, $\tilde{v}^1_t := \mathbb{E}\left[v^1_t\right]$, by \eqref{lqmf.mfg.state_process.dummy}, we have
	\begin{align*}
	d \tilde{x}^1_t=&\left(A_t \tilde{x}^1_t+B_t \tilde{v}^1_t+\bar{A}_t \mathbb{E}\left[\hat{x}^1_t\right]+\bar{B}_t \mathbb{E}\left[\hat{v}^1_t\right]\right) d t, \quad \tilde{x}^1(0)=\tilde{x}^1_0,
	\end{align*}
	where $\tilde{x}^1_0 := \mathbb{E}\left[x^1_0\right]$. By chain rule,
	\begin{align*}
	\left\|\tilde{x}^1_t\right\|^2-\left\|\tilde{x}^1_0\right\|^2=& \ 2\int_0^t\left(\tilde{x}^1_s\right)^{\top}\left(A_s \tilde{x}^1_s+B_s \tilde{v}^1_s+\bar{A}_s \mathbb{E}\left[\hat{x}^1_s\right]+\bar{B}_s \mathbb{E}\left[\hat{v}^1_s\right]\right) ds \\
	\leq& \ 2\|A\|_T\int_0^t\left\|\tilde{x}^1_s\right\|^2 ds + \|B\|_T\int_0^t\left(\left\|\tilde{x}^1_t\right\|^2+\left\|\tilde{v}^1_s\right\|^2\right) ds \\
	&+ \|\bar{A}\|_T\int_0^t\left(\left\|\tilde{x}^1_t\right\|^2+\left\|\mathbb{E}\left[\hat{x}^1_s\right]\right\|^2\right) ds + \|\bar{B}\|_T\int_0^t\left(\left\|\tilde{x}^1_t\right\|^2+\left\|\mathbb{E}\left[\hat{v}^1_s\right]\right\|^2\right) ds
	\end{align*}
	By Gr{\"o}nwall's inequality,
	\begin{align*}
	\left\|\tilde{x}^1_t\right\|^2\leq \exp\{T(2\|A\|_T+\|B\|_T+\|\bar{A}\|_T+\|\bar{B}\|_T)\} \cdot\left(\left\|\tilde{x}^1_0\right\|^2+ \|B\|_T\int_0^t\left\|\tilde{v}^1_s\right\|^2 ds+\|\bar{A}\|_T\int_0^t\left\|\mathbb{E}\left[\hat{x}^1_s\right]\right\|^2 ds+\|\bar{B}\|_T\int_0^t\left\|\mathbb{E}\left[\hat{v}^1_s\right]\right\|^2 ds\right).
	\end{align*}
	By \eqref{lqmf.mfg.estimate.(Ex)2} in Lemma \ref{lqmf.mfg.lemma.boundedness.Ex2}, we have
	\begin{align*}
	\int_0^t\left\|\mathbb{E}\left[\hat{x}^1_s\right]\right\|^2 ds+\int_0^t\left\|\mathbb{E}\left[\hat{v}^1_s\right]\right\|^2 \leq M_{e,2}\|\mathbb{E}\left[\hat{x}^1_0\right]\|^2,
	\end{align*}
	where $M_{e,2}$ was defined in Lemma \ref{lqmf.mfg.lemma.boundedness.Ex2}. Therefore, we have
	\begin{align*}
	\left\|\tilde{x}^1_t\right\|^2\leq \left(1+\|B\|_T+M_{e,2}\|\bar{A}\|_T+M_{e,2}\|\bar{B}\|_T\right)\exp\{T(2\|A\|_T+\|B\|_T+\|\bar{A}\|_T+\|\bar{B}\|_T)\} \cdot\left(\left\|\tilde{x}^1_0\right\|^2+\int_0^t\left\|\tilde{v}^1_s\right\|^2 ds\right).
	\end{align*}
	In summary,
	\begin{align*}
	\mathbb{E}\left[\left\|x^1_t\right\|^2\right]=&\left\|\tilde{x}^1_t\right\|^2+\mathbb{E}\left[\left\|\mathbb{E}\left[x^1_t\right]-x^1_t\right\|^2\right] \\
	\leq& \left(1+\|B\|_T+M_{e,2}\|\bar{A}\|_T+M_{e,2}\|\bar{B}\|_T\right)\exp\{T(2\|A\|_T+\|B\|_T+\|\bar{A}\|_T+\|\bar{B}\|_T)\} \cdot\left(\left\|\mathbb{E}\left[x^1_0\right]\right\|^2+\int_0^t\left\|\mathbb{E}\left[v^1_s\right]\right\|^2 ds\right) \\
	&+M_{d,4}\left(\mathbb{E}\left[\left\|\mathbb{E}\left[x^1_0\right]-x^1_0\right\|^2\right]+ \int_0^t\mathbb{E}\left[\left\|\mathbb{E}\left[v^1_s\right]-v^1_s\right\|^2\right] ds + \int_0^t \|\sigma_s\|^2ds\right).
	\end{align*}
\end{proof}

\section{Proofs of Statements in Section \ref{lqmf.section.convergence_rate}} \label{lqmf.section.appendixA4}

\subsection{Proof of Lemma \ref{lqmf.mfg.lemma.approximation_state.optimal}}

\begin{proof}
	Since the collection of states $\left\{\hat{x}^j_{\cdot}\right\}_{j=1}^N$ and $\left\{\hat{y}^j_{\cdot}\right\}_{j=1}^N$ are symmetric, without loss of generality, we only prove the claim for $\hat{x}^1_{\cdot}$ and $\hat{y}^1_{\cdot}$. Firstly, write the difference, note that the noisy parts are cancelled by one another here:
	\begin{align*}
	(\hat{x}^1_t-\hat{y}^1_t)=\int_0^t \left(A_s (\hat{x}^1_s-\hat{y}^1_s)+\bar{A}_s\frac{1}{N-1}\sum_{j\neq 1} \left(\mathbb{E}\left[\hat{x}^j_s\right]-\hat{y}^j_s\right)+\bar{B}_s\frac{1}{N-1}\sum_{j\neq 1} \left(\mathbb{E}\left[\hat{v}^j_s\right]-\hat{v}^j_s\right) \right) d s.
	\end{align*}
	Taking the norm of both sides and by applying Young's inequality, we have
	\begin{align*}
	\|\hat{x}^1_t-\hat{y}^1_t\|^2\leq C_5\int_0^t \left( \|\hat{x}^1_s-\hat{y}^1_s\|^2+\left\|\frac{1}{N-1}\sum_{j\neq 1} (\mathbb{E}\left[\hat{x}^j_s\right]-\hat{y}^j_s)\right\|^2+\left\|\frac{1}{N-1}\sum_{j\neq 1} (\mathbb{E}\left[\hat{v}^j_s\right]-\hat{v}^j_s)\right\|^2 \right) d s,
	\end{align*}
	where $C_5=3\max\{\|A\|_T^2, \|\bar{A}\|_T^2, \|\bar{B}\|_T^2\}$ for all $s$, clearly it is independent of $N$. Then, we take the supremum and the expectation of both sides, so one obtain:
	\begin{align*}
	\mathbb{E}\left(\sup_{u\leq t}\|\hat{x}^1_u-\hat{y}^1_u\|^2\right) \leq C_5\int_0^t \left( \mathbb{E}\left(\sup_{u\leq s}\|\hat{x}^1_u-\hat{y}^1_u\|^2\right)+\mathbb{E}\left\|\frac{1}{N-1}\sum_{j\neq 1} (\mathbb{E}\left[\hat{x}^j_s\right]-\hat{y}^j_s)\right\|^2+\mathbb{E}\left\|\frac{1}{N-1}\sum_{j\neq 1} (\mathbb{E}\left[\hat{v}^j_s\right]-\hat{v}^j_s)\right\|^2 \right) d s
	\end{align*}
	On the other hand, we have
	{\small
		\begin{align*}
		\mathbb{E}\int_0^t\left\|\frac{1}{N-1}\sum_{j\neq 1} (\mathbb{E}\left[\hat{x}^j_s\right]-\hat{y}^j_s)\right\|^2ds \leq 2\int_0^t\mathbb{E}\left\|\frac{1}{N-1}\sum_{j\neq 1} (\mathbb{E}\left[\hat{x}^j_s\right]-\hat{x}^j_s)\right\|^2ds + 2\int_0^t\mathbb{E}\left\|\frac{1}{N-1}\sum_{j\neq 1} (\hat{x}^j_s-\hat{y}^j_s)\right\|^2ds,
		\end{align*}
	}
	in which, by Jensen's inequality,
	{\small
		\begin{align*}
		\int_0^t\mathbb{E}\left\|\frac{1}{N-1}\sum_{j\neq 1} (\hat{x}^j_s-\hat{y}^j_s)\right\|^2ds \leq \int_0^t\frac{1}{N-1}\sum_{j\neq 1}\mathbb{E}\left\| \hat{x}^j_s-\hat{y}^j_s\right\|^2ds
		= \int_0^t\mathbb{E}\left\| \hat{x}^1_s-\hat{y}^1_s\right\|^2ds.
		\end{align*}
	}
	where the last equality is due to the symmetric nature of $\left\{\hat{x}^j_{\cdot}-\hat{y}^j_{\cdot}\right\}_{j=1}^N$. Thus,
	\begin{align*}
	\mathbb{E}\left(\sup_{u\leq t}\|\hat{x}^1_u-\hat{y}^1_u\|^2\right) \leq C_5\int_0^t \left( 3\mathbb{E}\left(\sup_{u\leq s}\|\hat{x}^1_u-\hat{y}^1_u\|^2\right)+2\mathbb{E}\left\|\frac{1}{N-1}\sum_{j\neq 1} (\mathbb{E}\left[\hat{x}^j_s\right]-\hat{x}^j_s)\right\|^2+2\mathbb{E}\left\|\frac{1}{N-1}\sum_{j\neq 1} (\mathbb{E}\left[\hat{v}^j_s\right]-\hat{v}^j_s)\right\|^2 \right) d s.
	\end{align*}
	By Gr{\"o}nwall's inequality, we have
	\begin{align*}
	\mathbb{E}\left(\sup_{u\leq t}\|\hat{x}^1_u-\hat{y}^1_u\|^2\right) \leq 2C_5e^{3C_5t}\int_0^t \mathbb{E}\left\|\left(\frac{1}{N-1}\sum_{j\neq 1} (\mathbb{E}\left[\hat{x}^j_s\right]-\hat{x}^j_s)\right\|^2+\mathbb{E}\left\|\frac{1}{N-1}\sum_{j\neq 1} (\mathbb{E}\left[\hat{v}^j_s\right]-\hat{v}^j_s)\right\|^2 \right) d s.
	\end{align*}
	Note that the first term on the right hand side is
	\begin{align*}
	\mathbb{E}\int_0^t\left\|\frac{1}{N-1}\sum_{j\neq 1} (\mathbb{E}\left[\hat{x}^j_s\right]-\hat{x}^j_s)\right\|^2ds =& \int_0^t\frac{1}{(N-1)^2}\mathbb{E}\left\|\sum_{j\neq 1} (\mathbb{E}\left[\hat{x}^j_s\right]-\hat{x}^j_s)\right\|^2ds \\
	=& \int_0^t\frac{1}{(N-1)^2}\mathbb{E}\sum_{j\neq 1}\left\| \mathbb{E}\left[\hat{x}^j_s\right]-\hat{x}^j_s\right\|^2ds \\
	=& \int_0^t\frac{1}{N-1}\mathbb{E}\left\| \mathbb{E}\left[\hat{x}^j_s\right]-\hat{x}^j_s\right\|^2ds,
	\end{align*}
	where the second equality is due to the i.i.d. nature of $\hat{x}^i_t$ and $\hat{x}^j_t$, $i \neq j$, so that
	\begin{align*}
	\mathbb{E}\left(\left(\mathbb{E}\left[\hat{x}^i_t\right] -  \hat{x}^i_t\right)^{\top}\left(\mathbb{E}\left[\hat{x}^j_t\right] -  \hat{x}^j_t\right)\right) = 0,
	\end{align*} and the third equality is due to the symmetry again. Recalling that the mean field equilibrium $\hat{v}^i$ satisfies \eqref{lqmf.mfg.fixed_point.optimal_control} and it only depends on the corresponding optimal state process $\hat{x}^i$, we also have, for any $i \neq j$,
	\begin{align*}
	\mathbb{E}\left(\left(\mathbb{E}\left[\hat{v}^i_t\right] -  \hat{v}^i_t\right)^{\top}\left(\mathbb{E}\left[\hat{v}^j_t\right] -  \hat{v}^j_t\right)\right) = 0,
	\end{align*}
	hence we can use the same argument to deduce:
	\begin{align}
	\label{lqmf.mfg.appendix.inequality_in_lemma4_1.v}
	\mathbb{E}\left(\int_0^t\left\|\frac{1}{N-1}\sum_{j\neq 1} (\mathbb{E}\left[\hat{v}^j_s\right]-\hat{v}^j_s)\right\|^2ds\right) =& \int_0^t\frac{1}{N-1}\mathbb{E}\left\| \mathbb{E}\left[\hat{v}^j_s\right]-\hat{v}^j_s\right\|^2ds.
	\end{align}
	Therefore, we derive:
	\begin{align*}
	&\mathbb{E}\left(\sup_{u\leq t}\|\hat{x}^1_u-\hat{y}^1_u\|^2\right) \leq \frac{2C_5e^{3C_5t}}{N-1}\left(\int_0^t\mathbb{E}\left\| \mathbb{E}\left[\hat{x}^j_s\right]-\hat{x}^j_s\right\|^2ds+ \int_0^t\mathbb{E}\left\| \mathbb{E}\left[\hat{v}^j_s\right]-\hat{v}^j_s\right\|^2ds\right).
	\end{align*}
	By \eqref{lqmf.mfg.estimate.x-Ex}, we conclude the claim of \eqref{lqmf.mfg.convergence.xhat}.
\end{proof}

\subsection{Proof of Lemma \ref{lqmf.mfg.lemma.approximation_objective.optimal}}

\begin{proof}
	We only prove the result for the first player, the general cases for all other $i$'s follow by symmetry argument. Recall that $C_6 = \max\{\|Q\|_T,\|\bar{Q}\|_T,\|\bar{P}\|_T,\|N\|_T\}$ and $C_7=\max\{\|S\|_T^2,\|R\|_T^2,\|\bar{S}\|_T^2,\|\bar{R}\|_T^2\}$. We compare all terms of \eqref{lqmf.N_player.objective.appendix} and \eqref{lqmf.mfg.objective.appendix} one by one. By Cauchy-Schwarz inequality and Young's inequality, we have
	\begin{align*}
	\mathbb{E}\left[\int_0^T(\hat{x}^1_t)^{\top} Q_t \hat{x}^1_t-(\hat{y}^1_t)^{\top} Q_t \hat{y}^1_t dt\right] 
	=& \ \mathbb{E}\left[\int_0^T(\hat{x}^1_t+\hat{y}^1_t)^{\top} Q_t (\hat{x}^1_t-\hat{y}^1_t) dt\right] \\
	\leq& \ C_6\left(\mathbb{E}\left[\int_0^T\|\hat{x}^1_t\|^2+\|\hat{y}^1_t\|^2dt\right]\right)^{\frac{1}{2}}\cdot \left(\mathbb{E}\left[\int_0^T\|\hat{x}^1_t-\hat{y}^1_t\|^2dt\right]\right)^{\frac{1}{2}}.
	\end{align*}
	Similarly, we have
	\begin{align*}
	&\quad \ \mathbb{E}\left[\int_0^T\left(\hat{x}^1_t-S_t \mathbb{E}\left[\hat{x}^1_t\right]\right)^{\top} \bar{Q}_t\left(\hat{x}^1_t-S_t \mathbb{E}\left[\hat{x}^1_t\right]\right) - \left(\hat{y}^1_t-S_t \frac{1}{N-1}\sum_{j\neq 1}\hat{y}^j_t\right)^{\top} \bar{Q}_t\left(\hat{y}^1_t-S_t \frac{1}{N-1}\sum_{j\neq 1}\hat{y}^j_t\right)dt\right] \\
	&= \mathbb{E}\left[\int_0^T\left(\hat{x}^1_t-S_t \mathbb{E}\left[\hat{x}^1_t\right]+\hat{y}^1_t-S_t \frac{1}{N-1}\sum_{j\neq 1}\hat{y}^j_t\right)^{\top}\bar{Q}_t\left(\hat{x}^1_t-S_t \mathbb{E}\left[\hat{x}^1_t\right]-\hat{y}^1_t+S_t \frac{1}{N-1}\sum_{j\neq 1}\hat{y}^j_t\right)dt\right] \\
	&\leq \sqrt{4}  C_6\left(\mathbb{E}\left[\int_0^T\|\hat{x}^1_t\|^2+\|\bar{S}\|_T^2\mathbb{E}\left[\|\hat{x}^1_t\|^2\right]+\|\hat{y}^1_t\|^2+\frac{\|\bar{S}\|_T^2}{N-1}\sum_{j\neq 1}\|\hat{y}^j_t\|^2dt\right]\right)^\frac{1}{2} \\
	&\quad \cdot \sqrt{2}\left(\mathbb{E}\left[\int_0^T\|\hat{x}^1_t-\hat{y}^1_t\|^2+ \|\bar{S}\|_T^2\left\|\frac{1}{N-1}\sum_{j\neq 1}\left(\mathbb{E}\left[\hat{x}^j_t\right] -\hat{y}^j_t\right)\right\|^2dt\right]\right)^{\frac{1}{2}} \\
	&\leq 4C_6(1+\|\bar{S}\|_T^2)\left(\mathbb{E}\left[\int_0^T\|\hat{x}^1_t\|^2+\|\hat{y}^1_t\|^2dt\right]\right)^\frac{1}{2} \cdot\left(\mathbb{E}\left[\int_0^T\|\hat{x}^1_t-\hat{y}^1_t\|^2+ \left\|\frac{1}{N-1}\sum_{j\neq 1}\left(\mathbb{E}\left[\hat{x}^1_t\right] -\hat{y}^j_t\right)\right\|^2dt\right]\right)^{\frac{1}{2}},
	\end{align*}
	where the last inequality follows from the symmetry of $\left\{\hat{y}^j_{\cdot}\right\}_{j=1}^N$. By the same argument as in the proof of Lemma \ref{lqmf.mfg.lemma.approximation_state.optimal}, we have
	\begin{align}
	\label{lqmf.mfg.appendix.inequality_in_lemma4_3.x}
	\mathbb{E}\int_0^T\left\|\frac{1}{N-1}\sum_{j\neq 1} (\mathbb{E}\left[\hat{x}^j_t\right]-\hat{y}^j_t)\right\|^2dt \leq 2\mathbb{E}\left[\int_0^T\frac{1}{N-1}\left\| \mathbb{E}\left[\hat{x}^j_t\right]-\hat{x}^j_t\right\|^2 + \left\| \hat{x}^1_t-\hat{y}^1_t\right\|^2dt\right],
	\end{align}
	thus
	\begin{align*}
	&\mathbb{E}\left[\int_0^T\left(\hat{x}^1_t-S_t \mathbb{E}\left[\hat{x}^1_t\right]\right)^{\top} \bar{Q}_t\left(\hat{x}^1_t-S_t \mathbb{E}\left[\hat{x}^1_t\right]\right) - \left(\hat{y}^1_t-S_t \frac{1}{N-1}\sum_{j\neq 1}\hat{y}^j_t\right)^{\top} \bar{Q}_t\left(\hat{y}^1_t-S_t \frac{1}{N-1}\sum_{j\neq 1}\hat{y}^j_t\right)dt\right] \\
	&\leq 12C_6(1+C_7)\left(\mathbb{E}\left[\int_0^T\|\hat{x}^1_t\|^2+\|\hat{y}^1_t\|^2dt\right]\right)^\frac{1}{2} \cdot\left(\mathbb{E}\left[\int_0^T\|\hat{x}^1_t-\hat{y}^1_t\|^2+ \frac{1}{N-1}\left\| \mathbb{E}\left[\hat{x}^j_t\right]-\hat{x}^j_t\right\|^2dt\right]\right)^{\frac{1}{2}},
	\end{align*}
	where the last inequality follows by using the fact that $\|\bar{S}\|_T^2 \leq C_7$.
	By the same argument, we further obtain:
	\begin{align*}
	&\mathbb{E}\left[\int_0^T\left(\hat{v}^1_t-R_t \mathbb{E}\left[\hat{v}^1_t\right]\right)^{\top} \bar{P}_t\left(\hat{v}^1_t-R_t \mathbb{E}\left[\hat{v}^1_t\right]\right) -\left(\hat{v}^1_t-R_t \frac{1}{N-1}\sum_{j\neq 1}\hat{v}^j_t\right)^{\top} \bar{P}_t\left(\hat{v}^1_t-R_t \frac{1}{N-1}\sum_{j\neq 1}\hat{v}^j_t\right)dt\right] \\
	&= \mathbb{E}\left[\int_0^T\left(\hat{v}^1_t-R_t \mathbb{E}\left[\hat{v}^1_t\right]+\hat{v}^1_t-R_t \frac{1}{N-1}\sum_{j\neq 1}\hat{v}^j_t\right)^{\top}\bar{P}_t\left(\hat{v}^1_t-R_t \mathbb{E}\left[\hat{v}^1_t\right]-\hat{v}^1_t+R_t \frac{1}{N-1}\sum_{j\neq 1}\hat{v}^j_t\right)dt\right] \\
	&\leq \sqrt{4}C_6\left(\mathbb{E}\left[\int_0^T2\|\hat{v}^1_t\|^2+\|R\|_T^2\mathbb{E}\left[\|\hat{v}^1_t\|^2\right]+\frac{\|R\|_T^2}{N-1}\sum_{j\neq 1}\|\hat{v}^j_t\|^2dt\right]\right)^\frac{1}{2}\cdot\sqrt{2}\left(\mathbb{E}\left[\int_0^T\frac{\|R\|_T^2}{N-1}\left\|\mathbb{E}\left[\hat{v}^j_t\right] -\hat{v}^j_t\right\|^2dt\right]\right)^\frac{1}{2} \\
	&\leq 4C_6(1+\|R\|_T^2)\left(\mathbb{E}\left[\int_0^T\|\hat{v}^1_t\|^2dt\right]\right)^\frac{1}{2}\cdot\left(\mathbb{E}\left[\int_0^T\frac{1}{N-1}\left\|\mathbb{E}\left[\hat{v}^j_t\right] -\hat{v}^j_t\right\|^2dt\right]\right)^\frac{1}{2} \\
	&\leq 4C_6(1+C_7)\left(\mathbb{E}\left[\int_0^T\|\hat{v}^1_t\|^2dt\right]\right)^\frac{1}{2}\cdot\left(\mathbb{E}\left[\int_0^T\frac{1}{N-1}\left\|\mathbb{E}\left[\hat{v}^j_t\right] -\hat{v}^j_t\right\|^2dt\right]\right)^\frac{1}{2},
	\end{align*}
	where the last inequality follows from the symmetry of $\left\{\hat{v}^j_{\cdot}\right\}_{j=1}^N$. Finally, for the cross term, it is:
	\begin{align*}
	&\mathbb{E}\left[\int_0^T2\left(x^1_t-\bar{S}_t \mathbb{E}\left[\hat{x}^1_t\right]\right)^{\top} N_t\left(\hat{v}^1_t-\bar{R}_t \mathbb{E}\left[\hat{v}^1_t\right]\right) - 2\left(\hat{y}^1_t-\bar{S}_t \frac{1}{N-1}\sum_{j\neq 1}\hat{y}^j_t\right)^{\top} N_t\left(\hat{v}^1_t-\bar{R}_t \frac{1}{N-1}\sum_{j\neq 1}\hat{v}^j_t\right)dt\right] \\
	&= 2\mathbb{E}\left[\int_0^T\left(\hat{x}^1_t-\bar{S}_t \mathbb{E}\left[\hat{x}^1_t\right]\right)^{\top} N_t\left(\hat{v}^1_t-\bar{R}_t \mathbb{E}\left[\hat{v}^1_t\right]\right) - \left(\hat{x}^1_t-\bar{S}_t \mathbb{E}\left[\hat{x}^1_t\right]\right)^{\top} N_t\left(\hat{v}^1_t-\bar{R}_t \frac{1}{N-1}\sum_{j\neq 1}\hat{v}^j_t\right)dt\right] \\
	&\quad+ 2\mathbb{E}\left[\int_0^T\left(\hat{x}^1_t-\bar{S}_t \mathbb{E}\left[\hat{x}^1_t\right]\right)^{\top} N_t\left(\hat{v}^1_t-\bar{R}_t \frac{1}{N-1}\sum_{j\neq 1}\hat{v}^j_t\right) - \left(\hat{y}^1_t-\bar{S}_t \frac{1}{N-1}\sum_{j\neq 1}\hat{y}^j_t\right)^{\top} N_t\left(\hat{v}^1_t-\bar{R}_t \frac{1}{N-1}\sum_{j\neq 1}\hat{v}^j_t\right)dt\right] \\
	&= 2\mathbb{E}\left[\int_0^T\left(\hat{x}^1_t-\bar{S}_t \mathbb{E}\left[\hat{x}^1_t\right]\right)^{\top} N_t\bar{R}_t\left( \frac{1}{N-1}\sum_{j\neq 1}\hat{v}^j_t- \mathbb{E}\left[\hat{v}^1_t\right]\right)dt\right] \\
	&\quad+ 2\mathbb{E}\left[\int_0^T\left(\hat{v}^1_t-\bar{R}_t \frac{1}{N-1}\sum_{j\neq 1}\hat{v}^j_t\right)^{\top} N_t^{\top}\left(\hat{x}^1_t-\bar{S}_t \mathbb{E}\left[\hat{x}^1_t\right]-\hat{y}^1_t+\bar{S}_t \frac{1}{N-1}\sum_{j\neq 1}\hat{y}^j_t\right)dt\right] \\
	&\leq 2\sqrt{2}C_6\|\bar{R}\|_T\left(\mathbb{E}\left[\int_0^T\|\hat{x}^1_t\|^2+ \|\bar{S}\|_T^2\mathbb{E}\left[\|\hat{x}^1_t\|^2\right]dt\right]\right)^{\frac{1}{2}}\cdot\left(\mathbb{E}\left[\int_0^T\left\|\frac{1}{N-1}\sum_{j\neq 1}  \left(\mathbb{E}\left[\hat{v}^j_t\right]-\hat{v}^j_t\right)\right\|^2dt\right]\right)^{\frac{1}{2}} \\
	&\quad+ 2C_6\sqrt{2}\left(\mathbb{E}\left[\int_0^T\|\hat{v}^1_t\|^2+\frac{\|\bar{R}\|_T^2}{N-1}\sum_{j\neq 1}\|\hat{v}^j_t\|^2dt\right]\right)^{\frac{1}{2}}\cdot\sqrt{2}\left(\mathbb{E}\left[\int_0^T\|\hat{x}^1_t-\hat{y}^1_t\|^2+\|\bar{S}\|_T^2\left\|\frac{1}{N-1}\sum_{j\neq 1} \left(\mathbb{E}\left[\hat{x}^j_t\right]- \hat{y}^j_t\right)\right\|^2dt\right]\right)^{\frac{1}{2}} \\
	&\leq 2\sqrt{2}C_6\|\bar{R}\|_T(1+\|\bar{S}\|_T^2)^{\frac{1}{2}}\left(\mathbb{E}\left[\int_0^T\|\hat{x}^1_t\|^2dt\right]\right)^{\frac{1}{2}}\cdot\left(\mathbb{E}\left[\int_0^T\frac{1}{N-1}\left\| \mathbb{E}\left[\hat{v}^j_s\right]-\hat{v}^j_s\right\|^2dt\right]\right)^{\frac{1}{2}} \\
	&\quad+ 4\sqrt{3}C_6(1+\|\bar{R}\|_T^2)^{\frac{1}{2}}(1+\|\bar{S}\|_T^2)^{\frac{1}{2
	}}\left(\mathbb{E}\left[\int_0^T\|\hat{v}^1_t\|^2dt\right]\right)^{\frac{1}{2}}\cdot\left(\mathbb{E}\left[\int_0^T\|\hat{x}^1_t-\hat{y}^1_t\|^2+\frac{1}{N-1}\left\| \mathbb{E}\left[\hat{x}^j_t\right]-\hat{x}^j_t\right\|^2dt\right]\right)^{\frac{1}{2}} \\
	&\leq 8C_6(1+C_7)\left(\mathbb{E}\left[\int_0^T\|\hat{x}^1_t\|^2+\|\hat{v}^1_t\|^2dt\right]\right)^{\frac{1}{2}} \cdot\left(\mathbb{E}\left[\int_0^T\|\hat{x}^1_t-\hat{y}^1_t\|^2+\frac{1}{N-1}\left\| \mathbb{E}\left[\hat{x}^j_t\right]-\hat{x}^j_t\right\|^2+\frac{1}{N-1}\left\| \mathbb{E}\left[\hat{v}^j_s\right]-\hat{v}^j_s\right\|^2dt\right]\right)^{\frac{1}{2}},
	\end{align*}
	where the second last inequality follows from \eqref{lqmf.mfg.appendix.inequality_in_lemma4_1.v} and \eqref{lqmf.mfg.appendix.inequality_in_lemma4_3.x}.
	We take the difference of \eqref{lqmf.N_player.objective.appendix} and \eqref{lqmf.mfg.objective.appendix}, and substitute the above inequalities into it. By Cauchy-Schwarz inequality, we have
	\begin{align*}
	&J(\hat{v}^1)-\mathcal{J}(\hat{v}^1,\hat{v}^{-1}) \\
	&\leq 20C_6(1+C_7)\Bigg\{\left(\mathbb{E}\bigg[ \int_0^T \|\hat{x}^1_t\|^2+\|\hat{y}^1_t\|^2+\|\hat{v}^1_t\|^2 dt\bigg]\right)^{\frac{1}{2}}\cdot \left(\mathbb{E}\bigg[\int_0^T \|\hat{x}^1_t-\hat{y}^1_t\|^2+ \frac{1}{N-1}\left\|\mathbb{E}\left[\hat{x}^1_t\right] -\hat{x}^j_t\right\|^2+\frac{1}{N-1}\left\| \mathbb{E}\left[\hat{v}^1_t\right]-\hat{v}^j_t\right\|^2 d t\bigg]\right)^{\frac{1}{2}} \\
	&\quad\quad\quad\quad\quad\quad\quad\quad+ \left(\mathbb{E}\bigg[\|\hat{x}^1_T\|^2+\|\hat{y}^1_T\|^2\bigg]\right)^{\frac{1}{2}} \cdot \left(\mathbb{E}\bigg[\|\hat{x}^1_T-\hat{y}^1_T\|^2+ \frac{1}{N-1}\left\|\mathbb{E}\left[\hat{x}^1_T\right] -\hat{x}^j_T\right\|^2 \bigg]\right)^{\frac{1}{2}}\Bigg\} \\
	&\leq 20C_6(1+C_7)\left(\mathbb{E}\bigg[ \int_0^T \|\hat{x}^1_t\|^2+\|\hat{y}^1_t\|^2+\|\hat{v}^1_t\|^2 dt\bigg]+\mathbb{E}\bigg[\|\hat{x}^1_T\|^2+\|\hat{y}^1_T\|^2\bigg]\right)^{\frac{1}{2}}\\
	&\quad\cdot \bigg(\mathbb{E}\bigg[\int_0^T \|\hat{x}^1_t-\hat{y}^1_t\|^2+ \frac{1}{N-1}\left\|\mathbb{E}\left[\hat{x}^1_t\right] -\hat{x}^j_t\right\|^2+\frac{1}{N-1}\left\| \mathbb{E}\left[\hat{v}^1_t\right]-\hat{v}^j_t\right\|^2 d t\bigg] \\
	&\quad\quad+\mathbb{E}\bigg[\|\hat{x}^1_T-\hat{y}^1_T\|^2 + \frac{1}{N-1}\left\|\mathbb{E}\left[\hat{x}^1_T\right] -\hat{x}^j_T\right\|^2 \bigg]\bigg)^{\frac{1}{2}} \\
	&\leq 60C_6(1+C_7)\left(\mathbb{E}\left[\left\|\hat{x}^1_T\right\|^2+\int_0^T  \left\|\hat{x}^1_t\right\|^2+ \left\|\hat{v}^1_t\right\|^2dt\right]+\mathbb{E}\bigg[ \|\hat{x}^1_T-\hat{y}^1_T\|^2+\int_0^T \|\hat{x}^1_t-\hat{y}^1_t\|^2 dt\bigg]\right)^{\frac{1}{2}}\\
	&\quad\cdot \bigg(\frac{1}{N-1}\mathbb{E}\bigg[\left\|\mathbb{E}\left[\hat{x}^1_T\right] -\hat{x}^j_T\right\|^2+\int_0^T \left\|\mathbb{E}\left[\hat{x}^1_t\right] -\hat{x}^j_t\right\|^2+\left\| \mathbb{E}\left[\hat{v}^1_t\right]-\hat{v}^j_t\right\|^2 d t\bigg] +\mathbb{E}\bigg[ \|\hat{x}^1_T-\hat{y}^1_T\|^2+\int_0^T \|\hat{x}^1_t-\hat{y}^1_t\|^2 dt\bigg]\bigg)^{\frac{1}{2}} \\
	&\leq \frac{1}{\sqrt{N-1}}60C_6(1+C_7)\left(M_{e,2} \left\|\mathbb{E}\left[x^1_0\right]\right\|^2 +(M_{e,1}+M_{e,3}(T+1)) \left(\mathbb{E}\left[\left\|\mathbb{E}\left[x^1_0\right]-x^1_0\right\|^2\right]+T\|\sigma\|_T^2\right)\right),
	\end{align*}
	where the last inequality follows from Lemma \ref{lqmf.mfg.lemma.boundedness.x-Ex}, Lemma \ref{lqmf.mfg.lemma.boundedness.Ex2} and Lemma \ref{lqmf.mfg.lemma.approximation_state.optimal}.
\end{proof}

\subsection{Proof of Lemma \ref{lqmf.mfg.lemma.approximation_state.pertubed}}

\begin{proof}
	By applying the same argument in the proof of Lemma \ref{lqmf.mfg.lemma.approximation_state.optimal}, we have
	\begin{align}
	\nonumber
	&\mathbb{E}\left(\sup_{u\leq t}\|x^1_u-y^1_u\|^2\right) \\
	\nonumber
	&\leq C_5\int_0^t \left( \mathbb{E}\left(\sup_{u\leq s}\|x^1_u-y^1_u\|^2\right)+\mathbb{E}\left\|\frac{1}{N-1}\sum_{j\neq 1} (\mathbb{E}\left[\hat{x}^j_s\right]-\hat{y}^j_s)\right\|^2+\mathbb{E}\left\|\frac{1}{N-1}\sum_{j\neq 1} (\mathbb{E}\left[\hat{v}^j_s\right]-\hat{v}^j_s)\right\|^2 \right) d s \\
	\label{lqmf.mfg.inequality.x1-y1.dummy}
	&\leq C_5\int_0^t \left( 3\mathbb{E}\left(\sup_{u\leq s}\|x^1_u-y^1_u\|^2\right)+2\mathbb{E}\left\|\frac{1}{N-1}\sum_{j\neq 1} (\mathbb{E}\left[\hat{x}^j_s\right]-\hat{x}^j_s)\right\|^2+2\mathbb{E}\left\|\frac{1}{N-1}\sum_{j\neq 1} (\mathbb{E}\left[\hat{v}^j_s\right]-\hat{v}^j_s)\right\|^2 \right) d s;
	\end{align}
	and since
	\begin{align*}
	&\mathbb{E}\left\|\frac{1}{N-1}\left(\mathbb{E}\left[\hat{x}^1_s\right]-y^1_s+\sum_{j\neq 1,i} (\mathbb{E}\left[\hat{x}^j_s\right]-\hat{y}^j_s)\right)\right\|^2 \\
	&\leq 2\int_0^t\mathbb{E}\left\|\frac{1}{N-1}\left(\mathbb{E}\left[\hat{x}^1_s\right]-x^1_s+\sum_{j\neq 1,i} (\mathbb{E}\left[\hat{x}^j_s\right]-\hat{x}^j_s)\right)\right\|^2ds + 2\int_0^t\mathbb{E}\left\|x^1_s-y^1_s+\frac{1}{N-1}\sum_{j\neq 1,i} (\hat{x}^j_s-\hat{y}^j_s)\right\|^2ds \\
	&\leq 4\int_0^t\left(\mathbb{E}\left\|\frac{1}{N-1}(\mathbb{E}\left[x^1_s\right]-x^1_s)\right\|^2 +\mathbb{E}\left\|\frac{1}{N-1}\sum_{j\neq 1,i} (\mathbb{E}\left[\hat{x}^j_s\right]-\hat{x}^j_s)\right\|^2\right)ds + 4\int_0^t\left(\frac{1}{N-1}\mathbb{E}\left\|x^1_s-y^1_s\right\|^2+\mathbb{E}\left\| \hat{x}^i_s-\hat{y}^i_s\right\|^2\right)ds;
	\end{align*}
	\begin{align}
	\nonumber
	&\mathbb{E}\left(\sup_{u\leq t}\|\hat{x}^i_u-\hat{y}^i_u\|^2\right) \\
	\nonumber
	\leq& C_5\int_0^t \bigg( \mathbb{E}\sup_{u\leq s}\|\hat{x}^i_u-\hat{y}^i_u\|^2+\mathbb{E}\left\|\frac{1}{N-1}\left(\mathbb{E}\left[\hat{x}^1_s\right]-y^1_s+\sum_{j\neq 1,i} (\mathbb{E}\left[\hat{x}^j_s\right]-\hat{y}^j_s)\right)\right\|^2 +\mathbb{E}\left\|\frac{1}{N-1}\left(\mathbb{E}\left[\hat{v}^1_s\right]-v^1_s+\sum_{j\neq 1,i} (\mathbb{E}\left[\hat{v}^j_s\right]-\hat{v}^j_s)\right)\right\|^2 \bigg) d s \\
	\nonumber
	\leq& C_5\int_0^t \bigg( 5\mathbb{E}\sup_{u\leq s}\|\hat{x}^i_u-\hat{y}^i_u\|^2+\frac{4}{N-1}\mathbb{E}\sup_{u\leq s}\left\|x^1_u-y^1_u\right\|^2 +4\mathbb{E}\left\|\frac{1}{N-1}(\mathbb{E}\left[x^1_s\right]-x^1_s)\right\|^2 +4\mathbb{E}\left\|\frac{1}{N-1}\sum_{j\neq 1,i} (\mathbb{E}\left[\hat{x}^j_s\right]-\hat{x}^j_s)\right\|^2 \\
	\label{lqmf.mfg.inequality.xi-yi}
	&\quad+4\mathbb{E}\left\|\frac{1}{N-1}(\mathbb{E}\left[v^1_s\right]-v^1_s)\right\|^2+4\mathbb{E}\left\|\frac{1}{N-1}\sum_{j\neq 1,i} (\mathbb{E}\left[\hat{v}^j_s\right]-\hat{v}^j_s)\right\|^2 \bigg) d s.
	\end{align}
	We combine equalities \eqref{lqmf.mfg.inequality.x1-y1.dummy} and \eqref{lqmf.mfg.inequality.xi-yi}, and by Gr{\"o}nwall's inequality, we have
	\begin{align*}
	&\mathbb{E}\left(\sup_{u\leq t}\|x^1_t-y^1_t\|^2\right) + \mathbb{E}\left(\sup_{u\leq t}\|\hat{x}^i_t-\hat{y}^i_t\|^2\right) \\
	\leq& C_5e^{7C_5t}\int_0^t \Bigg( 2\mathbb{E}\left\|\frac{1}{N-1}\sum_{j\neq 1} (\mathbb{E}\left[\hat{x}^j_s\right]-\hat{x}^j_s)\right\|^2+2\mathbb{E}\left\|\frac{1}{N-1}\sum_{j\neq 1} (\mathbb{E}\left[\hat{v}^j_s\right]-\hat{v}^j_s)\right\|^2 +4\mathbb{E}\left\|\frac{1}{N-1}\sum_{j\neq 1,i} (\mathbb{E}\left[\hat{x}^j_s\right]-\hat{x}^j_s)\right\|^2 \\
	&\quad\quad\quad\quad+4\mathbb{E}\left\|\frac{1}{N-1}\sum_{j\neq 1,i} (\mathbb{E}\left[\hat{v}^j_s\right]-\hat{v}^j_s)\right\|^2+4\mathbb{E}\left\|\frac{1}{N-1}(\mathbb{E}\left[x^1_s\right]-x^1_s)\right\|^2+4\mathbb{E}\left\|\frac{1}{N-1}(\mathbb{E}\left[v^1_s\right]-v^1_s)\right\|^2 \Bigg) d s.
	\end{align*}
	By the same argument as in the proof of Lemma \ref{lqmf.mfg.lemma.approximation_state.optimal}, we obtain
	\begin{align*}
	&\quad \ \mathbb{E}\left(\sup_{u\leq t}\|x^1_t-y^1_t\|^2\right) + \mathbb{E}\left(\sup_{u\leq t}\|\hat{x}^i_t-\hat{y}^i_t\|^2\right) \\
	&\leq 6C_5e^{7C_5t}\bigg(\frac{1}{N-1}\int_0^t \mathbb{E}\left\| \mathbb{E}\left[\hat{x}^i_s\right]-\hat{x}^i_s\right\|^2+\mathbb{E}\left\| \mathbb{E}\left[\hat{v}^i_s\right]-\hat{v}^i_s\right\|^2 d s +\frac{1}{(N-1)^2}\int_0^t\mathbb{E}\left\|\mathbb{E}\left[x^1_s\right]-x^1_s\right\|^2+\mathbb{E}\left\|\mathbb{E}\left[v^1_s\right]-v^1_s\right\|^2 d s\bigg) \\
	&\leq 6C_5e^{7C_5t}\bigg\{\frac{M_{e,1}}{N-1}\left(\mathbb{E}\left[\left\|\mathbb{E}\left[x^i_0\right]-x^i_0\right\|^2\right]+T\|\sigma\|_T^2\right) +\frac{M_{d,4}}{(N-1)^2}\left(\mathbb{E}\left[\left\|\mathbb{E}\left[x^1_0\right]-x^1_0\right\|^2\right]+ \int_0^t\mathbb{E}\left[\left\|\mathbb{E}\left[v^1_s\right]-v^1_s\right\|^2\right] ds +T\|\sigma\|_T^2\right)\bigg\} \\
	&\leq 6C_5e^{7C_5t}\bigg\{\left(\frac{M_{e,1}}{N-1}+\frac{M_{d,4}}{(N-1)^2}\right)\left(\mathbb{E}\left[\left\|\mathbb{E}\left[x^i_0\right]-x^i_0\right\|^2\right]+T\|\sigma\|_T^2\right) +\frac{M_{d,4}}{(N-1)^2}\int_0^t\mathbb{E}\left[\left\|\mathbb{E}\left[v^1_s\right]-v^1_s\right\|^2\right] ds\bigg\},
	\end{align*}
	where the last inequality follows from Lemma \ref{lqmf.mfg.lemma.boundedness.x-Ex} and Lemma \ref{lqmf.mfg.lemma.boundedness.x-Ex.dummy}.
	
\end{proof}

\subsection{Proof of Lemma \ref{lqmf.mfg.lemma.approximation_objective.pertubed}}

\begin{proof}
	By Cauchy-Schwarz inequality and Young's inequality, the general case for every $i$ follows by symmetry argument,
	\begin{align*}
	\mathbb{E}\left[\int_0^T(x^1_t)^{\top} Q_t x^1_t-(y^1_t)^{\top} Q_t y^1_t dt\right] 
	=& \mathbb{E}\left[\int_0^T(x^1_t+y^1_t)^{\top} Q_t (x^1_t-y^1_t) dt\right] \\
	\leq& C_6\left(\mathbb{E}\left[\int_0^T\|x^1_t\|^2+\|y^1_t\|^2dt\right]\right)^{\frac{1}{2}} \cdot \left(\mathbb{E}\left[\int_0^T\|x^1_t-y^1_t\|^2dt\right]\right)^{\frac{1}{2}}.
	\end{align*}
	Similarly, we have
	\begin{align*}
	&\mathbb{E}\left[\int_0^T\left(x^1_t-S_t \mathbb{E}\left[\hat{x}^1_t\right]\right)^{\top} \bar{Q}_t\left(x^1_t-S_t \mathbb{E}\left[\hat{x}^1_t\right]\right) - \left(y^1_t-S_t \frac{1}{N-1}\sum_{j\neq 1}\hat{y}^j_t\right)^{\top} \bar{Q}_t\left(y^1_t-S_t \frac{1}{N-1}\sum_{j\neq 1}\hat{y}^j_t\right)dt\right] \\
	&= \mathbb{E}\left[\int_0^T\left(x^1_t-S_t \mathbb{E}\left[\hat{x}^1_t\right]+y^1_t-S_t \frac{1}{N-1}\sum_{j\neq 1}\hat{y}^j_t\right)^{\top}\bar{Q}_t\left(x^1_t-S_t \mathbb{E}\left[\hat{x}^1_t\right]-y^1_t+S_t \frac{1}{N-1}\sum_{j\neq 1}\hat{y}^j_t\right)dt\right] \\
	&\leq \sqrt{4}C_6\left(\mathbb{E}\left[\int_0^T\|x^1_t\|^2+\|\bar{S}\|_T^2\mathbb{E}\left[\|\hat{x}^1_t\|^2\right]+\|y^1_t\|^2+\frac{\|\bar{S}\|_T^2}{N-1}\sum_{j\neq 1}\|\hat{y}^j_t\|^2dt\right]\right)^\frac{1}{2} \\
	&\quad\cdot\sqrt{2}\left(\mathbb{E}\left[\int_0^T\|x^1_t-y^1_t\|^2+ \|\bar{S}\|_T^2\left\|\frac{1}{N-1}\sum_{j\neq 1}\left(\mathbb{E}\left[\hat{x}^j_t\right] -\hat{y}^j_t\right)\right\|^2dt\right]\right)^{\frac{1}{2}} \\
	&\leq \sqrt{4}C_6(1+\|\bar{S}\|_T^2)\left(\mathbb{E}\left[\int_0^T\|x^1_t\|^2+\|y^1_t\|^2+\|\hat{x}^1_t\|^2+\|\hat{y}^1_t\|^2dt\right]\right)^\frac{1}{2} \cdot\sqrt{2}\left(\mathbb{E}\left[\int_0^T\|x^1_t-y^1_t\|^2+ \left\|\frac{1}{N-1}\sum_{j\neq 1}\left(\mathbb{E}\left[\hat{x}^1_t\right] -\hat{y}^j_t\right)\right\|^2dt\right]\right)^{\frac{1}{2}},
	\end{align*}
	where the last inequality follows from the symmetry of $\left\{\hat{y}^j_{\cdot}\right\}_{j=1}^N$. We apply the same argument as shown in the proof of Lemma \ref{lqmf.mfg.lemma.approximation_state.optimal}, it follows that
	\begin{align}
	\label{lqmf.mfg.appendix.inequality_in_lemma4_3.x.2}
	\mathbb{E}\int_0^T\left\|\frac{1}{N-1}\sum_{j\neq 1} (\mathbb{E}\left[\hat{x}^j_t\right]-\hat{y}^j_t)\right\|^2dt \leq 2\mathbb{E}\left[\int_0^T\frac{1}{N-1}\left\| \mathbb{E}\left[\hat{x}^j_t\right]-\hat{x}^j_t\right\|^2 + \left\| \hat{x}^1_t-\hat{y}^1_t\right\|^2dt\right],
	\end{align}
	thus
	\begin{align*}
	&\mathbb{E}\left[\int_0^T\left(x^1_t-S_t \mathbb{E}\left[\hat{x}^1_t\right]\right)^{\top} \bar{Q}_t\left(x^1_t-S_t \mathbb{E}\left[\hat{x}^1_t\right]\right) - \left(y^1_t-S_t \frac{1}{N-1}\sum_{j\neq 1}\hat{y}^j_t\right)^{\top} \bar{Q}_t\left(y^1_t-S_t \frac{1}{N-1}\sum_{j\neq 1}\hat{y}^j_t\right)dt\right] \\
	&\leq 4C_6(1+C_7)\left(\mathbb{E}\left[\int_0^T\|x^1_t\|^2+\|y^1_t\|^2+\|\hat{x}^1_t\|^2+\|\hat{y}^1_t\|^2dt\right]\right)^\frac{1}{2}\cdot\left(\mathbb{E}\left[\int_0^T\|x^1_t-y^1_t\|^2+ \left\| \hat{x}^1_t-\hat{y}^1_t\right\|^2+ \frac{1}{N-1}\left\| \mathbb{E}\left[\hat{x}^j_t\right]-\hat{x}^j_t\right\|^2dt\right]\right)^{\frac{1}{2}},
	\end{align*}
	where the last inequality follows by using the fact that $\|\bar{S}\|_T^2 \leq C_7$.
	By the same argument, we further obtain:
	\begin{align*}
	&\mathbb{E}\left[\int_0^T\left(v^1_t-R_t \mathbb{E}\left[\hat{v}^1_t\right]\right)^{\top} \bar{P}_t\left(v^1_t-R_t \mathbb{E}\left[\hat{v}^1_t\right]\right) -\left(v^1_t-R_t \frac{1}{N-1}\sum_{j\neq 1}\hat{v}^j_t\right)^{\top} \bar{P}_t\left(v^1_t-R_t \frac{1}{N-1}\sum_{j\neq 1}\hat{v}^j_t\right)dt\right] \\
	&= \mathbb{E}\left[\int_0^T\left(v^1_t-R_t \mathbb{E}\left[\hat{v}^1_t\right]+v^1_t-R_t \frac{1}{N-1}\sum_{j\neq 1}\hat{v}^j_t\right)^{\top}\bar{P}_t\left(v^1_t-R_t \mathbb{E}\left[\hat{v}^1_t\right]-v^1_t+R_t \frac{1}{N-1}\sum_{j\neq 1}\hat{v}^j_t\right)dt\right] \\
	&\leq \sqrt{4}C_6\left(\mathbb{E}\left[\int_0^T2\|v^1_t\|^2+\|R\|_T^2\mathbb{E}\left[\|\hat{v}^1_t\|^2\right]+\frac{\|R\|_T^2}{N-1}\sum_{j\neq 1}\|\hat{v}^j_t\|^2dt\right]\right)^\frac{1}{2}\cdot\sqrt{2}\left(\mathbb{E}\left[\int_0^T\frac{\|R\|_T^2}{N-1}\left\|\mathbb{E}\left[\hat{v}^j_t\right] -\hat{v}^j_t\right\|^2dt\right]\right)^\frac{1}{2} \\
	&\leq 4C_6(1+\|R\|_T^2)\left(\mathbb{E}\left[\int_0^T\|v^1_t\|^2+\|\hat{v}^1_t\|^2dt\right]\right)^\frac{1}{2}\cdot\left(\mathbb{E}\left[\int_0^T\frac{1}{N-1}\left\|\mathbb{E}\left[\hat{v}^j_t\right] -\hat{v}^j_t\right\|^2dt\right]\right)^\frac{1}{2} \\
	&\leq 4C_6(1+C_7)\left(\mathbb{E}\left[\int_0^T\|v^1_t\|^2+\|\hat{v}^1_t\|^2dt\right]\right)^\frac{1}{2}\cdot\left(\mathbb{E}\left[\int_0^T\frac{1}{N-1}\left\|\mathbb{E}\left[\hat{v}^j_t\right] -\hat{v}^j_t\right\|^2dt\right]\right)^\frac{1}{2},
	\end{align*}
	where the last inequality follows from the symmetry of $\left\{\hat{v}^j_{\cdot}\right\}_{j=1}^N$. Finally, for the cross term, it is:
	\begin{align*}
	&\mathbb{E}\left[\int_0^T2\left(x^1_t-\bar{S}_t \mathbb{E}\left[\hat{x}^1_t\right]\right)^{\top} N_t\left(v^1_t-\bar{R}_t \mathbb{E}\left[\hat{v}^1_t\right]\right) - 2\left(y^1_t-\bar{S}_t \frac{1}{N-1}\sum_{j\neq 1}\hat{y}^j_t\right)^{\top} N_t\left(v^1_t-\bar{R}_t \frac{1}{N-1}\sum_{j\neq 1}\hat{v}^j_t\right)dt\right] \\
	&= 2\mathbb{E}\left[\int_0^T\left(x^1_t-\bar{S}_t \mathbb{E}\left[\hat{x}^1_t\right]\right)^{\top} N_t\left(v^1_t-\bar{R}_t \mathbb{E}\left[\hat{v}^1_t\right]\right) - \left(x^1_t-\bar{S}_t \mathbb{E}\left[\hat{x}^1_t\right]\right)^{\top} N_t\left(v^1_t-\bar{R}_t \frac{1}{N-1}\sum_{j\neq 1}\hat{v}^j_t\right)dt\right] \\
	&\quad+ 2\mathbb{E}\left[\int_0^T\left(x^1_t-\bar{S}_t \mathbb{E}\left[\hat{x}^1_t\right]\right)^{\top} N_t\left(v^1_t-\bar{R}_t \frac{1}{N-1}\sum_{j\neq 1}\hat{v}^j_t\right) - \left(y^1_t-\bar{S}_t \frac{1}{N-1}\sum_{j\neq 1}\hat{y}^j_t\right)^{\top} N_t\left(v^1_t-\bar{R}_t \frac{1}{N-1}\sum_{j\neq 1}\hat{v}^j_t\right)dt\right] \\
	&= 2\mathbb{E}\left[\int_0^T\left(x^1_t-\bar{S}_t \mathbb{E}\left[\hat{x}^1_t\right]\right)^{\top} N_t\bar{R}_t\left( \frac{1}{N-1}\sum_{j\neq 1}\hat{v}^j_t- \mathbb{E}\left[\hat{v}^1_t\right]\right)dt\right] \\
	&\quad+ 2\mathbb{E}\left[\int_0^T\left(v^1_t-\bar{R}_t \frac{1}{N-1}\sum_{j\neq 1}\hat{v}^j_t\right)^{\top} N_t^{\top}\left(x^1_t-\bar{S}_t \mathbb{E}\left[\hat{x}^1_t\right]-y^1_t+\bar{S}_t \frac{1}{N-1}\sum_{j\neq 1}\hat{y}^j_t\right)dt\right] \\
	&\leq 2\sqrt{2}C_6\|\bar{R}\|_T\left(\mathbb{E}\left[\int_0^T\|x^1_t\|^2+ \|\bar{S}\|_T^2\mathbb{E}\left[\|\hat{x}^1_t\|^2\right]dt\right]\right)^{\frac{1}{2}}\cdot\left(\mathbb{E}\left[\int_0^T\left\|\frac{1}{N-1}\sum_{j\neq 1}  \left(\mathbb{E}\left[\hat{v}^j_t\right]-\hat{v}^j_t\right)\right\|^2dt\right]\right)^{\frac{1}{2}} \\
	&\quad+ 4C_6\left(\mathbb{E}\left[\int_0^T\|v^1_t\|^2+\frac{\|\bar{R}\|_T^2}{N-1}\sum_{j\neq 1}\|\hat{v}^j_t\|^2dt\right]\right)^{\frac{1}{2}}\cdot\left(\mathbb{E}\left[\int_0^T\|x^1_t-y^1_t\|^2+\|\bar{S}\|_T^2\left\|\frac{1}{N-1}\sum_{j\neq 1} \left(\mathbb{E}\left[\hat{x}^j_t\right]- \hat{y}^j_t\right)\right\|^2dt\right]\right)^{\frac{1}{2}} \\
	&\leq 2\sqrt{2}C_6\|\bar{R}\|_T(1+\|\bar{S}\|_T^2)^{\frac{1}{2}}\left(\mathbb{E}\left[\int_0^T\|x^1_t\|^2+\|\hat{x}^1_t\|^2dt\right]\right)^{\frac{1}{2}}\cdot\left(\mathbb{E}\left[\int_0^T\frac{1}{N-1}\left\| \mathbb{E}\left[\hat{v}^j_s\right]-\hat{v}^j_s\right\|^2dt\right]\right)^{\frac{1}{2}} \\
	&\quad+ 4\sqrt{2}C_6(1+\|\bar{R}\|_T^2)^{\frac{1}{2}}(1+\|\bar{S}\|_T^2)^{\frac{1}{2
	}}\left(\mathbb{E}\left[\int_0^T\|v^1_t\|^2+\|\hat{v}^1_t\|^2dt\right]\right)^{\frac{1}{2}} \cdot\left(\mathbb{E}\left[\int_0^T\|x^1_t-y^1_t\|^2+\|\hat{x}^1_t-\hat{x}^1_t\|^2+\frac{1}{N-1}\left\| \mathbb{E}\left[\hat{x}^j_t\right]-\hat{x}^j_t\right\|^2dt\right]\right)^{\frac{1}{2}} \\
	&\leq 4\sqrt{2}C_6(1+C_7)\left(\mathbb{E}\left[\int_0^T\|x^1_t\|^2+\|v^1_t\|^2+\|\hat{x}^1_t\|^2+\|\hat{v}^1_t\|^2dt\right]\right)^{\frac{1}{2}} \\
	&\quad\cdot\left(\mathbb{E}\left[\int_0^T\|x^1_t-y^1_t\|^2+\|\hat{x}^1_t-\hat{x}^1_t\|^2+\frac{1}{N-1}\left\| \mathbb{E}\left[\hat{x}^j_t\right]-\hat{x}^j_t\right\|^2+\frac{1}{N-1}\left\| \mathbb{E}\left[\hat{v}^j_s\right]-\hat{v}^j_s\right\|^2dt\right]\right)^{\frac{1}{2}},
	\end{align*}
	where the second last inequality follows from \eqref{lqmf.mfg.appendix.inequality_in_lemma4_1.v} and \eqref{lqmf.mfg.appendix.inequality_in_lemma4_3.x.2}.
	We take the difference of \eqref{lqmf.N_player.objective.appendix} and \eqref{lqmf.mfg.objective.appendix}, and substitute the above inequalities into it. By Cauchy-Schwarz inequality, we have
	\begin{align*}
	&J(v^1)-\mathcal{J}(v^1,\hat{v}^{-1}) \\
	&\leq 12C_6(1+C_7)\Bigg\{\left(\mathbb{E}\bigg[ \int_0^T \|x^1_t\|^2+\|y^1_t\|^2+\|v^1_t\|^2+\|\hat{x}^1_t\|^2+\|\hat{y}^1_t\|^2+\|\hat{v}^1_t\|^2 dt\bigg]\right)^{\frac{1}{2}}\\
	&\qquad \qquad \qquad \qquad\cdot \left(\mathbb{E}\bigg[\int_0^T \|x^1_t-y^1_t\|^2+\|\hat{x}^1_t-\hat{x}^1_t\|^2+ \frac{1}{N-1}\left\|\mathbb{E}\left[\hat{x}^1_t\right] -\hat{x}^j_t\right\|^2+\frac{1}{N-1}\left\| \mathbb{E}\left[\hat{v}^1_t\right]-\hat{v}^j_t\right\|^2 d t\bigg]\right)^{\frac{1}{2}} \\
	&\qquad \qquad \qquad \qquad+ \left(\mathbb{E}\bigg[\|x^1_T\|^2+\|y^1_T\|^2+\|\hat{x}^1_T\|^2+\|\hat{y}^1_T\|^2\bigg]\right)^{\frac{1}{2}} \cdot \left(\mathbb{E}\bigg[\|x^1_T-y^1_T\|^2+\|\hat{x}^1_T-\hat{x}^1_T\|^2+ \frac{1}{N-1}\left\|\mathbb{E}\left[x^1_T\right] -\hat{x}^j_T\right\|^2 \bigg]\right)^{\frac{1}{2}}\Bigg\} \\
	\leq& 12C_6(1+C_7)\left(\mathbb{E}\bigg[ \int_0^T \|x^1_t\|^2+\|y^1_t\|^2+\|v^1_t\|^2+\|\hat{x}^1_t\|^2+\|\hat{y}^1_t\|^2+\|\hat{v}^1_t\|^2 dt\bigg]+\mathbb{E}\bigg[\|x^1_T\|^2+\|y^1_T\|^2+\|\hat{x}^1_T\|^2+\|\hat{y}^1_T\|^2\bigg]\right)^{\frac{1}{2}}\\
	&\qquad \qquad \qquad\cdot \Bigg(\mathbb{E}\bigg[\int_0^T \|x^1_t-y^1_t\|^2+\|\hat{x}^1_t-\hat{y}^1_t\|^2+ \frac{1}{N-1}\left\|\mathbb{E}\left[\hat{x}^1_t\right] -\hat{x}^j_t\right\|^2+\frac{1}{N-1}\left\| \mathbb{E}\left[\hat{v}^1_t\right]-\hat{v}^j_t\right\|^2 d t\bigg] \\
	&\qquad \qquad \qquad \qquad+\mathbb{E}\bigg[\|x^1_T-y^1_T\|^2+\|\hat{x}^1_T-\hat{y}^1_T\|^2 + \frac{1}{N-1}\left\|\mathbb{E}\left[x^1_T\right] -\hat{x}^j_T\right\|^2 \bigg]\Bigg)^{\frac{1}{2}} \\
	\leq& 36C_6(1+C_7)\Bigg(\mathbb{E}\bigg[ \int_0^T \|x^1_t\|^2+\|x^1_t-y^1_t\|^2+\|v^1_t\|^2+\|\hat{x}^1_t\|^2+\|\hat{x}^1_t-\hat{y}^1_t\|^2+\|\hat{v}^1_t\|^2 dt\bigg]\\
	&\qquad \qquad \qquad \qquad+\mathbb{E}\bigg[\|x^1_T\|^2+\|x^1_T-y^1_T\|^2+\|\hat{x}^1_T\|^2+\|\hat{x}^1_T-\hat{y}^1_T\|^2\bigg]\Bigg)^{\frac{1}{2}}\\
	&\qquad \qquad \quad\cdot \Bigg(\mathbb{E}\bigg[\int_0^T \|x^1_t-y^1_t\|^2+\|\hat{x}^1_t-\hat{y}^1_t\|^2+ \frac{1}{N-1}\left\|\mathbb{E}\left[\hat{x}^1_t\right] -\hat{x}^j_t\right\|^2+\frac{1}{N-1}\left\| \mathbb{E}\left[\hat{v}^1_t\right]-\hat{v}^j_t\right\|^2 d t\bigg] \\
	&\qquad \qquad \qquad \qquad+\mathbb{E}\bigg[\|x^1_T-y^1_T\|^2+\|\hat{x}^1_T-\hat{y}^1_T\|^2 + \frac{1}{N-1}\left\|\mathbb{E}\left[x^1_T\right] -\hat{x}^j_T\right\|^2 \bigg]\Bigg)^{\frac{1}{2}} \\
	&\leq 36C_6(1+C_7)\Bigg(\mathbb{E}\left[\left\|x^1_T\right\|^2+\int_0^T  \left\|x^1_t\right\|^2+ \left\|v^1_t\right\|^2dt\right]+\mathbb{E}\left[\left\|\hat{x}^1_T\right\|^2+\int_0^T  \left\|\hat{x}^1_t\right\|^2+ \left\|\hat{v}^1_t\right\|^2dt\right]\\
	&\qquad \qquad \qquad \qquad+\mathbb{E}\bigg[ \|x^1_T-y^1_T\|^2+\int_0^T \|x^1_t-y^1_t\|^2 dt\bigg]+\mathbb{E}\bigg[ \|\hat{x}^1_T-\hat{y}^1_T\|^2+\int_0^T \|\hat{x}^1_t-\hat{y}^1_t\|^2 dt\bigg]\Bigg)^{\frac{1}{2}}\\
	&\qquad \qquad \qquad\cdot \Bigg(\frac{1}{N-1}\mathbb{E}\bigg[\left\|\mathbb{E}\left[\hat{x}^1_T\right] -\hat{x}^j_T\right\|^2+\int_0^T \left\|\mathbb{E}\left[\hat{x}^1_t\right] -\hat{x}^j_t\right\|^2+\left\| \mathbb{E}\left[\hat{v}^1_t\right]-\hat{v}^j_t\right\|^2 d t\bigg] +\mathbb{E}\bigg[ \|x^1_T-y^1_T\|^2+\int_0^T \|x^1_t-y^1_t\|^2 dt\bigg] \\
	&\qquad \qquad \qquad \qquad+\mathbb{E}\bigg[ \|\hat{x}^1_T-\hat{y}^1_T\|^2+\int_0^T \|\hat{x}^1_t-\hat{y}^1_t\|^2 dt\bigg]\Bigg)^{\frac{1}{2}} \\
	\leq& 36C_6(1+C_7)\Bigg((T+1)\sup_{0\leq u\leq T}\mathbb{E}\left[\left\|x^1_u\right\|^2\right]+\mathbb{E}\left[\left\|\hat{x}^1_T\right\|^2+\int_0^T  \left\|\hat{x}^1_t\right\|^2+ \left\|\hat{v}^1_t\right\|^2dt\right]\\
	&\qquad \qquad \qquad \qquad+(T+1)\mathbb{E}\bigg[ \sup_{0\leq u\leq T}\|x^1_u-y^1_u\|^2+ \sup_{0\leq u\leq T}\|\hat{x}^1_u-\hat{y}^1_u\|^2\bigg]+ \int_0^T\mathbb{E}\left[\left\|v^1_t\right\|^2\right]dt\Bigg)^{\frac{1}{2}}\\
	&\qquad \qquad \qquad\cdot \Bigg(\frac{1}{N-1}\mathbb{E}\bigg[\left\|\mathbb{E}\left[\hat{x}^1_T\right] -\hat{x}^j_T\right\|^2+\int_0^T \left\|\mathbb{E}\left[\hat{x}^1_t\right] -\hat{x}^j_t\right\|^2+\left\| \mathbb{E}\left[\hat{v}^1_t\right]-\hat{v}^j_t\right\|^2 d t\bigg] \\
	&\qquad \qquad \qquad \qquad+(T+1)\mathbb{E}\bigg[ \sup_{0\leq u\leq T}\|x^1_u-y^1_u\|^2+ \sup_{0\leq u\leq T}\|\hat{x}^1_u-\hat{y}^1_u\|^2\bigg]\Bigg)^{\frac{1}{2}} \\
	&\leq \frac{36C_6(1+C_7)}{\sqrt{N-1}}\Bigg((M_{e,2}+M_{d,5}(T+1))\left\|\mathbb{E}\left[x^1_0\right]\right\|^2 +(M_{e,1}+M_{d,4}(T+1)+M_{d,6}(T+1))\left(\mathbb{E}\left[\left\|\mathbb{E}\left[x^1_0\right]-x^1_0\right\|^2\right] +T\|\sigma\|_T^2\right) \\
	&\qquad \qquad \qquad \qquad+M_{d,5}(T+1)\int_0^t\left\|\mathbb{E}\left[v^1_s\right]\right\|^2 ds + \int_0^T\mathbb{E}\left[\left\|v^1_t\right\|^2\right]dt+(M_{d,4}+M_{d,7})(T+1) \int_0^t\mathbb{E}\left[\left\|\mathbb{E}\left[v^1_s\right]-v^1_s\right\|^2\right] ds\Bigg)^{\frac{1}{2}}\\
	&\quad\cdot \bigg((M_{e,1}+M_{d,6}(N-1)(T+1))\left(\mathbb{E}\left[\left\|\mathbb{E}\left[x^i_0\right]-x^i_0\right\|^2\right]+T\|\sigma\|_T^2\right) +M_{d,7}(N-1)(T+1)\int_0^t\mathbb{E}\left[\left\|\mathbb{E}\left[v^1_s\right]-v^1_s\right\|^2\right] ds\bigg)^{\frac{1}{2}} \\
	&\leq \frac{36C_6(1+C_7)}{\sqrt{N-1}}\Bigg((M_{e,2}+M_{d,5}(T+1))\left\|\mathbb{E}\left[x^1_0\right]\right\|^2 +(M_{e,1}+M_{d,4}(T+1)+M_{d,6}(T+1))\left(\mathbb{E}\left[\left\|\mathbb{E}\left[x^1_0\right]-x^1_0\right\|^2\right] +T\|\sigma\|_T^2\right) \\
	&\qquad \qquad \qquad \qquad+M_{d,5}(T+1)\int_0^t\left\|\mathbb{E}\left[v^1_s\right]\right\|^2 ds + \int_0^T\mathbb{E}\left[\left\|v^1_t\right\|^2\right]dt+(M_{d,4}+M_{d,7})(T+1) \int_0^t\mathbb{E}\left[\left\|\mathbb{E}\left[v^1_s\right]-v^1_s\right\|^2\right] ds\Bigg)
	\end{align*}
	where the second last inequality follows from the estimates in Lemma \ref{lqmf.mfg.lemma.boundedness.x-Ex}, Lemma \ref{lqmf.mfg.lemma.boundedness.Ex2}, Lemma \ref{lqmf.mfg.lemma.boundedness.Ex2.dummy} and Lemma \ref{lqmf.mfg.lemma.approximation_state.pertubed}.
\end{proof}

\section{Proofs of Statements in Section \ref{lqmf.section.emftc}} \label{lqmf.section.appendixA5}

\subsection{Proof of Theorem \ref{lqmf.mfc.theorem_smp}}

\begin{proof}
	We define
	\begin{align*}
	\widetilde{x}_t := \lim_{\theta \to 0}\frac{ x^{v+\theta\widetilde{v}}_t - x^{v}_t}{\theta}.
	\end{align*}
	By the argument in \cite{bensoussan2017interpretation}, we have
	
	\begin{align*}
	\lim_{\theta \to 0}\frac{ \mathbb{E}\left[x^{v+\theta\widetilde{v}}_t\right] - \mathbb{E}\left[x^{v}_t\right]}{\theta} = \mathbb{E}\left[\widetilde{x}_t\right],
	\end{align*}
	hence
	\begin{align*}
	d \widetilde{x}_t=\left(\mathsf{A}_t \widetilde{x}_t+\mathsf{B}_t \widetilde{v}_t+\bar{\mathsf{A}}_t \mathbb{E}\left[\widetilde{x}_t\right]+\bar{\mathsf{B}}_t \mathbb{E}\left[\widetilde{v}_t\right]\right) d t, \quad \widetilde{x}(0)=0,
	\end{align*}
	and
	\begin{align*}
	\frac{d J(v+\theta\widetilde{v})}{d\theta}\bigg|_{\theta=0} =& \mathbb{E}\Bigg[ \int_0^T \widetilde{x}_t^{\top} \mathsf{Q}_t x_t+\widetilde{v}_t^{\top} \mathsf{P}_t v_t+\left(\widetilde{x}_t-\mathsf{S}_t \mathbb{E}\left[\widetilde{x}_t\right]\right)^{\top} \bar{\mathsf{Q}}_t\left(x_t-\mathsf{S}_t \mathbb{E}\left[x_t\right]\right) \\
	&\quad\quad+ \left(\widetilde{x}_t-\bar{\mathsf{S}}_t \mathbb{E}\left[\widetilde{x}_t\right]\right)^{\top} \mathsf{N}_t\left(v_t-\bar{\mathsf{R}}_t \mathbb{E}\left[v_t\right]\right) + \left(\widetilde{v}_t-\bar{\mathsf{R}}_t \mathbb{E}\left[\widetilde{v}_t\right]\right)^{\top} \mathsf{N}_t^{\top}\left(x_t-\bar{\mathsf{S}}_t \mathbb{E}\left[x_t\right]\right) \\
	&\quad\quad+ \left(\widetilde{v}_t-\mathsf{R}_t \mathbb{E}\left[\widetilde{v}_t\right]\right)^{\top} \bar{\mathsf{P}}_t\left(v_t-\mathsf{R}_t \mathbb{E}\left[v_t\right]\right) d t\Bigg] \\
	&\quad\quad+\mathbb{E}\left[ \widetilde{x}_T^{\top} \mathsf{Q}_T x_T+\left(\widetilde{x}_T-\mathsf{S}_T\mathbb{E}\left[\widetilde{x}_T\right] \right)^{\top} \bar{\mathsf{Q}}_T\left(x_T-\mathsf{S}_T \mathbb{E}\left[x_T\right]\right)\right].
	\end{align*}
	We can obtain the Euler first order condition: for any $\widetilde{v} \in $, there exist the optimal state process and the optimal control $\hat{x}_t$ and $\hat{v}_t$ such that,
	\begin{align*}
	0 \leq \frac{d J(\hat{v}+\theta\widetilde{v})}{d\theta}\bigg|_{\theta=0} =& \mathbb{E}\Bigg[ \int_0^T \widetilde{x}_t^{\top} \mathsf{Q}_t \hat{x}_t+\widetilde{v}_t^{\top} \mathsf{P}_t \hat{v}_t+\left(\widetilde{x}_t-\mathsf{S}_t \mathbb{E}\left[\widetilde{x}_t\right]\right)^{\top} \bar{\mathsf{Q}}_t\left(\hat{x}_t-\mathsf{S}_t \mathbb{E}\left[\hat{x}_t\right]\right) \\
	&\quad\quad+ \left(\widetilde{x}_t-\bar{\mathsf{S}}_t \mathbb{E}\left[\widetilde{x}_t\right]\right)^{\top} \mathsf{N}_t\left(\hat{v}_t-\bar{\mathsf{R}}_t \mathbb{E}\left[\hat{v}_t\right]\right) + \left(\widetilde{v}_t-\bar{\mathsf{R}}_t \mathbb{E}\left[\widetilde{v}_t\right]\right)^{\top} \mathsf{N}_t^{\top}\left(\hat{x}_t-\bar{\mathsf{S}}_t \mathbb{E}\left[\hat{x}_t\right]\right) \\
	&\quad\quad+ \left(\widetilde{v}_t-\mathsf{R}_t \mathbb{E}\left[\widetilde{v}_t\right]\right)^{\top} \bar{\mathsf{P}}_t\left(\hat{v}_t-\mathsf{R}_t \mathbb{E}\left[\hat{v}_t\right]\right) d t\Bigg] \\
	&\quad\quad+\mathbb{E}\left[ \widetilde{x}_T^{\top} \mathsf{Q}_T \hat{x}_T+\left(\widetilde{x}_T-\mathsf{S}_T\mathbb{E}\left[\widetilde{x}_T\right] \right)^{\top} \bar{\mathsf{Q}}_T\left(\hat{x}_T-\mathsf{S}_T \mathbb{E}\left[\hat{x}_T\right]\right)\right].
	\end{align*}
	By arrangement, the inequality becomes
	\begin{align}
	\nonumber 0 \leq& \mathbb{E}\Bigg[\int_0^T \widetilde{x}_t^{\top}(\mathsf{Q}_t + \bar{\mathsf{Q}}_t) \hat{x}_t +  \widetilde{x}_t^{\top}( - \bar{\mathsf{Q}}_t \mathsf{S}_t - \mathsf{S}_t^{\top} \bar{\mathsf{Q}}_t  + \mathsf{S}_t^{\top} \bar{\mathsf{Q}}_t \mathsf{S}_t ) \mathbb{E}\left[\hat{x}_t\right] \\
	\nonumber &\quad\quad+\widetilde{x}_t^{\top}\mathsf{N}_t\hat{v}_t + \widetilde{x}_t^{\top}( - \mathsf{N}_t\bar{\mathsf{R}}_t - \bar{\mathsf{S}}_t^{\top}\mathsf{N}_t  + \bar{\mathsf{S}}_t^{\top}\mathsf{N}_t\bar{\mathsf{R}}_t ) \mathbb{E}\left[\hat{v}_t\right] \\
	\nonumber &\quad\quad+\widetilde{v}_t^{\top}\mathsf{N}_t^{\top}\hat{x}_t + \widetilde{v}_t^{\top}( - \mathsf{N}_t^{\top}\bar{\mathsf{S}}_t - \bar{\mathsf{R}}_t^{\top}\mathsf{N}_t^{\top}  + \bar{\mathsf{R}}_t^{\top}\mathsf{N}_t^{\top}\bar{\mathsf{S}}_t ) \mathbb{E}\left[\hat{x}_t\right] \\
	\nonumber &\quad\quad+   \widetilde{v}_t^{\top}(\mathsf{P}_t + \bar{\mathsf{P}}_t) \hat{v}_t +  \widetilde{v}_t^{\top}( - \bar{\mathsf{P}}_t \mathsf{R}_t - \mathsf{R}_t^{\top} \bar{\mathsf{P}}_t  + \mathsf{R}_t^{\top} \bar{\mathsf{P}}_t \mathsf{R}_t ) \mathbb{E}\left[\hat{v}_t\right] dt\Bigg] \\
	\label{lqmf.mfc.Euler_condition} &\quad\quad+\mathbb{E}\left[\widetilde{x}_T^{\top}(\mathsf{Q}_T + \bar{\mathsf{Q}}_T) \hat{x}_T +   \widetilde{x}_T^{\top}( - \bar{\mathsf{Q}}_T \mathsf{S}_T - \mathsf{S}_T^{\top} \bar{\mathsf{Q}}_T  + \mathsf{S}_T^{\top} \bar{\mathsf{Q}}_T \mathsf{S}_T ) \mathbb{E}\left[\hat{x}_T\right]\right].
	\end{align}
	By the chain rule,
	\begin{align*}
	d \left(\widetilde{x}_t^{\top} p_t\right) =& \widetilde{x}_t^{\top} d p_t + d\widetilde{x}_t^{\top} p_t \\
	=& - \left(\widetilde{x}_t^{\top}\mathsf{A}_t^{\top} p_t  + \widetilde{x}_t^{\top}\bar{\mathsf{A}}_t^{\top} \mathbb{E}\left[p_t\right] \right.\\
	&+\widetilde{x}_t^{\top}(\mathsf{Q}_t + \bar{\mathsf{Q}}_t) \hat{x}_t +   \widetilde{x}_t^{\top}( - \bar{\mathsf{Q}}_t \mathsf{S}_t - \mathsf{S}_t^{\top} \bar{\mathsf{Q}}_t  + \mathsf{S}_t^{\top} \bar{\mathsf{Q}}_t \mathsf{S}_t ) \mathbb{E}\left[\hat{x}_t\right]  +\widetilde{x}_t^{\top}\mathsf{N}_t\hat{v}_t + \widetilde{x}_t^{\top}( - \mathsf{N}_t\bar{\mathsf{R}}_t - \bar{\mathsf{S}}_t^{\top}\mathsf{N}_t  + \bar{\mathsf{S}}_t^{\top}\mathsf{N}_t\bar{\mathsf{R}}_t ) \mathbb{E}\left[\hat{v}_t\right]  \big)dt \\
	&+ \left(\widetilde{x}_t^{\top} \mathsf{A}_t^{\top}p_t +\widetilde{v}_t^{\top}\mathsf{B}_t^{\top}p_t +\mathbb{E}\left[\widetilde{x}_t\right]^{\top}\bar{\mathsf{A}}_t^{\top}p_t +\mathbb{E}\left[\widetilde{v}_t\right]^{\top}\bar{\mathsf{B}}_t^{\top}p_t \right) dt + \widetilde{x}_t^{\top} \theta_t dW_t.
	\end{align*}
	Taking integration against time and expectation of both sides,
	\begin{align*}
	&\mathbb{E}\left[\widetilde{x}_T^{\top}(\mathsf{Q}_T + \bar{\mathsf{Q}}_T) \hat{x}_T +   \widetilde{x}_T^{\top}( - \bar{\mathsf{Q}}_T \mathsf{S}_T - \mathsf{S}_T^{\top} \bar{\mathsf{Q}}_T  + \mathsf{S}_T^{\top} \bar{\mathsf{Q}}_T \mathsf{S}_T ) \mathbb{E}\left[\hat{x}_T\right]\right] \\
	=& \mathbb{E}\left[\int_0^T - \widetilde{x}_t^{\top}(\mathsf{Q}_t + \bar{\mathsf{Q}}_t) \hat{x}_t -  \widetilde{x}_t^{\top}( - \bar{\mathsf{Q}}_t \mathsf{S}_t - \mathsf{S}_t^{\top} \bar{\mathsf{Q}}_t  + \mathsf{S}_t^{\top} \bar{\mathsf{Q}}_t \mathsf{S}_t ) \mathbb{E}\left[\hat{x}_t\right]  -\widetilde{x}_t^{\top}\mathsf{N}_t\hat{v}_t - \widetilde{x}_t^{\top}( - \mathsf{N}_t\bar{\mathsf{R}}_t - \bar{\mathsf{S}}_t^{\top}\mathsf{N}_t  + \bar{\mathsf{S}}_t^{\top}\mathsf{N}_t\bar{\mathsf{R}}_t ) \mathbb{E}\left[\hat{v}_t\right]  dt\right] \\
	&+ \mathbb{E}\left[\int_0^T \widetilde{v}_t^{\top}\mathsf{B}_t^{\top}p_t  +\widetilde{v}_t^{\top}\bar{\mathsf{B}}_t^{\top}\mathbb{E}\left[p_t\right]  dt\right]
	\end{align*}
	Substituting the above equality into the Euler first order condition \eqref{lqmf.mfc.Euler_condition}, we have, for any $\widetilde{v}_t$,
	\begin{align}
	\nonumber 0 \leq& \ \mathbb{E}\left[\int_0^T \widetilde{v}_t^{\top}\mathsf{B}_t^{\top}p_t  +\widetilde{v}_t^{\top}\bar{\mathsf{B}}_t^{\top}\mathbb{E}\left[p_t\right]+\widetilde{v}_t^{\top}\mathsf{N}_t^{\top}\hat{x}_t + \widetilde{v}_t^{\top}( - \mathsf{N}_t^{\top}\bar{\mathsf{S}}_t - \bar{\mathsf{R}}_t^{\top}\mathsf{N}_t^{\top}  + \bar{\mathsf{R}}_t^{\top}\mathsf{N}_t^{\top}\bar{\mathsf{S}}_t ) \mathbb{E}\left[\hat{x}_t\right] \right.\\
	\nonumber &+   \widetilde{v}_t^{\top}(\mathsf{P}_t + \bar{\mathsf{P}}_t) \hat{v}_t +  \widetilde{v}_t^{\top}( - \bar{\mathsf{P}}_t \mathsf{R}_t - \mathsf{R}_t^{\top} \bar{\mathsf{P}}_t  + \mathsf{R}_t^{\top} \bar{\mathsf{P}}_t \mathsf{R}_t ) \mathbb{E}\left[\hat{v}_t\right] dt\bigg]
	\end{align}
	Thus, for any $t \in [0,T]$,
	\begin{align*}
	0 = \mathsf{B}_t^{\top}p_t  +\bar{\mathsf{B}}_t^{\top}\mathbb{E}\left[p_t\right]+\mathsf{N}_t^{\top}\hat{x}_t + ( - \mathsf{N}_t^{\top}\bar{\mathsf{S}}_t - \bar{\mathsf{R}}_t^{\top}\mathsf{N}_t^{\top}  + \bar{\mathsf{R}}_t^{\top}\mathsf{N}_t^{\top}\bar{\mathsf{S}}_t ) \mathbb{E}\left[\hat{x}_t\right] + (\mathsf{P}_t + \bar{\mathsf{P}}_t) \hat{v}_t +  ( - \bar{\mathsf{P}}_t \mathsf{R}_t - \mathsf{R}_t^{\top} \bar{\mathsf{P}}_t  + \mathsf{R}_t^{\top} \bar{\mathsf{P}}_t \mathsf{R}_t ) \mathbb{E}\left[\hat{v}_t\right].
	\end{align*}
	We take expectation on both sides and rearrange it to get:
	\begin{align}
	\label{lqmf.mfc.optimal_control.mean}
	\mathbb{E}\left[\hat{v}_t\right] = -\left(\mathsf{P}_t + (I-\bar{\mathsf{R}}_t^{\top})\bar{\mathsf{P}}_t(I-\bar{\mathsf{R}}_t)\right)^{-1}\left((\mathsf{B}_t^{\top}+\bar{\mathsf{B}}_t^{\top})\mathbb{E}\left[p_t\right] +  (I-\bar{\mathsf{R}}_t^{\top})\mathsf{N}_t^{\top}(I-\bar{\mathsf{S}}_t) \mathbb{E}\left[\hat{x}_t\right]\right).
	\end{align}
	Therefore,
	\begin{align*}
	\hat{v}_t =& -(\mathsf{P}_t + \bar{\mathsf{P}}_t)^{-1}\left(\mathsf{B}_t^{\top}p_t  +\bar{\mathsf{B}}_t^{\top}\mathbb{E}\left[p_t\right]+\mathsf{N}_t^{\top}\hat{x}_t + \left( - \mathsf{N}_t^{\top} + (I-\bar{\mathsf{R}}_t^{\top})\mathsf{N}_t^{\top}(I-\bar{\mathsf{S}}_t) \right) \mathbb{E}\left[\hat{x}_t\right] \right.\\
	&+  \left(-\bar{\mathsf{P}}_t + (I-\bar{\mathsf{R}}_t^{\top})\bar{\mathsf{P}}_t(I-\bar{\mathsf{R}}_t)\right) \mathbb{E}\left[\hat{v}_t\right]\big) \\
	=& -(\mathsf{P}_t + \bar{\mathsf{P}}_t)^{-1}\left(\mathsf{B}_t^{\top}p_t  +\bar{\mathsf{B}}_t^{\top}\mathbb{E}\left[p_t\right]+\mathsf{N}_t^{\top}\hat{x}_t + \left( - \mathsf{N}_t^{\top} + (I-\bar{\mathsf{R}}_t^{\top})\mathsf{N}_t^{\top}(I-\bar{\mathsf{S}}_t) \right) \mathbb{E}\left[\hat{x}_t\right] \right.\\
	&-  \left(-\bar{\mathsf{P}}_t + (I-\bar{\mathsf{R}}_t^{\top})\bar{\mathsf{P}}_t(I-\bar{\mathsf{R}}_t)\right) \left(\mathsf{P}_t + (I-\bar{\mathsf{R}}_t^{\top})\bar{\mathsf{P}}_t(I-\bar{\mathsf{R}}_t)\right)^{-1}\left((\mathsf{B}_t^{\top}+\bar{\mathsf{B}}_t^{\top})\mathbb{E}\left[p_t\right] +  (I-\bar{\mathsf{R}}_t^{\top})\mathsf{N}_t^{\top}(I-\bar{\mathsf{S}}_t) \mathbb{E}\left[\hat{x}_t\right]\right)\big) \\
	=& -(\mathsf{P}_t + \bar{\mathsf{P}}_t)^{-1}\mathsf{B}_t^{\top}p_t  -(\mathsf{P}_t + \bar{\mathsf{P}}_t)^{-1}\bar{\mathsf{B}}_t^{\top}\mathbb{E}\left[p_t\right] \\
	&+(\mathsf{P}_t + \bar{\mathsf{P}}_t)^{-1}  \left(-\bar{\mathsf{P}}_t + (I-\bar{\mathsf{R}}_t^{\top})\bar{\mathsf{P}}_t(I-\bar{\mathsf{R}}_t)\right) \left(\mathsf{P}_t + (I-\bar{\mathsf{R}}_t^{\top})\bar{\mathsf{P}}_t(I-\bar{\mathsf{R}}_t)\right)^{-1}(\mathsf{B}_t^{\top}+\bar{\mathsf{B}}_t^{\top})\mathbb{E}\left[p_t\right]  \\
	&-(\mathsf{P}_t + \bar{\mathsf{P}}_t)^{-1}\mathsf{N}_t^{\top}\hat{x}_t -(\mathsf{P}_t + \bar{\mathsf{P}}_t)^{-1} \left( - \mathsf{N}_t^{\top} + (I-\bar{\mathsf{R}}_t^{\top})\mathsf{N}_t^{\top}(I-\bar{\mathsf{S}}_t) \right) \mathbb{E}\left[\hat{x}_t\right] \\
	&+(\mathsf{P}_t + \bar{\mathsf{P}}_t)^{-1}  \left(-\bar{\mathsf{P}}_t + (I-\bar{\mathsf{R}}_t^{\top})\bar{\mathsf{P}}_t(I-\bar{\mathsf{R}}_t)\right) \left(\mathsf{P}_t + (I-\bar{\mathsf{R}}_t^{\top})\bar{\mathsf{P}}_t(I-\bar{\mathsf{R}}_t)\right)^{-1} (I-\bar{\mathsf{R}}_t^{\top})\mathsf{N}_t^{\top}(I-\bar{\mathsf{S}}_t) \mathbb{E}\left[\hat{x}_t\right] \\
	=& (\mathsf{P}_t + \bar{\mathsf{P}}_t)^{-1}\mathsf{B}_t^{\top}\left(\mathbb{E}\left[p_t\right]-p_t\right)-\left(\mathsf{P}_t + (I-\bar{\mathsf{R}}_t^{\top})\bar{\mathsf{P}}_t(I-\bar{\mathsf{R}}_t)\right)^{-1}(\mathsf{B}_t^{\top}+\bar{\mathsf{B}}_t^{\top})\mathbb{E}\left[p_t\right]  \\
	&+(\mathsf{P}_t + \bar{\mathsf{P}}_t)^{-1}\mathsf{N}_t^{\top}\left(\mathbb{E}\left[\hat{x}_t\right]-\hat{x}_t\right)-\left(\mathsf{P}_t + (I-\bar{\mathsf{R}}_t^{\top})\bar{\mathsf{P}}_t(I-\bar{\mathsf{R}}_t)\right)^{-1} (I-\bar{\mathsf{R}}_t^{\top})\mathsf{N}_t^{\top}(I-\bar{\mathsf{S}}_t) \mathbb{E}\left[\hat{x}_t\right].
	\end{align*}
The last equality follows from:
\begin{align*}
&(\mathsf{P}_t + \bar{\mathsf{P}}_t)^{-1}  \left(-\bar{\mathsf{P}}_t + (I-\bar{\mathsf{R}}_t^{\top})\bar{\mathsf{P}}_t(I-\bar{\mathsf{R}}_t)\right) \left(\mathsf{P}_t + (I-\bar{\mathsf{R}}_t^{\top})\bar{\mathsf{P}}_t(I-\bar{\mathsf{R}}_t)\right)^{-1} =(\mathsf{P}_t + \bar{\mathsf{P}}_t)^{-1} - \left(\mathsf{P}_t + (I-\bar{\mathsf{R}}_t^{\top})\bar{\mathsf{P}}_t(I-\bar{\mathsf{R}}_t)\right)^{-1}.
\end{align*}		
\end{proof}


\section{Comparison with some existing works}  \label{lqmf.section.appendixB}

\subsection{Comparison with \cite{alasseur2020extended}} \label{lqmf.section.alasseur}

The work \cite{alasseur2020extended} considers a dynamic grid optimization problem in both the EMFG and EMFTC settings. The energy-grid system is supposed to comprise both local prosumers and an external producer. The representative agent, a local prosumer, has three state processes denoted by $Q_0$, $Q_1$, and $S$, where $Q_0$ represents the energy delivered to the grid by the external producer, $Q_1$ is the demand (positive in value) or supply (negative in value) rate of the representative agent, and $S(t)$ is the amount of energy stored in the local device by the representative agent at time $t$. The dynamics of these state processes are governed by the following dynamical systems:
\begin{align}
\label{lqmf.comp.jota.state}
\left\{\begin{array}{ccl}
dQ_0(t) &=& b_0(Q_0,t)dt + \beta_0db(t), \quad Q_0(0)=Q_{00}; \\
dQ_1(t) &=& b_1(Q_1,t)dt + \sigma(Q_1)dw(t) + \beta_1db(t), \quad Q_1(0)=Q_{10}; \\
dS(t) &=& v(t)dt, \quad S(0)=0,
\end{array}\right.
\end{align}
where $b_0(Q_0,t)$ and $b_1(Q_1,t)$ are assumed to be in mean reverting form
\begin{align}
	\label{lqmf.comp.mean_reverting}
	b_0(Q_0,t):=-\alpha_0(Q_0(t)-\gamma_0(t)); \quad b_1(Q_1,t):=-\alpha_1(Q_1(t)-\gamma_1(t)); \quad
	\sigma(Q_1):=\sigma.
\end{align}
As outlined in Section 4 of \cite{alasseur2020extended}, the objective of the representative agent is to minimize the functional $J(v)$, which comprises two terms. The first term involves the expected value of the integral of a running cost function over a finite horizon $[0,T]$, which includes the current storage cost: $\frac{A_2}{2}S(t)^2 + A_1S(t) + \frac{C}{2}v(t)^2$, the demand charge cost: $\frac{K_0}{2}Q_0(t)^2+\frac{K_1}{2}\left|Q_1(t)-v(t)\right|^2$, and the volumetric charge cost: $p(\mathbb{E}^{\mathcal{B}^t}\left[v(t)-Q_0(t)-Q_1(t)\right])(v(t)-Q_0(t)-Q_1(t))$ with $p$ the pricing rule. The second term is the terminal storage cost: $\frac{B_2}{2}S(T)^2-B_1S(T)$. More precisely, the functional $J(v)$ is defined as follows:
\begin{align}
\nonumber J(v):=&\mathbb{E}\bigg[\int_0^T \frac{A_2}{2}S(t)^2 + A_1S(t) + \frac{C}{2}v(t)^2 + \frac{K_0}{2}Q_0(t)^2+\frac{K_1}{2}\left|Q_1(t)-v(t)\right|^2dt\bigg] \\
\nonumber &+\mathbb{E}\bigg[\int_0^T p(\mathbb{E}^{\mathcal{B}^t}\left[v(t)-Q_0(t)-Q_1(t)\right])(v(t)-Q_0(t)-Q_1(t)) dt\bigg] \\
\label{lqmf.comp.jota.costfunction} &+\mathbb{E}\left[\frac{B_2}{2}S(T)^2-B_1S(T)\right],
\end{align}
where $A_1$, $A_1$, $C$, $K_0$, $K_1$, $B_1$, and $B_2$ are constants, as specified in \cite{alasseur2020extended} and $\mathbb{E}^{\mathcal{B}^t}$ is the conditional expectation given the $\sigma$-field $\mathcal{B}^t = \sigma(b(s):0\leq s\leq t)$, i.e. the historical information generated by the common noise. In this model setting, \cite{alasseur2020extended} claimed in Proposition 3.3 that if $\hat{v}$ is an optimal control for the EMFTC problem with the pricing rule $p$, then $\hat{v}$ is the mean field equilibrium for the EMFG with the pricing rule $p^{MFG}(x):=p(x)+xp'(x)$. They also provided the explicit solution to the EMFTC problem in the linear quadratic case, in other words, the pricing rule was supposed to be:

\begin{align}
	\label{lqmf.comp.price}
p(x):=p_0 x +p_1, \quad \text{for some constants} \quad p_1 > 0.
\end{align}

For our present work, our methodology and results can be easily extended to scenarios involving common noise, the research problem of \cite{alasseur2020extended} can be immediately resolved using our approach which we shall illustrate more in details in Section \ref{lqmf.comp_section.grid}. It is important to emphasize that there are several fundamental differences between our linear quadratic model and that of \cite{alasseur2020extended}. Firstly, the state processes $Q_0$ and $Q_1$ specified in \cite{alasseur2020extended} were not influenced by the control variable $v$, and the control-dependent state $S$ does not contain a noise term, while our present study allows their presence. As shown in \eqref{lqmf.mfg.state_process}, we consider the most comprehensive state process under the linear setting. Secondly, in contrast to the state process \eqref{lqmf.comp.jota.state} as proposed by \cite{alasseur2020extended}, our state process \eqref{lqmf.mfg.state_process} incorporates both the mean-field terms of states and controls. Thirdly, the objective functional \eqref{lqmf.comp.jota.costfunction} as considered in \cite{alasseur2020extended} did not contain the mean-field terms of states, whereas our objective functional \eqref{lqmf.mfg.objective} includes both the mean-field terms of states and controls. Last but not least, the interactions of the state, control, mean field terms of states and controls in our objective functional are completely allowed here in our model, especially, our objective functional is entirely non-seperable including the cross term of all these for terms.


\subsection{Comparison with Dynamic Programming Approach} \label{lqmf.section.master_equation}

When addressing the dynamic grid optimization problem mentioned in the previous section, we also adopted a new master equation approach in \cite{bensoussan2022optimization} to cope with an extended problem of \cite{alasseur2020extended}. We demonstrated the application of this master equation method by studying the same problem as presented in the paper \cite{alasseur2020extended}. In the setting of \cite{bensoussan2022optimization}, the state processes and objective functionals of the problem are still given by \eqref{lqmf.comp.jota.state} and \eqref{lqmf.comp.jota.costfunction}, respectively, with some differences in symbols, for instance: $a=A_2$, $l=A_1$, $c=C$, $h_0=B_2$, and $h_1=-B_1$. More formally, the objective functional of \cite{bensoussan2022optimization} is given by
\begin{align*}
	J(v):=&\mathbb{E}\bigg[\int_0^T \frac{a}{2}S(t)^2 + lS(t) + \frac{c}{2}v(t)^2 + \frac{K_0}{2}Q_0(t)^2+\frac{K_1}{2}\left|Q_1(t)-v(t)\right|^2dt\bigg] \\
	&+\mathbb{E}\bigg[\int_0^T p(\mathbb{E}^{\mathcal{B}^t}\left[v(t)-Q_0(t)-Q_1(t)\right])(v(t)-Q_0(t)-Q_1(t)) dt\bigg] \\
	&+\mathbb{E}\left[\frac{h_0}{2}S(T)^2+h_1S(T)\right],
\end{align*}
and the state processes are given by
\begin{align*}
	\left\{\begin{array}{ccl}
		dQ_0(t) &=& b_0(Q_0,t)dt + \beta_0db(t), \quad Q_0(0)=Q_{00}; \\
		dQ_1(t) &=& b_1(Q_1,t)dt + \sigma(Q_1)dw(t) + \beta_1db(t), \quad Q_1(0)=Q_{10}; \\
		dS(t) &=& v(t)dt, \quad S(0)=0,
	\end{array}\right.
\end{align*}
satisfying the assumptions \eqref{lqmf.comp.mean_reverting} and \eqref{lqmf.comp.price}. In Section \ref{lqmf.comp_section.grid} we shall use our approach to address the dynamic power grid problem proposed by \cite{alasseur2020extended}, and to elaborate further on the differences between our present method and those of \cite{alasseur2020extended} and \cite{bensoussan2022optimization}.

\subsection{Solving the Dynamic Grid Optimization Problem via Maximum Principle} \label{lqmf.comp_section.grid}

We define the coefficients of the dynamic grid optimization problem as follows:
\begin{align}
\label{lqmf.comp.grid.coefficients.redefinition}
\begin{array}{lllll}
\mathsf{A} := \left(\begin{array}{ccc}
-\alpha_0 & 0 & 0\\
0 & -\alpha_1 & 0 \\
0 & 0 & 0
\end{array}\right), \quad
&\mathsf{B} := \left(\begin{array}{c}
0 \\
0 \\
1
\end{array}\right), \quad
&\mathsf{C} := \left(\begin{array}{c}
\alpha_0\gamma_0 \\
\alpha_1\gamma_1 \\
0
\end{array}\right), \quad
&\mathsf{D} := \left(\begin{array}{c}
0 \\
1 \\
0
\end{array}\right), \quad
&\mathsf{E} := \left(\begin{array}{c}
\beta_0 \\
\beta_1 \\
0
\end{array}\right), \\
\mathsf{Q} := \left(\begin{array}{ccc}
K_0 & 0 & 0\\
0 & K_1 & 0 \\
0 & 0 & a
\end{array}\right), \quad
&\mathsf{P} := c+K_1, \quad
&\mathsf{N} := \left(\begin{array}{c}
0 \\
-K_1  \\
0 
\end{array}\right), \quad
&\mathsf{R} := \left(\begin{array}{c}
1 \\
1 \\
0
\end{array}\right), \quad
&\bar{\mathsf{Q}} := p_0, \\
\mathsf{L} := \left(\begin{array}{c}
-p_1 \\
-p_1 \\
l
\end{array}\right), \quad
&\mathsf{M} := p_1, \quad
&\mathsf{Q}_T := \left(\begin{array}{ccc}
0 & 0 & 0\\
0 & 0 & 0 \\
0 & 0 & h_0
\end{array}\right), \quad
&\mathsf{L}_T := \left(\begin{array}{c}
0 \\
0 \\
h_1
\end{array}\right).\end{array}
\end{align}
Then the state processes $Q_0(t)$, $Q_1(t)$ and $S(t)$ correspond to the following vector-valued state process
\begin{align}
\label{lqmf.comp.grid.state.define}
x(t) := \left(Q_0(t),Q_1(t),S(t)\right)^T, \quad x_0 := \left(Q_{00},Q_{10},0\right)^T,
\end{align}
which satisfies
\begin{align}
\label{lqmf.comp.grid.state}
dx(t) = \left(\mathsf{A}x(t)+\mathsf{B}v(t)+\mathsf{C}\right) dt + \mathsf{D}\sigma dw(t) + \mathsf{E}db(t), \quad x(0)=x_0.
\end{align}
The objective functional now becomes
\begin{align}
\nonumber J(v):=&\mathbb{E}\bigg[\int_0^T \frac{1}{2}x(t)^{\top}\mathsf{Q}x(t) + \frac{1}{2}\mathsf{P}v(t)^2 + x(t)^{\top}\mathsf{N}v(t) + \mathsf{L}^{\top}x(t) + \mathsf{M}v(t)dt\bigg] \\
\label{lqmf.comp.grid.objective} &+\mathbb{E}\bigg[\int_0^T (\mathsf{R}^{\top}x(t)- v(t))\bar{\mathsf{Q}}(\mathsf{R}^{\top}\mathbb{E}^{\mathcal{B}^t}[x(t)]- \mathbb{E}^{\mathcal{B}^t}[v(t)]) dt\bigg] +\mathbb{E}\left[\frac{1}{2}x(T)^{\top}\mathsf{Q}_Tx(T)+\mathsf{L}_T^{\top}x(T)\right].
\end{align}

\begin{theorem}({\bf Maximum Principle})
The problem is solved by the pair $(\hat{x}(t),\phi(t),q(t),\bar{q}(t),\hat{v}(t))$ which is the solution of the following FBSDEs
\begin{align}
\label{lqmf.comp.grid.grid.fbsde}
\left\{\begin{array}{cll}
d\hat{x}(t) &=& \left(\mathsf{A}\hat{x}(t)+\mathsf{B}\hat{v}(t)+\mathsf{C}\right) dt + \mathsf{D}\sigma dw(t) + \mathsf{E}db(t), \\
-d\phi(t)&=&\left(A^{\top}\phi(t)+\mathsf{Q}\hat{x}(t)+\mathsf{N}\hat{v}(t)+\mathsf{L}+2\mathsf{R}\bar{\mathsf{Q}}(\mathsf{R}^{\top}\mathbb{E}^{\mathcal{B}^t}[\hat{x}(t)]- \mathbb{E}^{\mathcal{B}^t}[\hat{v}(t)])\right)dt \\
&&+q(t)dw(t)+\bar{q}(t)db(t), \\
\hat{x}(0)&=&x_0, \quad \phi(T)=\mathsf{Q}_T\hat{x}(T)+\mathsf{L}_T,
\end{array}\right.
\end{align}
where
\begin{align}
\label{lqmf.comp.grid.grid.optimal_control.with_Ev}
\hat{v}(t) =& -\mathsf{P}^{-1}\left(\mathsf{B}^{\top}\phi(t) + \mathsf{N}^{\top}\hat{x}(t) + \mathsf{M} - 2 \bar{\mathsf{Q}}\mathsf{R}^{\top}\mathbb{E}^{\mathcal{B}^t}[\hat{x}(t)]+2\bar{\mathsf{Q}}\mathbb{E}^{\mathcal{B}^t}[\hat{v}(t)]\right).
\end{align}
\end{theorem}

\begin{proof}
	Consider the derivative
	\begin{align*}
	\widetilde{x}(t) := \lim_{\theta \to 0}\frac{ x^{v+\theta\widetilde{v}}(t) - x^{v}(t)}{\theta}.
	\end{align*}
	It is clear that $\widetilde{x}(t)$ satisfies
	\begin{align*}
	d \widetilde{x}(t)=\left(\mathsf{A}(t) \widetilde{x}(t)+\mathsf{B}(t) \widetilde{v}(t)\right) d t, \quad \widetilde{x}(0)=x_0.
	\end{align*}
	By taking the G{\^a}teaux derivative, we have
	\begin{align*}
	\frac{d J(v+\theta\widetilde{v})}{d\theta}\bigg|_{\theta=0} =&\mathbb{E}\bigg[\int_0^T \widetilde{x}(t)^{\top}\mathsf{Q}x(t) + \mathsf{P}v(t)\widetilde{v}(t) + \widetilde{x}(t)^{\top}\mathsf{N}v(t) + x(t)^{\top}\mathsf{N}\widetilde{v}(t) + \mathsf{L}^{\top}\widetilde{x}(t) + \mathsf{M}\widetilde{v}(t)dt\bigg] \\
	&+\mathbb{E}\bigg[\int_0^T (\mathsf{R}^{\top}\widetilde{x}(t)- \widetilde{v}(t))\bar{\mathsf{Q}}(\mathsf{R}^{\top}\mathbb{E}^{\mathcal{B}^t}[x(t)]- \mathbb{E}^{\mathcal{B}^t}[v(t)]) + (\mathsf{R}^{\top}x(t)- v(t))\bar{\mathsf{Q}}(\mathsf{R^{\top}}\mathbb{E}^{\mathcal{B}^t}[\widetilde{x}(t)]- \mathbb{E}^{\mathcal{B}^t}[\widetilde{v}(t)]) dt\bigg] \\
	&+\mathbb{E}\left[\widetilde{x}(T)^{\top}\mathsf{Q}_Tx(T)+\mathsf{L}_T^{\top}\widetilde{x}(T)\right].
	\end{align*}
	Note that
	\begin{align*}
	&\mathbb{E}\left[(\mathsf{R}^{\top}x(t)- v(t))\bar{\mathsf{Q}}(\mathsf{R^{\top}}\mathbb{E}^{\mathcal{B}^t}[\widetilde{x}(t)]- \mathbb{E}^{\mathcal{B}^t}[\widetilde{v}(t)])\right] =\mathbb{E}\left[\mathbb{E}^{\mathcal{B}^t}\left[(\mathsf{R}^{\top}x(t)- v(t))\bar{\mathsf{Q}}(\mathsf{R^{\top}}\mathbb{E}^{\mathcal{B}^t}[\widetilde{x}(t)]- \mathbb{E}^{\mathcal{B}^t}[\widetilde{v}(t)])\right]\right] \\
	&=\mathbb{E}\left[\mathbb{E}^{\mathcal{B}^t}\left[(\mathsf{R}^{\top}\mathbb{E}^{\mathcal{B}^t}[x(t)]- \mathbb{E}^{\mathcal{B}^t}[v(t)])\bar{\mathsf{Q}}(\mathsf{R^{\top}}\widetilde{x}(t)- \widetilde{v}(t))\right]\right] \\
	&=\mathbb{E}\left[(\mathsf{R}^{\top}\mathbb{E}^{\mathcal{B}^t}[x(t)]- \mathbb{E}^{\mathcal{B}^t}[v(t)])\bar{\mathsf{Q}}(\mathsf{R^{\top}}\widetilde{x}(t)- \widetilde{v}(t))\right],
	\end{align*}
	and by the convexity of objective functional, it admits the minimum at the point $(\hat{x}(t),\hat{v}(t))$ satisfying the Euler's condition:
	\begin{align}
	\nonumber
	0 \leq &\mathbb{E}\bigg[\int_0^T \widetilde{x}(t)^{\top}\mathsf{Q}\hat{x}(t) + \mathsf{P}\hat{v}(t)\widetilde{v}(t) + \widetilde{x}(t)^{\top}\mathsf{N}\hat{v}(t) + \hat{x}(t)^{\top}\mathsf{N}\widetilde{v}(t) + \mathsf{L}^{\top}\widetilde{x}(t) + \mathsf{M}\widetilde{v}(t)dt\bigg] \\
	\label{lqmf.comp.grid.grid.Euler_condition}
	&+\mathbb{E}\bigg[\int_0^T 2(\mathsf{R}^{\top}\widetilde{x}(t)- \widetilde{v}(t))\bar{\mathsf{Q}}(\mathsf{R}^{\top}\mathbb{E}^{\mathcal{B}^t}[\hat{x}(t)]- \mathbb{E}^{\mathcal{B}^t}[\hat{v}(t)]) dt\bigg] +\mathbb{E}\left[\widetilde{x}(T)^{\top}\mathsf{Q}_T\hat{x}(T)+\mathsf{L}_T^{\top}\widetilde{x}(T)\right],
	\end{align}
	for any $\widetilde{v}(t) \in L_{\mathcal{G}_{\cdot}}^{2}\left(0, T ; \mathbb{R}^m\right)$. By using \eqref{lqmf.comp.grid.grid.fbsde} and integration-by-parts, we have
	\begin{align}
	\nonumber 
	&\mathbb{E}\left[\widetilde{x}(T)^{\top}\mathsf{Q}_T\hat{x}(T)+\mathsf{L}_T^{\top}\widetilde{x}(T)\right] = \mathbb{E}\left[\int_0^T \phi(t)^{\top}d\widetilde{x}(t)+\int_0^T \widetilde{x}(t)^{\top}d\phi(t)\right] \\
	\label{lqmf.comp.grid.grid.dual}
	&= \mathbb{E}\bigg[\int_0^T \phi(t)^{\top}\left(\mathsf{A} \widetilde{x}(t)+\mathsf{B} \widetilde{v}(t)\right) d t-\int_0^T \widetilde{x}(t)^{\top}A^{\top}\phi(t)+\widetilde{x}(t)^{\top}\mathsf{Q}\hat{x}(t)+\widetilde{x}(t)^{\top}\mathsf{N}\hat{v}(t)+\widetilde{x}(t)^{\top}\mathsf{L}+2\widetilde{x}(t)^{\top}\mathsf{R}\bar{\mathsf{Q}}(\mathsf{R}^{\top}\mathbb{E}^{\mathcal{B}^t}[\hat{x}(t)]- \mathbb{E}^{\mathcal{B}^t}[\hat{v}(t)])dt\bigg].
	\end{align}
	Combining \eqref{lqmf.comp.grid.grid.Euler_condition} and \eqref{lqmf.comp.grid.grid.dual}, we obtain for any $\widetilde{v}(t) \in L_{\mathcal{G}_{\cdot}}^{2}\left(0, T ; \mathbb{R}^m\right)$:
	\begin{align*}
	0 \leq &\mathbb{E}\bigg[\int_0^T  \phi(t)^{\top}\mathsf{B} \widetilde{v}(t) + \mathsf{P}\hat{v}(t)\widetilde{v}(t) + \hat{x}(t)^{\top}\mathsf{N}\widetilde{v}(t) + \mathsf{M}\widetilde{v}(t)dt\bigg] -\mathbb{E}\bigg[\int_0^T 2 \widetilde{v}(t)\bar{\mathsf{Q}}(\mathsf{R}^{\top}\mathbb{E}^{\mathcal{B}^t}[\hat{x}(t)]- \mathbb{E}^{\mathcal{B}^t}[\hat{v}(t)]) dt\bigg],
	\end{align*}
	and therefore,
	\begin{align*}
	0 = & \mathsf{B}^{\top}\phi(t)  + \mathsf{P}\hat{v}(t) + \mathsf{N}^{\top}\hat{x}(t) + \mathsf{M} - 2 \bar{\mathsf{Q}}\mathsf{R}^{\top}\mathbb{E}^{\mathcal{B}^t}[\hat{x}(t)]+2\bar{\mathsf{Q}}\mathbb{E}^{\mathcal{B}^t}[\hat{v}(t)].
	\end{align*}
\end{proof}

By substituting the coefficients \eqref{lqmf.comp.grid.coefficients.redefinition} back into \eqref{lqmf.comp.grid.grid.optimal_control.with_Ev}, we obtain
\begin{align}
	\label{lqmf.comp.grid.optimal_control.hat}
	\hat{v}(t) =& \frac{K_1Q_1(t)-\phi_3(t)  - \zeta\left(\mathbb{E}^{\mathcal{B}^t}[\hat{v}(t)]-\mathbb{E}^{\mathcal{B}^t}[Q_0(t)]-\mathbb{E}^{\mathcal{B}^t}[Q_1(t)]\right)}{c+K_1},
\end{align}
where $\zeta(x) := 2p_0x+p_1$ and $\phi_3(t)$ is the third component of the adjoint process $\phi(t)$. This corresponds to the result of Proposition 1 in \cite{bensoussan2022optimization}. Taking the conditional expectation of both sides of \eqref{lqmf.comp.grid.grid.optimal_control.with_Ev}, we further obtain:
\begin{align*}
\mathbb{E}^{\mathcal{B}^t}[\hat{v}(t)] =& -\frac{\mathbb{E}^{\mathcal{B}^t}[\phi_3(t)]-2p_0\mathbb{E}^{\mathcal{B}^t}[Q_0(t)]-(2p_0+K_1)\mathbb{E}^{\mathcal{B}^t}[Q_1(t)] +p_1}{c+K_1+2p_0}.
\end{align*}
which can be substituted back into \eqref{lqmf.comp.grid.grid.optimal_control.with_Ev} to deduce that:
\begin{align}
\label{lqmf.comp.grid.grid.optimal_control.without_Ev}
\hat{v}(t) =&
\frac{\mathbb{E}^{\mathcal{B}^t}[\phi_3(t)]-\phi_3(t) - K_1\left(\mathbb{E}^{\mathcal{B}^t}[Q_1(t)]-Q_1(t)\right)}{c+K_1} -\frac{\mathbb{E}^{\mathcal{B}^t}[\phi_3(t)]-2p_0\mathbb{E}^{\mathcal{B}^t}[Q_0(t)]-(2p_0+K_1)\mathbb{E}^{\mathcal{B}^t}[Q_1(t)] +p_1}{c+K_1+2p_0}.
\end{align}
We can replace \eqref{lqmf.comp.grid.optimal_control.hat} by \eqref{lqmf.comp.grid.grid.optimal_control.without_Ev} in the system of mean-field FBSDEs, by then the new system depends on $\mathbb{E}^{\mathcal{B}^t}[\phi(t)]$ rather than on $\mathbb{E}^{\mathcal{B}^t}[\hat{v}(t)]$.

Since the optimal control only depends on the third component of the adjoint process while $\phi_3(t)$ is separable, we can substitute \eqref{lqmf.comp.grid.coefficients.redefinition} back and consider the following FBSDE system:
\begin{align}
\label{lqmf.comp.grid.fbsde.adjoint_relation.hat}
\left\{\begin{array}{cll}
dQ_0(t) &=& -\alpha_0(Q_0(t)-\gamma_0(t))dt + \beta_0db(t), \quad Q_0(0)=Q_{00};\\
dQ_1(t) &=& -\alpha_1(Q_1(t)-\gamma_1(t))dt + \sigma dw(t) + \beta_1db(t), \quad Q_1(0)=Q_{10};\\
d\hat{S}(t) &=& \hat{v}(t)dt, \quad \hat{S}(0)=0;\\
-d\phi_3(t)&=&\left(a\hat{S}(t)+l\right)dt +q_3(t)dw(t)+\bar{q}_3(t)db(t), \quad \phi_3(T)=h_0\hat{S}(T)+h_1,
\end{array}\right.
\end{align}
where $\hat{v}(t)$ is still given by \eqref{lqmf.comp.grid.optimal_control.hat}. To compare with Section 5.1.1 of \cite{bensoussan2022optimization}, we take a similar {\em ansatz} as that of $\lambda(S,Q,T)$ in \cite{bensoussan2022optimization} and suppose:
\begin{align*}
\phi_3(t):=\lambda_0(t)\hat{S}_t+\lambda_1(t),
\end{align*}
where $\lambda_0(t)$ and $\lambda_1(t)$ are indetermined coefficients. We find that $\lambda_0(t)$ and $\lambda_1(t)$ satisfy:
\begin{align}
\left\{\begin{array}{rl}
	\label{lqmf.comp.grid.bsde.lambda_0_and_lambda_1}
	-d\lambda_0(t)=&\left(a-\frac{\lambda_0(t)^2}{c+K_1}\right)dt; \\
	-d\lambda_1(t)=&\left(-\frac{\lambda_0(t)}{c+K_1}\lambda_1(t) + \frac{\lambda_0(t)\left(K_1Q_1(t)-\zeta\left(\mathbb{E}^{\mathcal{B}^t}[\hat{v}(t)]-\mathbb{E}^{\mathcal{B}^t}[Q_0(t)]-\mathbb{E}^{\mathcal{B}^t}[Q_1(t)]\right)\right)}{c+K_1} + l \right)dt +q_3(t)dw(t)+\bar{q}_3(t)db(t),
\end{array}\right.
\end{align}
with $\lambda_0(T) = h_0$ and $\lambda_1(T) = h_1$.

Now, our $\lambda_0(t)$ and $\lambda_1(t)$ in \eqref{lqmf.comp.grid.bsde.lambda_0_and_lambda_1} correspond to their $\lambda_0(t)$ and $\lambda_1(t,Q_1(t))+\nu(t)$, respectively (as defined in their Eq. (32) in \cite{bensoussan2022optimization}). We note that the FBSDE system \eqref{lqmf.comp.grid.optimal_control.hat}-\eqref{lqmf.comp.grid.fbsde.adjoint_relation.hat} is not completely solved due to the presence of $\mathbb{E}^{\mathcal{B}^t}[\hat{v}(t)]$ in the BSDEs \eqref{lqmf.comp.grid.bsde.lambda_0_and_lambda_1}. Together with \eqref{lqmf.comp.grid.optimal_control.hat}, this gives rise to another mean-field coupled system, and we shall elaborate on how to obtain the value of $\mathbb{E}^{\mathcal{B}^t}[\hat{v}(t)]$ as follows.

\paragraph{Explicit Solution}

To obtain the explicit expression of $\mathbb{E}^{\mathcal{B}^t}[\hat{v}(t)]$ and have the explicit solution , we first consider the expected value form of the FBSDEs \eqref{lqmf.comp.grid.fbsde.adjoint_relation.hat}-\eqref{lqmf.comp.grid.optimal_control.hat}. Write $\bar{Q}_0(t) := \mathbb{E}^{\mathcal{B}^t}[Q_0(t)]$, $\bar{Q}_1(t) := \mathbb{E}^{\mathcal{B}^t}[Q_1(t)]$, $\bar{S}(t) := \mathbb{E}^{\mathcal{B}^t}[\hat{S}(t)]$, $\bar{\phi}_3(t) := \mathbb{E}^{\mathcal{B}^t}[\phi_3(t)]$ and $\bar{v}(t) := \mathbb{E}^{\mathcal{B}^t}[\hat{v}(t)]$, then
\begin{align}
\label{lqmf.comp.grid.fbsde.adjoint_relation.bar}
\left\{\begin{array}{cll}
d\bar{Q}_0(t) &=& -\alpha_0(\bar{Q}_0(t)-\gamma_0(t))dt + \beta_0db(t), \quad \bar{Q}_0(0)=Q_{00};\\
d\bar{Q}_1(t) &=& -\alpha_1(\bar{Q}_1(t)-\gamma_1(t))dt + \beta_1db(t), \quad \bar{Q}_1(0)=Q_{10};\\
d\bar{S}(t) &=& \bar{v}(t)dt, \quad \bar{S}(0)=0; \\
-d\bar{\phi_3}(t)&=&\left(a\bar{S}(t)+l\right)dt+\bar{q}_3(t)db(t), \quad \bar{\phi}_3(T)=h_0\bar{S}(T)+h_1,
\end{array}\right.
\end{align}
with
\begin{align}
\label{lqmf.comp.grid.optimal_control.bar}
\bar{v}(t) =& -\frac{\bar{\phi}_3(t)-2p_0\bar{Q}_0(t)-(2p_0+K_1)\bar{Q}_1(t) +p_1}{c+K_1+2p_0}.
\end{align}
It is clear that \eqref{lqmf.comp.grid.fbsde.adjoint_relation.bar}-\eqref{lqmf.comp.grid.optimal_control.bar} are classical FBSDEs without the mean field term. As a standard argument, we construct the {\em ansatz}:
\begin{align}
\bar{\phi}_3(t) = \bar{\lambda}_0(t)\bar{S}(t)+\bar{\Gamma}_0(t)\bar{Q}_0(t) + \bar{\Gamma}_1(t)\bar{Q}_1(t) +\bar{\nu}(t),
\end{align}
then it can be solved explicitly by the following (linear mixed with Riccati) ordinary differential equations: 
\begin{align}
\label{lqmf.comp.grid.ode.bar.1}
\frac{d\bar{\lambda}_0(t)}{dt}+a-\frac{\bar{\lambda}_0(t)^2}{c+K_1+2p_0}=0, \quad \bar{\lambda}_0(T) = h_0,
\end{align}
\begin{align}
\label{lqmf.comp.grid.ode.bar.2}
\frac{d\bar{\Gamma}_0(t)}{dt}-\left(\alpha_0+\frac{\bar{\lambda}_0(t)}{c+K_1+2p_0}\right)\bar{\Gamma}_0(t) + \frac{2\bar{\lambda}_0(t)p_0}{c+K_1+2p_0}=0, \quad \bar{\Gamma}_0(T) = 0,
\end{align}
\begin{align}
\label{lqmf.comp.grid.ode.bar.3}
\frac{d\bar{\Gamma}_1(t)}{dt}-\left(\alpha_1+\frac{\bar{\lambda}_0(t)}{c+K_1+2p_0}\right)\bar{\Gamma}_1(t) + \frac{\bar{\lambda}_0(t)(K_1+2p_0)}{c+K_1+2p_0}=0, \quad \bar{\Gamma}_1(T) = 0,
\end{align}
\begin{align}
\label{lqmf.comp.grid.ode.bar.4}
\frac{d\bar{\nu}(t)}{dt}-\frac{\bar{\lambda}_0(t)}{c+K_1+2p_0}\bar{\nu}(t)-\frac{\bar{\lambda}_0(t)p_1}{c+K_1+2p_0}+\gamma_0\bar{\Gamma}_0(t)+\gamma_1\bar{\Gamma}_1(t)=0, \quad \bar{\nu}(T) = h_1,
\end{align}
and
\begin{align}
\label{lqmf.comp.grid.ode.bar.5}
\bar{q}_3(t)=-\beta_0\bar{\Gamma}_0(t)-\beta_1\bar{\Gamma}_1(t).
\end{align}
To solve the ODE system \eqref{lqmf.comp.grid.ode.bar.1}-\eqref{lqmf.comp.grid.ode.bar.4}, we begin with solving the independent equation \eqref{lqmf.comp.grid.ode.bar.1} first. Once we obtain a solution, we substitute it into \eqref{lqmf.comp.grid.ode.bar.2} and \eqref{lqmf.comp.grid.ode.bar.3} to obtain solutions for $\bar{\Gamma}_0(t)$ and $\bar{\Gamma}_1(t)$, in order. Finally, we substitute the solutions of \eqref{lqmf.comp.grid.ode.bar.1}, \eqref{lqmf.comp.grid.ode.bar.2} and \eqref{lqmf.comp.grid.ode.bar.3} into \eqref{lqmf.comp.grid.ode.bar.4} to obtain the explicit solution for $\bar{\nu}(t)$. At this point, $\bar{v}(t)$ can be obtained by \eqref{lqmf.comp.grid.optimal_control.bar}. 

Apart from directly solving the equations \eqref{lqmf.comp.grid.bsde.lambda_0_and_lambda_1}, we provide a more convenient method to derive $\hat{v}(t)$ as below. We define the centered random variables: $\widetilde{Q}_0(t):=Q_0(t)-\bar{Q}_0(t)$, $\widetilde{Q}_1(t):=Q_1(t)-\bar{Q}_1(t)$, $\widetilde{S}(t):=\hat{S}(t)-\bar{S}(t)$, $\widetilde{\phi}_3(t):=\phi_3(t)-\bar{\phi}_3(t)$ and $\widetilde{v}(t):=\hat{v}(t)-\bar{v}(t)$. Thus by \eqref{lqmf.comp.grid.fbsde.adjoint_relation.hat}-\eqref{lqmf.comp.grid.optimal_control.hat} and \eqref{lqmf.comp.grid.fbsde.adjoint_relation.bar}-\eqref{lqmf.comp.grid.optimal_control.bar}, we obtain the following FBSDEs
\begin{align}
\label{lqmf.comp.grid.fbsde.adjoint_relation.widetilde}
\left\{\begin{array}{cll}
d\widetilde{Q}_0(t) &=& -\alpha_0\widetilde{Q}_0(t)dt, \quad \widetilde{Q}_0(0)=0;\\
d\widetilde{Q}_1(t) &=& -\alpha_1\widetilde{Q}_1(t)dt + \sigma dw(t), \quad \widetilde{Q}_1(0)=0; \\
d\widetilde{S}(t) &=& \widetilde{v}(t)dt, \quad \widetilde{S}(0)=0; \\ 
-d\widetilde{\phi}_3(t)&=&a\widetilde{S}(t)dt +q_3(t)dw(t), \quad \widetilde{\phi}_3(T)=h_0\widetilde{S}(T),
\end{array}\right.
\end{align}
with
\begin{align}
\label{lqmf.comp.grid.optimal_control.widetilde}
\widetilde{v}(t) =& \frac{K_1\widetilde{Q}_1(t)-\widetilde{\phi}_3(t)}{c+K_1}.
\end{align}
This system is also a classical FBSDE system, which can be easily solved by an {\em ansatz}:
\begin{align}
\widetilde{\phi}_3(t) = \lambda_0(t)\widetilde{S}(t)+\Gamma_0(t)\widetilde{Q}_0(t) + \Gamma_1(t)\widetilde{Q}_1(t),
\end{align}
where the coefficients $\lambda_0(t)$, $\Gamma_0(t)$ and $\Gamma_1(t)$ satisfy
\begin{align}
\label{lqmf.comp.grid.ode.widetilde.1}
\frac{d\lambda_0(t)}{dt}+a-\frac{\lambda_0(t)^2}{c+K_1}=0, \quad \lambda_0(T) = h_0,
\end{align}
\begin{align}
\label{lqmf.comp.grid.ode.widetilde.2}
\frac{d\Gamma_0(t)}{dt}-\left(\alpha_0+\frac{\lambda_0(t)}{c+K_1}\right)\Gamma_0(t)=0, \quad \Gamma_0(T) = 0,
\end{align}
\begin{align}
\label{lqmf.comp.grid.ode.widetilde.3}
\frac{d\Gamma_1(t)}{dt}-\left(\alpha_1+\frac{\lambda_0(t)}{c+K_1}\right)\Gamma_1(t) + \frac{\lambda_0(t)K_1}{c+K_1}=0, \quad \Gamma_1(T) = 0,
\end{align}
moreover, $q_3(t)=-\sigma \Gamma_1(t)$. Note that the ODE system \eqref{lqmf.comp.grid.ode.widetilde.1}-\eqref{lqmf.comp.grid.ode.widetilde.3} is decoupled, it can be easily solved in an explicit form. Then by combining the explicit solutions of the systems \eqref{lqmf.comp.grid.fbsde.adjoint_relation.bar}-\eqref{lqmf.comp.grid.optimal_control.bar} and \eqref{lqmf.comp.grid.fbsde.adjoint_relation.widetilde}-\eqref{lqmf.comp.grid.optimal_control.widetilde}, the solutions of \eqref{lqmf.comp.grid.fbsde.adjoint_relation.hat}-\eqref{lqmf.comp.grid.optimal_control.hat} can be expressed by
$\hat{S}(t)=\widetilde{S}(t)+\bar{S}(t)$, $\phi_3(t)=\widetilde{\phi}_3(t)+\bar{\phi}_3(t)$ and $\hat{v}(t)=\widetilde{v}(t)+\bar{v}(t)$. Comparing to the master equation approach shown in \cite{bensoussan2022optimization}, our proposed approach is more mechanical and tractable.

\end{appendices}


\end{document}